\documentclass[11pt,reqno]{amsart} 
\usepackage{amssymb,amsthm,amsmath,amsfonts,latexsym,tikz,hyperref,enumerate,mathtools}

\usepackage{cleveref}
\usepackage{cite}

\usepackage[hmargin=.8in,vmargin=1in]{geometry}
\usepackage{tikz}
\tikzstyle{vertex}=[circle, draw, inner sep=0pt, minimum size=4pt]
\definecolor{darkgreen}{cmyk}{.9,0,.9,.20}

\newcommand{\ben}{\begin{enumerate}}
\newcommand{\een}{\end{enumerate}}
\newcommand{\ble}{\begin{lem}}
\newcommand{\ele}{\end{lem}}
\newcommand{\bth}{\begin{thm}}
\renewcommand{\eth}{\end{thm}}
\newcommand{\bpr}{\begin{prop}}
\newcommand{\epr}{\end{prop}}
\newcommand{\bco}{\begin{cor}}
\newcommand{\eco}{\end{cor}}
\newcommand{\bcon}{\begin{conj}}
\newcommand{\econ}{\end{conj}}
\newcommand{\bde}{\begin{defn}}
\newcommand{\ede}{\end{defn}}
\newcommand{\bex}{\begin{exa}}
\newcommand{\eex}{\end{exa}}
\newcommand{\barr}{\begin{array}}
\newcommand{\earr}{\end{array}}
\newcommand{\btab}{\begin{tabular}}
\newcommand{\etab}{\end{tabular}}
\newcommand{\beq}{\begin{equation}}
\newcommand{\eeq}{\end{equation}}
\newcommand{\bea}{\begin{eqnarray*}}
\newcommand{\eea}{\end{eqnarray*}}
\newcommand{\bal}{\begin{align*}}
\newcommand{\bce}{\begin{center}}
\newcommand{\ece}{\end{center}}
\newcommand{\bpi}{\begin{picture}}
\newcommand{\epi}{\end{picture}}
\newcommand{\bpp}{\begin{picture}}
\newcommand{\epp}{\end{picture}}
\newcommand{\bfi}{\begin{figure} \begin{center}}
\newcommand{\efi}{\end{center} \end{figure}}
\newcommand{\bprf}{\begin{proof}}
\newcommand{\eprf}{\end{proof}\medskip}

\newcommand{\bbR}{{\mathbb R}}

\newcommand{\cB}{\mathcal{B}}

\newcommand{\cC}{\mathcal{ C}}

\newcommand{\cF}{\mathcal{F} }

\newcommand{\cH}{\mathcal{ H}}

\newcommand{\cP}{{\mathcal{P}}}

\newcommand{\cQ}{\mathcal{ Q}}
\newcommand{\cR}{\mathcal{ R}}

\newcommand{\cS}{\mathcal{S}}
\newcommand{\cT}{\mathcal{ T}}

\newcommand{\cY}{\mathcal{ Y}}

\newcommand{\vv}{\mathbf{v}}
\newcommand{\ww}{\mathbf{w}}

\newcommand{\RR}{\mathbb{R}}

\DeclareMathOperator{\des}{des}

\DeclareMathOperator{\Pyr}{Pyr}
\DeclareMathOperator{\Conv}{ConvexHull}

\DeclareMathOperator{\simp}{\Delta}
\DeclareMathOperator{\vol}{Vol}
\DeclareMathOperator{\nvol}{nVol}
\DeclareMathOperator{\relvol}{relVol}
\DeclareMathOperator{\ehr}{Ehr}
\newcommand{\e}{\textbf{e}}

\title{Partial Permutohedra}
\author[Behrend]{Roger E. Behrend}
\address[R.~E.~Behrend]{School of Mathematics, Cardiff University, Cardiff, UK}
\email{\textcolor{blue}{\href{mailto:}{behrendr@cardiff.ac.uk}}}

\author[Castillo]{Federico Castillo}
\address[F. Castillo]{Departamento de Matem\'aticas, Pontificia Universidad Cat\'olica de Chile, Santiago, Chile}
\email{\textcolor{blue}{\href{mailto:}{federico.castillo@mat.uc.cl}}}

\author[Chavez]{Anastasia Chavez}
\address[A. Chavez]{Department of Mathematics and Computer Science, Saint Mary's College of California, Moraga, CA, USA}
\email{\textcolor{blue}{\href{mailto:}{amc59@stmarys-ca.edu}}}

\author[Diaz-Lopez]{Alexander Diaz-Lopez}
\address[A. Diaz-Lopez]{Department of Mathematics and Statistics, Villanova University, Villanova, PA, USA}
\email{\textcolor{blue}{\href{mailto:}{alexander.diaz-lopez@villanova.edu}}}

\author[Escobar]{Laura Escobar}
\address[L. Escobar]{Department of Mathematics, University of California Santa Cruz, Santa Cruz, CA, USA}
\email{\textcolor{blue}{\href{mailto:}{lauraescobar@ucsc.edu}}}

\author[Harris]{Pamela E. Harris}
\address[P.~E. Harris]{Department of Mathematical Sciences, University of Wisconsin Milwaukee, Milwaukee, WI, USA}
\email{\textcolor{blue}{\href{mailto:peharris@uwm.edu}{peharris@uwm.edu}}}

\author[Insko]{Erik Insko}
\address[E.~Insko]{Department of Mathematics, Central College, Pella, IA, USA}
\email{\textcolor{blue}{\href{mailto:inskoe@central.edu}{inskoe@central.edu}}}
 
\usepackage{latexsym,amssymb,amsfonts,amsmath,graphicx,fullpage,url,color,tikz}
\usepackage[utf8]{inputenc}
\usepackage[vcentermath,enableskew]{youngtab}
\usepackage[all]{xy}
\usepackage{verbatim}
\usepackage{ytableau}
\usepackage{makecell}
\usepackage[export]{adjustbox}
\usepackage{multicol}

\newtheorem{theorem}{Theorem}[section]
\newtheorem{proposition}[theorem]{Proposition}
\newtheorem{lemma}[theorem]{Lemma}
\newtheorem{corollary}[theorem]{Corollary}

\newtheorem{conjecture}[theorem]{Conjecture}
\newtheorem{oproblem}[theorem]{Open~Problem}

\theoremstyle{definition}
\newtheorem{remark}[theorem]{Remark}
\newtheorem{definition}[theorem]{Definition}
\newtheorem{example}[theorem]{Example}

\theoremstyle{remark}
 
\numberwithin{equation}{section}
\usepackage{multicol}

 
\begin{document}
\begin{abstract}
Partial permutohedra are lattice polytopes which were recently introduced  and studied by Heuer and Striker.
For positive integers $m$ and $n$, the partial permutohedron~$\cP(m,n)$ is the convex hull of all vectors in $\{0,1,\ldots,n\}^m$ whose nonzero entries are distinct. 
We study the face lattice, volume and Ehrhart polynomial of $\cP(m,n)$,  and our methods and results include the following. 
For any $m$ and $n$, we obtain a  bijection between the nonempty faces of $\cP(m,n)$ and certain chains of subsets of $\{1,\dots,m\}$,  thereby confirming a conjecture of Heuer and Striker, and we then use this characterization of faces to obtain a closed expression for the $h$-polynomial of~$\cP(m,n)$.  
For any $m$ and $n$ with $n\ge m-1$, we  use a pyramidal subdivision of $\mathcal{P}(m,n)$ to establish a recursive formula for the normalized volume of $\mathcal{P}(m,n)$, from which we then obtain closed expressions for this volume. 
We also use a sculpting  process (in which $\cP(m,n)$ is reached by successively removing certain pieces from a simplex or hypercube) to obtain closed expressions for the  Ehrhart polynomial of $\cP(m,n)$ with arbitrary $m$ and fixed $n\le 3$, the normalized volume of~$\cP(m,4)$ with arbitrary $m$, and the Ehrhart polynomial of $\cP(m,n)$ with fixed $m\le4$ and arbitrary $n\ge m-1$.
\end{abstract} 

\maketitle

\section{Introduction}\label{section:intro}
Computing the volume of a polytope is hard, even when the complete face structure is known~\cite{khachiyan}.
In fact, few exact volume formulae have been discovered in much generality.
Stanley gave a notable volume formula for the regular permutohedron $\Pi(1,2,\dots,m)$, specifically that its relative volume is $m^{m-2}$~\cite[Example~3.1]{stanley}.
More generally, Postnikov~\cite{postnikov} studied the permutohedron $\Pi(z_1,\ldots,z_m)$ (i.e., the convex hull of all vectors obtained by permuting the entries of an arbitrary vector $(z_1,\ldots,z_m)$
in~$\RR^m$) as well as a class of generalized
permutohedra, and obtained
three distinct formulae for the relative volume of $\Pi(z_1,\dots,z_m)$~\cite[Theorems~3.1,~5.1 and~17.1]{postnikov}, each one subtle in its own way.

In this paper\footnote{An extended 12-page abstract of this work has been published in the refereed proceedings of the~35th International Conference on Formal Power Series and Algebraic Combinatorics (University of California, Davis, July 17-21, 2023)~\cite{PP-FPSAC}.}, we study a related family of polytopes  called \emph{partial permutohedra}, which were introduced recently by Heuer and Striker~\cite{HS}.  For positive integers~$m$ and~$n$, the partial permutohedron~$\cP(m,n)$
is the convex hull of all vectors in $\{0,1,\ldots,n\}^m$ whose nonzero entries are distinct.  
It immediately follows that~$\cP(m,n)$ is a lattice polytope.

Partial permutohedra have connections to several other previously-studied polytopes, including the following. 
In Section~\ref{sec:gp}, we show that~$\cP(m,n)$ with any~$m$ and~$n$ is,
after being lifted from~$\RR^m$ to~$\RR^{m+1}$, a case of a generalized permutohedron of~\cite{postnikov}.  
In Corollary~\ref{cor-anti-block}, we show that partial permutohedra are anti-blocking versions of certain permutohedra: specifically, $\cP(m,n)$
is an anti-blocking version of $\Pi(0,\ldots,0,1,2,\dots,n)$ 
(with $m-n$ zeros) for $n\le m-2$, or $\Pi(n-m+1,n-m+2,\dots,n)$ 
for $n\ge m-1$.  
Additionally, $\cP(m,n)$ with any $n\ge m$ is combinatorially equivalent to the $m$-stellohedron (see Remark~\ref{rem-stell}), where the connection to the $m$-stellohedron for $m=n$ was noted in~\cite{HS}.
This polytope was originally defined in ~\cite[Section 10.4]{PRW}, and has been used recently in connection with matroid theory~\cite{stellahedron}.
In Remarks~\ref{rem-winpolytope} and~\ref{rem-parkingfunctionpolytope}, we note
that~$\cP(m,m-1)$ with any $m\ge2$ is the polytope  
of win vectors of the complete graph~$K_m$~\cite{bartels1997polytope},
and (after translation by $(1,\ldots,1)$) the polytope of parking functions of length~$m$~\cite{AW,AMM_solution,AMM_problem}.
Furthermore, as noted in Remark~\ref{rem-parkingfunctionpolytopegen}, it has
recently been shown in~\cite{HanadaLentferVindasMelendez} that~$\cP(m,n)$ with any $n\ge m-1$ is (again after translation by $(1,\ldots,1)$) the polytope
of certain generalized parking functions of length~$m$.

In Section~\ref{section:prelim}, we provide relevant background information on polytopes, and review
the results of Heuer and Striker~\cite{HS} on partial permutohedra.

In Section~\ref{sec-faces}, we expand on the work of Heuer and Striker~\cite{HS} by obtaining,
in Theorem~\ref{thm-faces-chains}, a bijection between the nonempty faces of~$\cP(m,n)$ and certain chains of subsets of $\{1,\dots,m\}$, 
for any~$m$ and~$n$, thus proving Conjecture~5.25 of~\cite{HS}.
An alternative proof of the conjecture was recently obtained independently 
by Black and Sanyal~\cite[Theorem~7.5]{black2022flag}.
We then use this characterization of the faces of~$\cP(m,n)$ to obtain, 
in Theorem~\ref{thm-h-vector-bijective},
a closed expression for the $h$-polynomial of~$\cP(m,n)$ with any $m$ and $n$ in terms of Eulerian polynomials.

In Section~\ref{sec:volppm}, we consider the volume of~$\cP(m,n)$ for $n\ge m-1$.
In Theorem~\ref{thm:recursion for volume of P(m,n)}, we 
use a technique, in which~$\cP(m,n)$ is subdivided into certain pyramids, to establish a recursive formula for the normalized volume of~$\cP(m,n)$ with $n\geq m-1$.
Using this recursion, we then obtain, in Theorem~\ref{thm:closedformulav(m,n)}, closed formulae for the normalized volume of~$\cP(m,n)$ with $n\geq m-1$. Our proofs of Theorems~\ref{thm:recursion for volume of P(m,n)} and~\ref{thm:closedformulav(m,n)} follow methods similar to those in~\cite[Section~4]{AW} and~\cite[Part~(d)]{AMM_solution} for computations of the volume of the polytope of parking functions of length~$m$ (and hence of~$\cP(m,m-1)$).
Another proof of Theorem~\ref{thm:closedformulav(m,n)} was recently obtained independently 
by Hanada, Lentfer and Vindas-Meléndez~\cite[Corollary~3.29]{HanadaLentferVindasMelendez}. 
We also, in~\eqref{eq-vol-drac-seq} and~\eqref{eq-vol-drac-seq2}, provide
certain expressions for the normalized volume of~$\cP(m,n)$ with $n\geq m-1$, which
are obtained using results for the volumes of generalized permutohedra~\cite[Theorems~9.3 and~10.1]{postnikov}.
The fact that we are able to obtain the closed formulae of Theorem~\ref{thm:closedformulav(m,n)} for the volume of~$\cP(m,n)$ with $n\geq m-1$ is related to the fact that the volume of $\Pi(1,\ldots,m)$
(or $\Pi(0,\ldots,m-1)$) is given by a simple closed formula.
However, as explained in Remark~\ref{rem-arg},
because we do not have a closed formula for the volume of $\Pi(0,\dots,0,1,\ldots,n)$
(with at least two 0's), finding a general formula for the volume
of~$\cP(m,n)$ with $m>n+1$ becomes much more difficult.
Moreover, for $m>n$ the combinatorial type of~$\cP(m,n)$ depends on both~$m$ and~$n$ 
(whereas, as explained in Remark~\ref{rem-isomorphic}, for $n\ge m$ it depends only on $m$), which further suggests that finding a completely general closed formula for the volume in this case is unlikely.

In Sections~\ref{sec:volppn}--\ref{sec:Ehrhart}, we continue to address the problem of computing the volumes of partial permutohedra, and we also study the Ehrhart polynomials of some cases.
One of our main techniques in these sections is based on the idea, exploited by algebraists in the days of yore, of completing the (hyper)cube. 
We start with a lattice polytope for which we know the  volume or Ehrhart polynomial, and then carefully remove pieces until we reach the polytope of interest. 
This idea makes itself apparent after analyzing certain expressions, as given in Example~\ref{ex:formulas}, for the normalized volume of~$\cP(m,n)$ for small fixed $m$ and any $n\ge m-1$.  These expressions are polynomials in~$n$,
with all coefficients negative, except in the leading term which is $m!\,n^m$, the normalized volume of a hypercube in $\RR^m$ of side-length~$n$.  This suggests that we start with a hypercube from which we can sculpt a partial permutohedron, and this is precisely what we do in Section~\ref{sec:Ehrhart}.  This sculpting approach has been used recently to compute the Ehrhart polynomials of matroid polytopes starting from the hypersimplex. See~\cite{ferroni} for sparse paving matroids and~\cite{jeremy} for paving matroids.

In Section~\ref{sec:volppn}, we use a sculpting process, in which~$\cP(m,n)$ is 
sculpted from a $\binom{n+1}{2}$-dilated standard $m$-simplex,
to compute the normalized volume of~$\cP(m,n)$ with arbitrary~$m$ and fixed $n\leq 4$.
In Theorem~\ref{theorem:pm2}, we show that the normalized volume of~$\cP(m,2)$ is $3^m-3$, thereby confirming Conjecture~5.30 of~\cite{HS}, and in Theorems~\ref{theorem:Pm3} and~\ref{theorem:Pm4}, we
give explicit formulae for the normalized volumes of~$\cP(m,3)$ and~$\cP(m,4)$.  Theorems~\ref{theorem:pm2} and~\ref{theorem:Pm3}
also provide explicit expressions for the 
Ehrhart polynomials of~$\cP(m,2)$ and~$\cP(m,3)$ with arbitrary~$m$.
By examining the details of each case with $n\le4$, the reader will appreciate that the steps involved become progressively harder, and that it may be impractical to proceed beyond $n=4$ using these methods.
Nevertheless, in Conjecture~\ref{conj-vmn}, we 
use the formulae obtained for $n\le 4$ to 
conjecture that for arbitrary $n$, the normalized volume of~$\cP(m,n)$ can be expressed in a certain form.

In Section~\ref{sec:Ehrhart}, we return to the case of $n\geq m-1$,
and use a scuplting process, in which~$\cP(m,n)$ is sculpted
from an $m$-cube of side-length $n$, to obtain explicit expressions for the
Ehrhart polynomial of~$\cP(m,n)$ with fixed $m\le 4$ and 
arbitrary $n\geq m-1$.  See, for example, Theorems~\ref{thm:p3n} and~\ref{thm:p4n} for the cases $m=3$ and $m=4$, respectively.  We also, in~\eqref{eq-Ehr-drac-seq}, provide an expression for the Ehrhart polynomial of~$\cP(m,n)$ with $n\geq m-1$, which is obtained using a result for Ehrhart polynomials of generalized permutohedra~\cite[Theorem~11.3]{postnikov}. 
Finally, in Conjecture~\ref{conj-Ehr-expl}, we conjecture a closed formula for the Ehrhart polynomial of~$\cP(m,n)$ with $n\geq m-1$.

\subsection*{Acknowledgements}
The authors thank Spencer Backman, Luis Ferroni, Jessica Striker and Shaun Sullivan for helpful exchanges, and the anonymous referees for valuable suggestions.
We thank the American Institute of Mathematics for research support through a SQuaRE grant.  Behrend was  partially supported by Leverhulme Trust Grant RPG-2019-083.
Castillo was partially supported by FONDECYT Grant 1221133.
Escobar was partially supported by NSF Grant DMS-1855598 and NSF CAREER Grant DMS-2142656.
Harris was supported through a Karen Uhlenbeck EDGE Fellowship.

\section{Background} \label{section:prelim}
We work over the Euclidean space $\mathbb{R}^m$ with basis $\{\e_1,\dots,\e_m\}$, and the dot product $\langle \e_i,\e_j\rangle=\delta_{ij}$, where $\delta_{ij}$ is the Kronecker delta function.
For convenience, we will sometimes write~$\mathbf{e}_0$ for the origin~$\mathbf{0}$.
We will also use the notation $[m]$ for the set $\{1,\ldots,m\}$, and $\mathfrak{S}_m$ for the set of permutations on~$[m]$.
\subsection{Polytopes}\label{subsection:polytopes} 	
A \emph{polytope} is the convex hull of finitely many points in $\mathbb{R}^m$.
Alternatively, a polytope is a bounded solution set of a finite system of linear inequalities.
We say that a linear inequality $\langle \mathbf{a},\mathbf{x}\rangle\geq b$ is valid on a polytope~$\cP$ if every point of~$\cP$ satisfies it.
A valid linear inequality defines a \emph{face} $\cF$ of $\cP$, namely $\cF=\cP\cap\{\mathbf{x}\in\mathbb{R}^m\mid\langle \mathbf{a},\mathbf{x}\rangle= b\}$.
Faces of dimension~$0$,~$1$ or $\dim(\cP)-1$ are called \emph{vertices}, \emph{edges} or \emph{facets}, respectively.

A polytope whose vertices are all integer points is called a \emph{lattice polytope}.
Two important examples of lattice polytopes are the standard $m$-simplex and the regular permutohedron. The standard $m$-simplex 
$\Delta_m=\textrm{ConvexHull}(\{\mathbf
{0},\e_1,\e_2,\ldots,\e_m\})$ has $m+1$ facets: specifically,~$m$ facets induced by the inequalities $x_i\geq 0$ for $i\in[m]$, and an additional facet induced by the inequality $x_1+\dots+x_m\leq 1$. We will sometimes
also use $\Delta_k$ to denote a $k$-simplex $\textrm{ConvexHull}(\{\mathbf{0}\}\cup\{\e_i\mid i\in S\})$ in~$\RR^m$, where~$S$ is a $k$-element subset of $[m]$.
The regular permutohedron, as introduced in Section~\ref{section:intro}, is 
$\Pi(1,2,\ldots, m)=\textrm{ConvexHull}(\{(\sigma(1),\sigma(2),\ldots, \sigma(m))\mid\sigma \in \mathfrak{S}_m\})$. 
Note that $\Pi(1,2,\ldots,m)$ is an $(m-1)$-dimensional polytope in $\mathbb{R}^m$, with every $\mathbf{x}\in\Pi(1,2,\ldots,m)$ satisfying $x_1+x_2+\ldots+x_m=\binom{m+1}{2}$.
	
\subsection{Volumes}\label{subsection:volumes}
There is a unique translation-invariant measure on $\mathbb{R}^m$, up to a scalar.
This scalar is often chosen so that volume formulae are simpler to state, and hence the choice of scalar can vary. The most familiar choice is such that the volume of the hypercube $[0,1]^m$ is~$1$. We call this volume the Euclidean/Lebesgue volume, or simply the volume, and denote it as $\vol$. Geometric objects which use hypercubes as their building blocks usually have simpler volume formulae when expressed using the Euclidean volume. 
 	    
When working with lattice polytopes, we often use the standard $m$-simplex $\Delta_m$ as our building block. The Euclidean volume of $\Delta_m$ is $1/m!$, since the hypercube $[0,1]^m$ can be triangulated into~$m!$ congruent copies of $\Delta_m$.
To simplify volume expressions of full-dimensional polytopes in $\mathbb{R}^m$, we use the normalized volume, denoted $\nvol$, and defined such that the normalized volume of any lattice $m$-simplex with vertices $ \{\mathbf{0},\vv_1,\vv_2,\ldots,\vv_m \} $ is the determinant of the matrix whose rows are $\vv_1,\vv_2,\ldots,\vv_m$.  With this definition, we have
\begin{equation*} \nvol(\Delta_m) =1, \text{ and } \nvol([0,1]^m)=m! \vol([0,1]^m)=m!. \label{eq:nvol} \end{equation*}  
Moreover, the normalized volume of any full-dimensional lattice polytope $\cP$ in $\mathbb{R}^m$ is $\nvol(\cP)=m!\vol(P)$, which is an integer, since we can triangulate such a polytope into lattice simplices, each of which has an
integer normalized volume.
 	    
\subsection{Relative volume of non-full-dimensional polytopes}\label{sec:nonfulldimvolume}
For polytopes in $\mathbb{R}^m$ that are not full-dimensional, such as the $(m-1)$-dimensional
regular permutohedron $\Pi(1,\dots,m)$, the volume needs to be defined carefully. 
 	
Let $\mathcal{P}\subseteq \mathbb{R}^m$ be a $d$-dimensional lattice polytope, and $\mathcal{A}$ be the affine hull of~$\mathcal{P}$. 
If $\textbf{0}\in \mathcal{A}$ (i.e.,~$\mathcal{A}$ is a linear subspace), then let $L$ be the $d$-dimensional lattice $L=\mathcal{A}\cap \mathbb{Z}^m$, and $\{\vv_1,\ldots,\vv_d\}$ be a $\mathbb{Z}$-basis for $L$. We define the \emph{relative} volume  $\relvol$ of~$\mathcal{P}$ as the unique translation invariant measure on $\mathcal{A}$ such that $\relvol \left(\textrm{ConvexHull}(\{\textbf{0},\vv_1,\ldots,\vv_d\}) \right) = 1/d!$.
This definition is independent of the choice of basis $\{\vv_1,\ldots,\vv_d\}$, since $\relvol \left(\textrm{ConvexHull}(\{\textbf{0},\ww_1,\ldots,\ww_d\})\right) = 1/d!$ for any other $\mathbb{Z}$-basis $\{ \ww_1, \dots, \ww_d \}$ of $\mathcal{A}$
(which follows from the fact that an invertible $d\times d$ integer-entry matrix
with an integer-entry inverse is unimodular, i.e., has determinant $\pm1$). 
If $\textbf{0}\notin\mathcal{A}$, then we use a translate $\mathcal{A'}$ of $\mathcal{A}$ that passes through $\textbf{0}$, and compute the volume of the translated polytope on $\mathcal{A'}$. Since the Euclidean volume is translation invariant, this definition is independent of the choice of~$\mathcal{A'}$.

\begin{remark}
If $d=m$, then the affine hull $\mathcal{A}$ of~$\mathcal{P}$ is $\mathbb{R}^m$, $\{\e_1,\ldots, \e_m\}$ is a basis for $L=\mathcal{A}\,\cap\,\mathbb{Z}^m$, and $\textrm{ConvexHull}(\{\textbf{0},\e_1,\ldots,\e_m\})$ is the standard $m$-simplex $\Delta_m$. 
Thus, in this case, the relative volume agrees with the full-dimensional Euclidean/Lebesgue volume,
so in general we use volume to mean the standard Lebesgue measure if $\cP$ is full-dimensional or the relative version if $\cP$ is non-full-dimensional.
\end{remark}

\subsection{Pyramids}
Some of the main computations in this paper involve the determination of volumes of pyramids. Given a polytope~$\mathcal{B}$ in $\RR^m$ and a point~$\mathbf{v}\in\RR^m$ not in the affine hull of $\mathcal{B}$, we call $\textrm{Pyr}(\mathcal{B},\mathbf{v})=\textrm{ConvexHull}(\mathcal{B}\cup\{\mathbf{v}\})$ the \emph{pyramid} over the base $\mathcal{B}$ with apex $\mathbf{v}$. 

Given a hyperplane $\mathcal{H} = \{ \mathbf{x} \in \mathbb{R}^m \mid \langle \mathbf{n}, \mathbf{x} \rangle = b \}$, with $\mathbf{n}\in\mathbb{Z}^m\setminus\{\mathbf{0}\}$, $b\in\mathbb{Z}$ and such that the greatest common divisor of the entries of $\mathbf{n}$ is~$1$, recall that the \textit{lattice distance} of a point $\mathbf{v}\in\mathbb{Z}^m$ to $\mathcal{H}$ is the absolute value of $\langle \mathbf{n}, \mathbf{v} \rangle - b$. 

Let $\mathcal{B}$ be a lattice polytope of codimension~$1$ in $\mathbb{R}^m$ and $\textrm{Pyr}(\mathcal{B},\mathbf{v})$ be the $m$-dimensional pyramid over~$\mathcal{B}$ with apex $\mathbf{v}\in\mathbb{Z}^m$.
We then have the following formula for the volume:
\begin{equation}\label{eq-pyr-vol}
\vol\left(\textrm{Pyr}(\mathcal{B},\mathbf{v})\right)=\relvol \mathcal{B}\cdot D\cdot(1/m),
\end{equation}
where $D$ is the lattice distance from $\mathbf{v}$ to the affine hull $\mathcal{A}$ of $\mathcal{B}$ (and where $\mathcal{A}$ is a hyperplane since~$\mathcal{B}$ has dimension $m-1$).
Thus, 
$$\nvol\left(\textrm{Pyr}(\mathcal{B},\mathbf{v})\right)=\relvol \mathcal{B}\cdot D\cdot(m-1)!.$$
 		
Volumes of arbitrary polytopes are notoriously hard to compute.
One strategy is to triangulate a polytope, and compute the volume of each simplex using determinants.
We will use a related technique, involving decomposition into pyramids, which relies on the following well-known result.
\begin{lemma}[Lemma 4.3.2 in \cite{triangulations}]\label{lem:coning}
Let $\cP$ be a full-dimensional lattice polytope and $\mathbf{v}$ be a vertex of $\cP$.
For each facet $\mathcal{F}$ of $\cP$ that does not contain $\mathbf{v}$, form the pyramid $\Pyr(\mathcal{F},\mathbf{v})$.
The collection of these pyramids for all such facets gives a polyhedral subdivision of $\cP$, and thus $$\vol (\cP) = \sum_{\substack{\text{facets }\cF\text{ of }\cP\\\mathbf{v}\notin\mathcal{F}}}\vol\left(\Pyr(\mathcal{F},\mathbf{v})\right).$$
\end{lemma}

\subsection{Description of the partial permutohedron}
In this section, we introduce the partial permutohedron $\cP(m,n)$, for any positive integers $m$ and $n$, using the same approach as that used by Heuer and Striker~\cite[Section~5]{HS}.

We start by defining partial permutation matrices. 
\begin{definition}
\label{pperm}
For positive integers $m$ and $n$, an $m \times n$ \emph{partial permutation matrix} is an $m\times n$ matrix with at most one nonzero entry in each row and column, where any such nonzero entry is a~$1$. Equivalently, it is an $m \times n$ matrix $M$ with entries $M_{ij}$ in $\{0,1\}$, such that
\[\textstyle\sum_{i'=1}^{m} M_{i'j} \in \left\{0,1\right\}\text{ for all } 1 \leq j \leq n,\quad \sum_{j'=1}^{n} M_{ij'} \in \left\{0,1\right\}\text{ for all } 1 \leq i \leq m.\]
We denote the set of all $m \times n$ partial permutation matrices as~$P_{m,n}$. 
Given a partial permutation matrix $M\in P_{m,n}$, its \emph{one-line notation} $w(M)$ is a word $w_1 w_2 \ldots w_m$, where $w_i=j$ if there exists $j$ such that $M_{ij}=1$, and $w_i=0$ otherwise.
\end{definition}

It can be seen that $|P_{m,n}|=\sum_{k=0}^{\min(m,n)}\binom{m}{k}\binom{n}{k}k!$.

\begin{example}
Let $$M = \begin{pmatrix} 0 & 1 & 0 & 0 & 0 & 0 \\ 0 & 0 & 0 & 0 & 0 & 1 \\ 0 & 0 & 0 & 0 & 0 & 0 \\ 0 & 0 & 1 & 0 & 0 & 0 \\ 0 &  0& 0 & 0 & 0 & 0  \end{pmatrix}.$$ Then $w(M) = 26030$.
\end{example}

It was shown by Heuer and Striker~\cite[Proposition  5.3]{HS} that $w(P_{m,n})=\{w(M)\mid M\in P_{m,n}\}$ can be characterized as the set of all words of length~$m$ with entries  in $\left\{0,1,\ldots,n\right\}$ and for which the nonzero entries are distinct.

\begin{definition}
Let the \emph{partial permutohedron} \emph{$\cP(m,n)$} be the polytope given by the convex hull of all words in $w(P_{m,n})$, as vectors in $\mathbb{R}^{m}$.
Thus, $\cP(m,n)$ is the convex hull of all vectors in $\{0,1,\ldots,n\}^m$ whose nonzero entries are distinct.
\end{definition}

It follows from the definition that~$\cP(m,n)$ is a lattice polytope.
The dimension, vertices and facets of $\cP(m,n)$ were characterized by Heuer and Striker~\cite{HS}, as follows.

\begin{proposition}[Remark 5.5 in \cite{HS}]\label{p-dim}
The partial permutohedron $\cP(m,n)$ has dimension $m$.
\end{proposition}

\begin{proposition}[Proposition 5.7 in \cite{HS}]\label{p-vertices}
The vertices of $\cP(m,n)$ are the vectors in $\bbR^m$ with entries of zero in any $m-k$ positions, and with the other $k$ entries being $n,n-1,\ldots,n-k+1$ in any order, where~$k$ ranges from $0$ to $\min(m,n)$.
It follows that $\cP(m,n)$ has $\sum_{k=0}^{\min(m,n)}\frac{m!}{(m-k)!}$ vertices.
\end{proposition}

\begin{proposition}[Theorems 5.10 and 5.11 in \cite{HS}]\label{p-facetdesc}
The facet description of $\cP(m,n)$ is
\begin{equation}\label{facetdesc}\cP(m,n)=
\left\{\mathbf{x}\in\mathbb{R}^m\:\left|\,\begin{array}{rcll}0&\leq& x_i,& \textrm{for all }i\in[m],\\[1.5mm]
\sum_{i\in S} x_i&\leq&\binom{n+1}{2}-\binom{n+1-|S|}{2},&\textrm{for all nonempty }S\subseteq[m]\\&&&\textrm{with }|S|\leq n-1\textrm{ or }|S|=m\end{array}\!\!
\right.\right\},\end{equation}
where the inequalities correspond to distinct facets,
and where $\binom{n+1-|S|}{2}$ is taken to be $0$ if $n+1-|S|\le1$ (which occurs if $|S|=m\ge n$).
It follows that
$\cP(m,n)$ has $m+\sum_{k=\max(1,m-n+1)}^m\binom{m}{k}$ facets.
\end{proposition}
Note that the facets which do not contain the origin are precisely those given by equalities in the second set of inequalities of~\eqref{facetdesc}.
Note also that for $|S|\leq n-1$ in the second of inequalities, $\binom{n+1}{2}-\binom{n+1-|S|}{2}$ can alternatively be written as $|S|\,n-\binom{|S|}{2}$.

\begin{remark}In \cite[Theorem~5.27 with $z=(n,\ldots,1)$]{HS}, it is shown that $\cP(m,n)$ is a projection of the $(m,n)$-partial permutation polytope,
which is the convex hull of the set $P_{m,n}$ of $m\times n$ partial permutation matrices,
and is also known as the polytope of $m\times n$ doubly substochastic matrices (see, for example,~\cite[Sec.~9.8]{Bru06})
and the matching polytope of the complete bipartite graph~$K_{m,n}$ (see, for example,~\cite[Chapters~18 and 25]{Schrijver}
or~\cite[Corollary~5.5]{KohOlsSan20} for $m=n$).  
In~\cite[Theorem~5.28 with $z=(n,\ldots,1)$]{HS}, it is shown that~$\cP(m,n)$ is also a projection of the $(m,n)$-partial alternating sign matrix polytope.  Indeed, the main aim of~\cite{HS} was to define and study the $(m,n)$-partial alternating sign matrix polytope.
\end{remark}

\section{Faces of the partial permutohedron}\label{sec-faces}
In this section, we explore the faces of the partial permutohedron $\cP(m,n)$.
In so doing, it will be useful to observe that, by Proposition~\ref{p-facetdesc}, the facets of $\cP(m,n)$ are:
\begin{enumerate}
\item $\{\mathbf{x}\in\cP(m,n) \mid x_i=0\}$, for all $i\in[m]$.
\smallskip
\item $\left\{\mathbf{x}\in\cP(m,n) \,\middle |\, \sum_{i\in S} x_i=\binom{n+1}{2}-\binom{n+1-|S|}{2}\right\}$, for all nonempty $S\subsetneq[m]$ with $|S|\le n-1$.
\smallskip
\item $\left\{\mathbf{x}\in\cP(m,n)\,\middle |\, \sum_{i=1}^m x_i=\begin{cases}\binom{n+1}{2}-\binom{n+1-m}{2},\textrm{ if }n\geq m\\
\binom{n+1}{2},\textrm{ if }n\le m\end{cases}\right\}$.
\end{enumerate}
Throughout this section, we use the convention that $\binom{k}{2}$ is taken to be~$0$ 
for any integer $k\le1$.  Thus, for example, the facet in~(3) above can be written as $\{\mathbf{x}\in\cP(m,n)\mid\sum_{i=1}^mx_i=\binom{n+1}{2}-
\binom{n+1-m}{2}\}$.

\subsection{Characterization of the face lattice of $\cP(m,n)$}
Heuer and Striker~\cite[Theorem 5.24]{HS} proved that, for any $m$ and $k$, the faces of $\cP(m,m)$ with dimension $k$ are in bijection with certain chains in the Boolean lattice $\mathcal{B}_m$ with $k$ so-called missing
ranks.  Heuer and Striker~\cite[Conjecture 5.25]{HS} also conjectured that this result can be generalized to $\cP(m,n)$, for any $m$, $n$ and $k$. 
We prove this conjecture in Theorem \ref{thm-faces-chains},
but first we define all of the objects needed to state the result precisely.

\begin{definition}\label{def-Boolean}
The \emph{Boolean lattice} $\mathcal{B}_m$ is the poset consisting of subsets $A\subseteq[m]$, ordered by
inclusion, where $A\in\mathcal{B}_m$ has rank $|A|$, the cardinality of $A$. A \emph{chain} $C$ in $\mathcal{B}_m$
is a nonempty ordered collection $C= (A_1\subsetneq A_2\subsetneq\cdots\subsetneq A_\ell)$ of subsets $A_i\in \mathcal{B}_m$. We say that a rank $i$ is
\emph{missing} from a chain $C$ in~$\mathcal{B}_m$ if there is no subset of rank $i$ in $C$ and there is a subset of rank greater than $i$ in $C$.
\end{definition}

\begin{remark}\label{rem-missingranks}
It follows from the definition that the number of missing ranks in a chain $(A_1\subsetneq \cdots \subsetneq A_\ell)$ 
in~$\mathcal{B}_m$ is $|A_\ell|-\ell+1$.
\end{remark}

\begin{definition}\label{definition-Cmn}
Let $\mathcal{C}(m,n)$ denote the set of all chains $(A_1\subsetneq \cdots\subsetneq A_\ell)$ in $\mathcal{B}_m$ which satisfy the following:
\begin{itemize}
\item If $A_1\neq\varnothing$, then $|A_\ell\setminus A_1|\le n-1$.
\item If $A_1=\varnothing$ and $\ell\geq2$, then $|A_\ell\setminus A_2|\le n-1$.
\end{itemize}
In other words, $\mathcal{C}(m,n)$ consists of the chain $(\varnothing)$ together with all other chains in $\mathcal{B}_m$
for which the difference in size between the largest subset and the smallest nonempty subset is at most $n-1$.
\end{definition}

\begin{remark}\label{rem-isomorphic}
If $n\ge m$, then $\cC(m,n)$ is simply the set of all chains in $\mathcal{B}_m$.
Hence, for fixed $m$, all sets $\cC(m,n)$ with $n\ge m$ are identical.
\end{remark}

We begin with the following technical result which is used in the proof of the subsequent Theorem~\ref{thm-faces-chains}.
\begin{proposition}\label{p-simple}
The partial permutohedron $\cP(m,n)$ is a simple polytope.
\end{proposition}

\begin{proof}
Since, by Proposition~\ref{p-dim}, $\cP(m,n)$ is $m$-dimensional, this result follows from the fact that each vertex of $\cP(m,n)$ is contained in exactly $m$ facets.
Specifically, by Proposition~\ref{p-vertices}, for any vertex $\mathbf{v}$ of $\cP(m,n)$, there exist unique $i_1,\ldots,i_k\in[m]$
such that $v_j=0$ for $j\in[m]\setminus\{i_1,\ldots,i_k\}$, and $v_{i_j}=n-j+1$ for $j=1,\ldots,k$.
It can then be seen that $\mathbf{v}$ is contained in the $m-k$ facets $\{\mathbf{x}\in\cP(m,n)\mid x_j=0\}$ for $j\in[m]\setminus\{i_1,\ldots,i_k\}$,
the $k-1$ facets $\{\mathbf{x}\in\cP(m,n)\mid x_{i_1}+\ldots+x_{i_j}=\binom{n+1}{2}-\binom{n-j+1}{2}\}$ for $j=1,\ldots,k-1$,
and the single facet $\{\mathbf{x}\in\cP(m,n)\mid x_{i_1}+\ldots+x_{i_k}=\binom{n+1}{2}-\binom{n-k+1}{2}\}$ if
$k\ne n$, or $\{\mathbf{x}\in\cP(m,n)\mid\sum_{i=1}^mx_i=\binom{n+1}{2}\}$ if $k=n$ (which implies $n\le m$).
Furthermore, $\mathbf{v}$ is not contained in any other facets.
\end{proof}

We are now ready to prove Conjecture 5.25 of~\cite{HS} for the faces of~$\cP(m,n)$ with any~$m$ and~$n$.
Recently, an alternative proof of this conjecture was independently obtained by Black and Sanyal~\cite[Theorem~7.5]{black2022flag} in the context of monotone path polytopes of polymatroids.  Related results, including expressions for the $f$-vector, are obtained 
in the context of parking function polytopes (see Remarks~\ref{rem-parkingfunctionpolytope} and~\ref{rem-parkingfunctionpolytopegen}) in~\cite[Section~3]{AW} for the case $n=m-1$, and in~\cite[Propositions~3.14 and~3.15]{HanadaLentferVindasMelendez} for $n\ge m-1$.

\begin{theorem}\label{thm-faces-chains}
Given a chain $C=(A_1\subsetneq \cdots\subsetneq A_\ell)$ in $\mathcal{C}(m,n)$, let $\cF_C$ be the intersection of $\cP(m,n)$ with the following hyperplanes:
\renewcommand{\labelenumi}{(\roman{enumi})}
\begin{enumerate}
\item $\{\mathbf{x}\in\bbR^m \mid x_i=0\}$, for all $i\in[m]\setminus A_\ell$.
\item $\left\{\mathbf{x}\in\bbR^m\,\middle |\, \sum_{i\in A_\ell\setminus A_j} x_i=\binom{n+1}{2}-\binom{n+1-|A_\ell\setminus A_{j}|}{2}\right\}$, for all $2\le j\le \ell-1$,
and also for $j=1$ unless $A_1=\varnothing$ and $|A_\ell|\ge n$.
\item $\left\{\mathbf{x}\in\bbR^m\,\middle |\,\sum_{i=1}^mx_i=\binom{n+1}{2}\right\}$, if $A_1=\varnothing$ and $|A_\ell|\ge n$.
\end{enumerate}
Then the following is a bijection:
\begin{align*}
\mathcal{C}(m,n)&\longrightarrow \{\text{nonempty faces of } \cP(m,n)\},\\
C&\longmapsto \cF_C.
\end{align*}
Moreover, this bijection maps chains with $k$ missing ranks to faces of dimension $k$,
for each $k=0,\ldots,m$.
\end{theorem}

\begin{proof}
We start by showing that we have a well-defined map from $\mathcal{C}(m,n)$ to the set of 
nonempty faces of~$\cP(m,n)$, i.e., that $\cF_C$ is a nonempty face of $\cP(m,n)$, for all $C\in \mathcal{C}(m,n)$.
Observe that for any one of the $m-|A_\ell|+\ell-1$ hyperplanes $\mathcal{H}$ in the definition of $\cF_C$, the intersection
of $\cP(m,n)$ with~$\cH$ is a facet of $\cP(m,n)$.  Specifically, using the numbering of facet types
given at the start of Section~\ref{sec-faces}
and the numbering of hyperplane types given in the definition of $\cF_C$, if~$\cH$ is a hyperplane of type~(i)
then $\cP(m,n)\cap\cH$ is a facet of type~(1), if~$\cH$ is a hyperplane of type~(ii) 
then $\cP(m,n)\cap\cH$ is a facet of type~(2) or (if $j=1$, $A_1=\varnothing$, $A_\ell=[m]$ and $n>m$) type~(3), and if~$\cH$ is a hyperplane of type~(iii)
then $\cP(m,n)\cap\cH$  is a facet of type~(3).
Since any intersection of facets of a polytope is a face of the polytope, it follows that $\cF_C$ is a face of $\cP(m,n)$.
Furthermore,~$\cF_C$ is nonempty since it contains certain vertices of $\cP(m,n)$
(which are thus the vertices of~$\cF_C$), as follows.
Essentially, each such vertex can be obtained as a vector in $\RR^m$ by placing $0$'s into 
positions $[m]\setminus A_\ell$, placing the largest possible entries 
(specifically, $n,n-1,\ldots,n-|A_\ell\setminus A_{\ell-1}|+1$) in any order into positions $A_\ell\setminus A_{\ell-1}$, placing the next largest possible entries (specifically, $n-|A_\ell\setminus A_{\ell-1}|,\ldots,n-|A_\ell\setminus A_{\ell-2}|+1$) in any order into $A_{\ell-1}\setminus A_{\ell-2}$, etc., and placing
the smallest possible entries (which may include 0's) in any order into either~$A_1$ (if $A_1\ne\varnothing$) or $A_2$ (if $A_1=\varnothing$ and $|A_\ell|\ge n$). For full details
of this construction and its validity, see Proposition~\ref{prop-facevertices} below.

Proceeding to the injectivity of the map, this is a straightforward consequence of Proposition~\ref{p-simple}, i.e., that $\cP(m,n)$ is simple.
Concretely, in a simple polytope, each nonempty intersection of facets determines a unique face.

We now show that the map is surjective.
Consider any nonempty face $F$ of $\cP(m,n)$. Then, using~\eqref{facetdesc},
there exist $T\subseteq [m]$ and 
\begin{equation}\label{eq-S-surj}\cS\subseteq\bigl\{\varnothing\subsetneq S\subseteq[m]\,\bigm|\,|S|\le n-1\text{ or }|S|=m\bigr\},\end{equation}
such that
\begin{equation}\label{eq-F-surj}F=\left\{\mathbf{x}\in\cP(m,n)\left|\:\begin{array}{@{}ll@{}}x_i=0,&\text{for all }i\in T,\\
\sum_{i\in S}x_i=\binom{n+1}{2}-\binom{n+1-|S|}{2},&\text{for all }S\in\mathcal{S}\end{array}\right.\right\}.
\end{equation}
We claim that if $S',S\in \cS$ are such that $|S'|\le|S|$, then $S'\subseteq S$.
The claim can immediately be seen to hold if $S=[m]$. 
So, consider now the remaining cases of 
$S',S\in\cS$ with $|S'|\le|S|\le n-1$, and let $\mathbf{v}$ be a vertex of $F$.
Then since $\mathbf{v}$ is also a vertex of $\cP(m,n)$, we have
$$\{v_i\mid i\in S'\}=\{n,n-1,\ldots,n+1-|S'|\}
\subseteq\{n,n-1,\ldots,n+1-|S|\}=\{v_i\mid i\in S\},$$
where the containment follows from $|S'|\le|S|$,
and the equalities follow from the form of vertices given by Proposition~\ref{p-vertices},
together with $\mathbf{v}\in F$, $\sum_{i\in S'}v_i=\binom{n+1}{2}-\binom{n+1-|S'|}{2}$,
$\sum_{i\in S}v_i=\binom{n+1}{2}-\binom{n+1-|S|}{2}$ and $|S'|,|S|\le n-1$.
The conclusion $S'\subseteq S$ follows using the fact
(as given by Proposition~\ref{p-vertices})
that all nonzero entries of $\mathbf{v}$ are distinct.
Using this claim, we can order the elements of $\cS$ as
$$S_1\subsetneq\cdots\subsetneq S_{\ell-1},$$
where $\ell=|\cS|+1$. 
We obtain a chain $C=(A_1 \subsetneq\cdots\subsetneq A_\ell)\in\cC(m,n)$ by setting $A_\ell=[m]\setminus T$ and then $A_{\ell-1}=A_\ell\setminus S_1$, \ldots, $A_1=A_\ell\setminus S_{\ell-1}$. 
(Note that $A_{\ell-1}\subseteq A_\ell$ follows immediately from $A_\ell=[m]\setminus T$ and $A_{\ell-1}=A_\ell\setminus S_1=[m]\setminus(S_1\cup T)$,
and $A_{\ell-1}\ne A_\ell$ follows by noting that we cannot have $S_1\cup T=T$, or equivalently cannot have $S_1\subset T$, 
since~\eqref{eq-F-surj} would then give the contradiction that $\mathbf{x}\in F$ satisfies $x_i=0$ for all $i\in S_1$ and 
$\sum_{i\in S_1}x_i=\binom{n+1}{2}-\binom{n+1-|S_1|}{2}\ne0$.)
It can be seen that we now have
$F=\cF_C$, thereby confirming the surjectivity of the map.

Lastly, we proceed to the dimension of $\cF_C$.
For a simple polytope of dimension~$d$, any nonempty intersection of $i$ distinct facets is a face of dimension $d-i$.  Therefore, using the facts that $\cP(m,n)$ is simple by Proposition~\ref{p-simple}, that $\cP(m,n)$ has dimension $m$ by Proposition~\ref{p-dim}, 
and that~$\cF_C$ is the nonempty intersection of $m-|A_\ell|+\ell-1$ facets
(which can easily be seen to be distinct), it follows that
\begin{equation*}\dim(\cF_C)=|A_\ell|-\ell+1,\end{equation*}
which, by Remark~\ref{rem-missingranks}, is the number of missing ranks of $C$, as required.
\end{proof}

Some remarks on Theorem~\ref{thm-faces-chains} are as follows.

\begin{remark}\label{rem:faces1}
For any chain $C=(A_1\subsetneq \cdots\subsetneq A_\ell)$ in $\mathcal{C}(m,n)$, the face
$\cF_C$ of $\cP(m,n)$ defined in Theorem~\ref{thm-faces-chains} can be expressed more compactly as
\begin{equation}\label{eq-Fc-compact}\textstyle
\cF_C=\left\{\mathbf{x}\in\cP(m,n)\,\middle|\,\sum_{i\in[m]\setminus A_j}x_i=\binom{n+1}{2}-
\binom{n+1-|A_\ell\setminus A_j|}{2},\text{ for all }1\le j\le\ell\right\}\end{equation}
(where $\binom{n+1-|A_\ell\setminus A_j|}{2}=0$ 
occurs if $j=1$, $A_1=\varnothing$ and $|A_\ell|\ge n$).
To check the validity of~\eqref{eq-Fc-compact}, observe that 
the case $j=\ell$ gives 
$\sum_{i\in[m]\setminus A_\ell}x_i=0$, which (since each entry of any $\mathbf{x}\in\cP(m,n)$ is nonnegative)
is equivalent to $x_i=0$ for all $i\in [m]\setminus A_\ell$, so that this case corresponds to the
intersection of $\cP(m,n)$ with all hyperplanes of type~(i) in Theorem~\ref{thm-faces-chains}.
The cases $1\le j\le\ell-1$ in~\eqref{eq-Fc-compact} give 
\begin{equation*}\textstyle\sum_{i\in[m]\setminus A_j}x_i=\sum_{i\in[m]\setminus A_\ell}x_i+\sum_{i\in A_\ell\setminus A_j}x_i=\binom{n+1}{2}-\binom{n+1-|A_\ell\setminus A_j|}{2},\end{equation*}
which using $\sum_{i\in [m]\setminus A_\ell}x_i=0$ from the case $j=\ell$ becomes
\begin{equation*}\textstyle
\sum_{i\in A_\ell\setminus A_j}x_i=\binom{n+1}{2}-\binom{n+1-|A_\ell\setminus A_j|}{2},\end{equation*}
so that these cases correspond to the intersection of $\cP(m,n)$ with all hyperplanes of types~(ii) and~(iii).
\end{remark}

\begin{remark}\label{rem:faces2}
It can be seen that, under the bijection of Theorem~\ref{thm-faces-chains},
the chains $(\varnothing)$, $([m])$ and $(\varnothing\subsetneq[m])$ 
(which are contained in $\mathcal{C}(m,n)$ for any~$m$ and~$n$) 
are mapped to the faces 
\begin{align*}
\cF_{(\varnothing)}&=\{\mathbf{0}\},\\
\cF_{([m])}&=\cP(m,n),\\
\cF_{(\varnothing\subsetneq[m])}&=\textstyle\bigl\{\mathbf{x}\in\cP(m,n)\bigm|\sum_{i=1}^mx_i=\binom{n+1}{2}-\binom{n-m+1}{2}\bigr\}.\end{align*}
\end{remark}

\begin{remark}\label{rem-face-lattice}
By extending $\cC(m,n)$ to a set $\cC'(m,n)=\{(\,)\}\cup\cC(m,n)$, where $(\,)$ is regarded as an empty chain,
and by defining $\cF_{(\,)}=\varnothing$, where $\varnothing$ is the empty face of $\cP(m,n)$,
we obtain an extension of the mapping of Theorem~\ref{thm-faces-chains} which is
a bijection from $\cC'(m,n)$ to the set of all faces of $\cP(m,n)$.
Through this bijection, the face lattice of $\cP(m,n)$ (i.e., the lattice formed by the set of faces of $\cP(m,n)$, ordered by containment)
now induces a partial order on $\cC'(m,n)$ (i.e., for $C_1,C_2\in\cC'(m,n)$, the partial order is defined by
$C_1\le C_2$ if and only if $\cF_{C_1}\subseteq\cF_{C_2}$).
This partial order can be described directly in terms of $\cC'(m,n)$ as follows.
For $C=(A_1\subsetneq\cdots\subsetneq A_\ell)\in\cC'(m,n)$, let
\begin{equation}\label{eq-RC}\cR_C=\begin{cases}[m]\,\cup\,\{\varnothing\subsetneq S\subseteq[m]\mid |S|\le n-1\text{ or }|S|=m\},&\text{if }C=(\,),\\
([m]\setminus A_\ell)\,\cup\,\{A_\ell\setminus A_{\ell-1},\ldots,A_\ell\setminus A_1\},&\text{if }A_1\ne\varnothing\text{ or }|A_\ell|\le n-1,\\
([m]\setminus A_\ell)\,\cup\,\{A_\ell\setminus A_{\ell-1},\ldots,A_\ell\setminus A_2,[m]\},&\text{if }A_1=\varnothing\text{ and }|A_\ell|\ge n.
\end{cases}\end{equation}
Then, for $C_1,C_2\in\cC'(m,n)$, we have
\begin{equation}\label{eq-RC-2}C_1\le C_2\quad\text{if and only if}\quad\cR_{C_2}\subseteq\cR_{C_1}.\end{equation}
The validity of this characterization 
can easily be checked by observing that, for $C\in\cC'(m,n)$, $\cR_C$ has the form $T\cup S$ for some $T\in[m]$ and $S\in\cS$,
where $\cS$ is given by~\eqref{eq-S-surj}, and that the face $F=\mathcal{F}_C$ which corresponds to~$C$ is given by~\eqref{eq-F-surj}.
\end{remark}

An example which illustrates the partial order on $\cC'(m,n)$, as characterized in Remark~\ref{rem-face-lattice}, is as follows.
\begin{example}
Let $m=6$ and $n=4$, and consider the chains
\begin{equation*}C_1=(\varnothing\subsetneq\{1,2\}\subsetneq\{1,2,3\}\subsetneq\{1,2,3,4\})
\quad\text{and}\quad C_2=(\{1,2,5\}\subsetneq\{1,2,3,4,5\})\end{equation*}
in $\cC(6,4)$.  Then, using~\eqref{eq-RC}, we have
\begin{equation*}\cR_{C_2}=\{6,\,\{3,4\}\}\subseteq\{5,\,6,\,\{4\},\,\{3,4\},\,\{1,2,3,4,5,6\}\}=
\cR_{C_1},\end{equation*}
and so, using~\eqref{eq-RC-2}, we have $C_1\le C_2$.
\end{example}

Using Remarks~\ref{rem-isomorphic} and~\ref{rem-face-lattice}, we obtain the following corollary to Theorem~\ref{thm-faces-chains}.
Recall that two polytopes are defined to be combinatorially equivalent if their face lattices are isomorphic.

\begin{corollary}\label{cor:comb-iso}
For fixed $m$, all partial permutohedra $\cP(m,n)$ with $n\ge m$ are combinatorially equivalent.
\end{corollary}

\begin{proof}
Let $\cC(m)$ denote the set of all chains (including the empty chain) in~$\cB_m$, and 
consider $\cP(m,n_1)$ and $\cP(m,n_2)$ with $n_1,n_2\ge m$.  Then, by Remark~\ref{rem-isomorphic}, Theorem~\ref{thm-faces-chains} and Remark~\ref{rem-face-lattice},
the set of faces of $\cP(m,n_1)$ and set of faces of $\cP(m,n_2)$ are in bijection, since each of these sets is in bijection with~$\cC(m)=\cC'(m,n_1)=\cC'(m,n_2)$.
Furthermore, the partial order on~$\cC(m)$ induced by the bijection with the set of faces of $\cP(m,n_1)$ is the same as that induced by the bijection with the set of faces of $\cP(m,n_2)$, where this can be seen
as follows.
For $n\ge m$,~\eqref{eq-RC} simplifies to 
\begin{equation}\label{eq-RC1}\cR_C=\begin{cases}[m]\,\cup\,\{S\mid\varnothing\subsetneq S\subseteq[m]\},&\text{if }C=(\,),\\
([m]\setminus A_\ell)\,\cup\,\{A_\ell\setminus A_{\ell-1},\ldots,A_\ell\setminus A_1\},&\text{otherwise},
\end{cases}\end{equation}
for any $C=(A_1\subsetneq\cdots\subsetneq A_\ell)\in\cC(m)=\cC'(m,n)$,
since in the formula for the first case in~\eqref{eq-RC}, the condition $|S|\le n-1$ or $|S|=m$ is equivalent to 
$|S|\le m$, and in the third case in~\eqref{eq-RC}, the condition $A_1=\varnothing$ and $|A_\ell|\ge n$ holds only if $A_1=\varnothing$ and $m=n=|A_\ell|$, for which the formula in the second case in~\eqref{eq-RC} can
be used instead.  Therefore, for $n\ge m$, $\cR_C$ is independent of~$n$, 
and so, using~\eqref{eq-RC-2}, the partial order induced on~$\cC(m)$ is independent of~$n$.
It now follows that the face lattices of $\cP(m,n_1)$ and $\cP(m,n_2)$ are isomorphic, as required.\end{proof}

\begin{remark}\label{rem-stell}
It is noted by Heuer and Striker~\cite[Theorem 5.17]{HS} that~$\cP(m,m)$
is combinatorially equivalent to the $m$-stellohedron, and hence, using Corollary~\ref{cor:comb-iso}, it follows that all $\cP(m,n)$ with $n\ge m$ are combinatorially equivalent to the $m$-stellohedron. 
For completeness, we now provide a definition of the $m$-stellohedron. For a graph~$G$ with vertex set~$[N]$, the graph associahedron of~$G$ can be
defined as $\textrm{Assoc}(G)=\sum_S\Conv(\{\e_i\mid i\in S\})$,
where this is a Minkowski sum over all nonempty and nonsingleton subsets~$S$ of~$[N]$ such that the subgraph of~$G$ induced by $S$ is connected, and $\e_i$ is the $i$th standard unit vector in $\RR^N$. 
(Singletons could be included here, which would simply result in a translation by
(1,\ldots,1).)
The $m$-stellohedron is the graph associahedron
of the star graph $K_{1,m}$, consisting of a central vertex $m+1$ connected to~$m$ vertices~$1,\ldots,m$, which gives 
\begin{equation}\textrm{Assoc}(K_{1,m})=\sum_{\varnothing\subsetneq S\subseteq[m]}\Conv(\{\e_i\mid i\in S\}\cup\{\e_{m+1}\}),\end{equation}
where $\e_i$ is in $\RR^{m+1}$.  It will also be
useful to consider a projection of $\textrm{Assoc}(K_{1,m})$ from $\RR^{m+1}$ to~$\RR^m$. Let $\psi\colon\RR^{m+1}\to\RR^m$ be given
by $(x_1,\ldots,x_{m+1})\mapsto(x_1,\ldots,x_m)$.  It can then be shown that $\psi\bigl(\textrm{Assoc}(K_{1,m})\bigr)$ is affinely isomorphic
(and hence combinatorially equivalent) to $\textrm{Assoc}(K_{1,m})$,
and that 
\begin{equation}\psi\bigl(\textrm{Assoc}(K_{1,m})\bigr)
=\sum_{\varnothing\subsetneq S\subseteq[m]}\Conv(\{\mathbf{0}\}\cup\{\e_i\mid i\in S\})
=\widehat{\Pi}(2^{m-1},2^{m-2},\ldots,2,1),\end{equation}
where $\e_i$ is in $\RR^m$,
and the final expression uses notation which will be introduced in Definition~\ref{def:Pi}.
\end{remark}

\begin{remark}
A generalization to all~$m$ and~$n$ of the combinatorial equivalence
of Remark~\ref{rem-stell} is obtained in~\cite[Corollary~7.4]{black2022flag}
(and is used in the proof of Theorem~\ref{thm-faces-chains} given
in~\cite[Theorem~7.5]{black2022flag}).
Specifically, it follows from~\cite[Corollary~7.4]{black2022flag} that~$\cP(m,n)$ with any~$m$ and~$n$ is combinatorially equivalent to 
\begin{equation}\label{blacksanyaldecomp1}
\sum_{\substack{S\subseteq[m]\\|S|\ge\max(1,m-n+1)}}\!\!\!\Conv(\{\e_i\mid i\in S\}\cup\{\e_{m+1}\}),\end{equation}
where this is a Minkowski sum, and~$\e_i$ is in $\RR^{m+1}$.  
It can also be shown, by using the projection~$\psi$ defined in Remark~\ref{rem-stell}, that $\cP(m,n)$ with any~$m$ and~$n$ is affinely isomorphic (and hence combinatorially equivalent) to the image under $\psi$ of~\eqref{blacksanyaldecomp1}, which is 
\begin{multline}
\sum_{\substack{S\subseteq[m]\\|S|\ge\max(1,m-n+1)}}\!\!\!\Conv(\{\mathbf{0}\}\cup\{\e_i\mid i\in S\})\\[-1mm]
=\begin{cases}\rule{0mm}{5.4mm}\widehat{\Pi}(2^{m-1},2^{m-2},\ldots,2,1),&\text{if }n\ge m,\\[2.8mm]
\widehat{\Pi}\Bigl(\sum_{i=0}^{n-1}\binom{m-1}{i},\sum_{i=0}^{n-2}\binom{m-2}{i},\ldots,\binom{m-n}{0},\underbrace{0,\ldots,0\!}_{m-n}\,\Bigr),&\text{if }n\le m,\end{cases}\end{multline}
where~$\e_i$ is in $\RR^m$,
and the right hand side uses notation which will be introduced in Definition~\ref{def:Pi}.  
Note also that an alternative Minkowski sum decomposition which is equal, rather 
than just combinatorially equivalent, to~$\cP(m,n)$ with $n\ge m-1$, will be given in~\eqref{eq-Minkowski1}. 
\end{remark}

Some examples which illustrate Theorem~\ref{thm-faces-chains} are as follows.

\begin{example}\label{ex-Pm1}
As an example of the bijection of Theorem~\ref{thm-faces-chains} with $m\ge n$, 
consider $\cP(m,1)$, which can easily be seen to be the standard $m$-simplex~$\Delta_m$.
Then the bijection from $\mathcal{C}(m,1)$ to the set of nonempty faces of $\cP(m,1)=\Delta_m$ is
\begin{align*}
(S)&\mapsto\{\mathbf{x}\in\Delta_m\mid x_i=0\text{ for all }i\in[m]\setminus S\}\\
&=\textrm{ConvexHull}(\{\mathbf{0}\}\cup\{\e_i\mid i\in S\}),\quad\text{for each }S\subseteq[m],\\[1mm]
(\varnothing\subsetneq S)&\mapsto\textstyle\bigl\{\mathbf{x}\in\Delta_m\bigm|x_i=0\text{ for all }i\in[m]\setminus S,\ \sum_{i=1}^mx_i=1\bigr\}\\
&=\textrm{ConvexHull}(\{\e_i\mid i\in S\}),\quad\text{for each nonempty }S\subseteq[m],\end{align*}
where a face corresponding to $(S)$ has dimension $|S|$, and a face corresponding to $(\varnothing\subsetneq S)$ has dimension $|S|-1$.
\end{example}

\begin{example}
As an example of the bijection of Theorem~\ref{thm-faces-chains} with $n\ge m$, consider the pentagon
\begin{align*}\cP(2,n)&=\{\mathbf{x}\in\RR^2\mid0\le x_1\le n,\ 0\le x_2\le n,\ x_1+x_2\le2n-1\}\\
&=\textrm{ConvexHull}(\{(0,0),\,(n,0),\,(0,n),\,(n,n-1),\,(n-1,n)\}),\end{align*}
for $n\ge2$.
Then the bijection from $\mathcal{C}(2,n)$ to the set of nonempty faces of $\cP(2,n)$ is
\begin{equation*}\begin{array}{l@{\qquad}l}
(\varnothing)\mapsto\{(0,0)\},\\
(\varnothing\subsetneq\{1\})\mapsto\{(n,0)\},&
(\varnothing\subsetneq\{2\})\mapsto\{(0,n)\},\\
(\varnothing\subsetneq\{1\}\subsetneq[2])\mapsto\{(n-1,n)\},&
(\varnothing\subsetneq\{2\}\subsetneq[2])\mapsto\{(n,n-1)\},\\
(\{1\})\mapsto[(0,0),(n,0)],&
(\{2\})\mapsto[(0,0),(0,n)],\\
(\{1\}\subsetneq[2])\mapsto[(0,n),(n-1,n)],&
(\{2\}\subsetneq[2]\})\mapsto[(n,0),(n,n-1)],\\
(\varnothing\subsetneq[2])\mapsto[(n,n-1),(n-1,n)],&
([2])\mapsto\cP(2,n),\end{array}\end{equation*}
where $[\mathbf{v},\mathbf{w}]$ denotes an edge between vertices $\mathbf{v}$ and $\mathbf{w}$
of $\cP(2,n)$.
\end{example}

\begin{example}\label{ex-Fmn}
As a final example of the bijection of Theorem~\ref{thm-faces-chains}, let $n\ge m$ and consider the facet $F(m,n)$ of~$\cP(m,n)$ which corresponds to the chain $(\varnothing\subsetneq[m])$, i.e.,
$F(m,n)=\cF_{(\varnothing\subsetneq[m])}=\bigl\{\mathbf{x}\in\cP(m,n)\bigm|\sum_{i=1}^mx_i=\binom{n+1}{2}-\binom{n-m+1}{2}\bigr\}$.
It can be seen, using Proposition~\ref{p-vertices}, that the vertices of~$F(m,n)$ consist of all vectors
obtained by permuting the entries of $(n,n-1,\ldots,n-m+1)$, so that~$F(m,n)$ is the permutohedron $\Pi(n,n-1,\ldots,n-m+1)$.
Using Remarks~\ref{rem-isomorphic} and~\ref{rem-face-lattice}, any nonempty face of $F(m,n)$ corresponds to a chain~$C$ in $\cB_m$ satisfying $\cR_{(\varnothing\subsetneq[m])}=\{[m]\}\subseteq\cR_C$,
which implies (using~\eqref{eq-RC} or~\eqref{eq-RC1}) that $C$ has the form $(\varnothing\subsetneq A_1\subsetneq\cdots\subsetneq A_{\ell}\subsetneq[m])$.
Accordingly, let $\widehat{\cC}(m)$ denote the set of all chains $(\varnothing\subsetneq A_1\subsetneq\cdots\subsetneq A_{\ell}\subsetneq[m])$ in $\cB_m$.
Then the restriction to $\widehat{\cC}(m)$ of the bijection of Theorem~\ref{thm-faces-chains}
is the well-known bijection (see, for example,~\cite[Proposition~2.6]{postnikov})
from $\widehat{\cC}(m)$ to the set of nonempty faces of $F(m,n)=\Pi(n,n-1,\ldots,n-m+1)$.
Specifically, $C=(\varnothing\subsetneq A_1\subsetneq \cdots\subsetneq A_\ell\subsetneq[m])\in\widehat{\cC}(m)$ is mapped to 
the face $\cF_C=\bigl\{\mathbf{x}\in F(m,n)\bigm|\sum_{i\in[m]\setminus A_j}x_i=\binom{m+1}{2}-\binom{|A_j|+1}{2},\text{ for all }1\le j\le\ell\bigr\}$,
which has dimension $m-|\ell|-1$.
\end{example}

For any $C\in \mathcal{C}(m,n)$, a construction of the vertices of the face $\cF_C$ was outlined briefly  within the proof of Theorem~\ref{thm-faces-chains},
in order to show to $\cF_C$ is nonempty.  In the following proposition, this construction is described in detail.

\begin{proposition}\label{prop-facevertices}
Let $C=(A_1\subsetneq \cdots\subsetneq A_\ell)$ be any chain in $\mathcal{C}(m,n)$, and consider the face $\cF_C$ of $\cP(m,n)$, as defined in Theorem~\ref{thm-faces-chains}.
Then the vertices of $\cF_C$ are the vectors in $\RR^m$ which satisfy the following:
\begin{enumerate}
\item All of the entries in positions $[m]\setminus A_\ell$ are $0$'s.
\item The entries in positions $A_{j+1}\setminus A_j$ are
\begin{equation*}n-|A_\ell\setminus A_{j+1}|,\:n-|A_\ell\setminus A_{j+1}|-1,\:\ldots,\:n-|A_\ell\setminus A_j|+1,\end{equation*}
in any order, for all $2\le j\le\ell-1$
and also for $j=1$ unless $A_1=\varnothing$ and $|A_\ell|\ge n$.
\item If $A_1\ne\varnothing$, then the entries in positions $A_1$ are
\begin{equation*}
n-|A_\ell\setminus A_1|,\:n-|A_\ell\setminus A_1|-1,\:\ldots,\:n-|A_\ell\setminus A_1|-k+1,\:\underbrace{0,\:\ldots,\:0\!}_{|A_1|-k}\,,\end{equation*}
in any order, and for any $0\le k\le\min(|A_1|,n-|A_\ell\setminus A_1|)$.
\item If $A_1=\varnothing$ and $|A_\ell|\ge n$, then the entries in positions $A_2$ are
\begin{equation*}n-|A_\ell\setminus A_2|,\:n-|A_\ell\setminus A_2|-1,\:\ldots,\:1,\:\underbrace{0,\:\ldots,\:0\!}_{|A_\ell|-n}\,,\end{equation*}
in any order.
\end{enumerate}
\end{proposition}
\begin{proof}
The result can be proved by using the characterization of the vertex set of $\cP(m,n)$ given in Proposition~\ref{p-vertices},
together the fact that the vertex set of $\cF_C$ is the intersection of the vertex set of $\cP(m,n)$ with the hyperplanes given in (i)--(iii) of Theorem~\ref{thm-faces-chains}.  
The validity of conditions~(1) and~(2) in the current proposition will be now be confirmed.  The validity of conditions~(3) and~(4)
can also be checked straightforwardly, but the details of this will be omitted.

It can immediately be seen that the intersection of the vertex set of~$\cP(m,n)$ with the hyperplanes of type~(i) gives condition~(1).
The intersection with the hyperplanes of type~(ii) implies that a vertex $\mathbf{v}$ of $\cF_C$ satisfies
$\sum_{i\in A_\ell\setminus A_j}v_i=\binom{n+1}{2}-\binom{n+1-|A_\ell\setminus A_{j}|}{2}$,
i.e.,
\begin{equation}\label{eq-vert-cond1}\textstyle\sum_{i\in A_\ell\setminus A_j}v_i=n+(n-1)+\ldots+(n-|A_\ell\setminus A_j|+1),\end{equation}
for all $2\le j\le \ell-1$, and also for $j=1$ unless $A_1=\varnothing$ and $|A_\ell|\ge n$.
By considering, in~\eqref{eq-vert-cond1}, the $j=\ell-1$ case, the $j=\ell-2$ case minus the $j=\ell-1$ case, 
the $j=\ell-3$ case minus the $j=\ell-2$ case, etc., it follows that the equations of~\eqref{eq-vert-cond1}
are equivalent to 
\begin{equation}\label{eq-vert-cond2}\textstyle\sum_{i\in A_{j+1}\setminus A_j}v_i=(n-|A_\ell\setminus A_{j+1}|)+(n-|A_\ell\setminus A_{j+1}|-1)+\ldots+(n-|A_\ell\setminus A_j|+1),\end{equation}
for the same range of $j$.  Using the characterization of $\mathbf{v}$ given by Proposition~\ref{p-vertices}
(and especially the property that any nonzero entries of $\mathbf{v}$ are distinct, and take the values $n$, $n-1$, \ldots, $n-k+1$ for some~$k$),
it now follows that the equations of~\eqref{eq-vert-cond2} give condition~(2).
\end{proof}

We illustrate Proposition~\ref{prop-facevertices} in the following example.
\begin{example}\label{ex-vertface}
Let $m=10$ and $n=6$, and consider the chains
\begin{align*}C_1&=(\{1,2,3\}\subsetneq\{1,2,3,4,5\}\subsetneq\{1,2,3,4,5,6,7\})\\
\text{and \ }C_2&=(\varnothing\subsetneq\{1,2,3\}\subsetneq\{1,2,3,4,5\}\subsetneq\{1,2,3,4,5,6,7\})\end{align*}
in $\cC(10,6)$.
Then the vertices of
$\cF_{C_1}=\{\mathbf{x}\in\cP(10,6)\mid x_8=x_9=x_{10}=0,\, x_6+x_7=11,\,x_4+x_5+x_6+x_7=18\}
=\{\mathbf{x}\in\cP(10,6)\mid x_8=x_9=x_{10}=0,\, x_6+x_7=11,\,x_4+x_5=7\}$
are of the form
\[\boxed{012}\boxed{34}\boxed{56}\:000,\ \ \boxed{002}\boxed{34}\boxed{56}\:000\text{ \ or \ }\;000\:\boxed{34}\boxed{56}\:000,\]
and the vertices of
$\cF_{C_2}=\{\mathbf{x}\in\cP(10,6)\mid x_8=x_9=x_{10}=0,\, x_6+x_7=11,\,x_4+x_5+x_6+x_7=18,\,
x_1+\ldots+x_{10}=21\}
=\{\mathbf{x}\in\cP(10,6)\mid x_8=x_9=x_{10}=0,\, x_6+x_7=11,\,x_4+x_5=7,\,x_1+x_2+x_3=3\}$
are of the form
\[\boxed{012}\boxed{34}\boxed{56}\:000,\]
where the entries within each box can appear in any order.
It follows that $\cF_{C_1}$ has $3!\cdot2!\cdot2!+3\cdot2!\cdot2!+2!\cdot2!=40$
vertices, and that $\cF_{C_2}$ has $3!\cdot2!\cdot2!=24$  vertices.
\end{example}

\subsection{The $h$-polynomial of $\cP(m,n)$}
We now compute the $h$-polynomial of $\cP(m,n)$ using Theorem~\ref{thm-faces-chains}.

Given a $d$-dimensional polytope $P$ and $0\le i\le d$, let $f_i(P)$ denote the number of $i$-dimensional faces of $P$.
The $f$-\emph{polynomial} of~$P$ is then defined as $f_P(t)=\sum_{i=0}^df_i(P)\,t^i$, and
the $h$-\emph{polynomial} of~$P$ is defined as $h_P(t)=f_P(t-1)$.  It is known that if~$P$ is a simple polytope,
then its $h$-polynomial is palindromic, i.e.,
\begin{equation}\label{eq-h-palind}
h_P(t)=t^d\,h_P(1/t),\end{equation}
and this corresponds to the Dehn--Sommerville relations for $f_i(P)$.

Since Eulerian polynomials will play a role in the $h$-polynomial of $\cP(m,n)$, we proceed to introduce them.
For a positive integer $m$, let $A_m(t)$ denote the Eulerian polynomial for~$\mathfrak{S}_m$, i.e., $A_m(t)=\sum_{i=0}^{m-1}A(m,i)\,t^i$,
where the Eulerian number $A(m,i)$ is the number of permutations in~$\mathfrak{S}_m$ with exactly $i$ descents. Also, let $A(0,i)=\delta_{0,i}$ and $A_0(t)=1$.  It can easily be seen that the Eulerian numbers and polynomials satisfy the symmetry
\begin{equation}\label{eq-Eulerian-symm}A(m,i)=A(m,m-i-1)\quad\text{and}\quad A_m(t)=t^{m-1}A_m(1/t),\quad\text{for positive }m.\end{equation}

The $f$-polynomials and $h$-polynomials of the standard $m$-simplex $\Delta_m$ and regular permutohedron
$\Pi(1,\ldots,m)$ are known to be 
\begin{gather}\label{eq-f-simp-perm}
f_{\Delta_m}(t)=\sum_{i=0}^m\binom{m+1}{i+1}t^i,\quad f_{\Pi(1,\ldots,m)}(t)=\sum_{i=0}^{m-1}(m-i)!\,S(m,m-i)\,t^i,\\
\label{eq-h-simp-perm}h_{\Delta_m}(t)=\sum_{i=0}^mt^i,\quad h_{\Pi(1,\ldots,m)}(t)=A_m(t),
\end{gather}
where $S(m,m-i)$ denotes a Stirling number of the second kind,
for which $(m-i)!\,S(m,m-i)$ is the number of chains $(\varnothing\subsetneq A_1\subsetneq \cdots\subsetneq A_{m-i-1}\subsetneq[m])$ in $\mathcal{B}_m$.
Note that~\eqref{eq-f-simp-perm} and~\eqref{eq-h-simp-perm} can be obtained from Examples~\ref{ex-Pm1} and~\ref{ex-Fmn},
in which the faces of~$\Delta_m$ and~$\Pi(1,\ldots,m)$,
respectively, are characterized.

\begin{theorem}\label{thm-h-vector-bijective}
The $h$-polynomial of $\cP(m,n)$ is 
\begin{equation}\label{eq-h-poly-gen}h_{\cP(m,n)}(t)=1+\sum_{i=0}^{n-1}\;\sum_{j=1}^{m-i}\binom{m}{i}\,A_i(t)\:t^j.\end{equation}
This is equivalent to the recurrence relation
\begin{equation}\label{eq-h-poly-rec}h_{\cP(m,n+1)}(t)=h_{\cP(m,n)}(t)+\binom{m}{n}\,A_n(t)\sum_{i=1}^{m-n}t^i,\end{equation}
together with the initial condition 
\begin{equation}\label{eq-h-poly-init}h_{\cP(m,1)}(t)=\sum_{i=0}^mt^i.\end{equation}
\end{theorem}

\begin{proof}
The equivalence of~\eqref{eq-h-poly-gen} to~\eqref{eq-h-poly-rec} and~\eqref{eq-h-poly-init} 
can easily be checked.  So, the theorem will be proved by confirming~\eqref{eq-h-poly-rec} and~\eqref{eq-h-poly-init}. 

The initial condition~\eqref{eq-h-poly-init} is immediate since,
as seen in Example~\ref{ex-Pm1}, we have $\cP(m,1)=\Delta_m$,
and $h_{\Delta_m}(t)$ is then given by~\eqref{eq-h-simp-perm}.
(Note that an alternative initial condition would be $h_{\cP(m,0)}(t)=1$, since this
is equivalent to $h_{\cP(m,1)}(t)=\sum_{i=0}^mt^i$ using~\eqref{eq-h-poly-rec} with $n=0$,
and it is also consistent with an interpretation of $\cP(m,0)$ as consisting of a single point, i.e., the origin in $\bbR^m$.)

We now proceed to the recurrence relation~\eqref{eq-h-poly-rec}.
For $n\ge m$, we have $\sum_{i=1}^{m-n}t^i=0$, since this sum is empty,
and Corollary~\ref{cor:comb-iso} gives $h_{\cP(m,n+1)}(t)=h_{\cP(m,n)}(t)$,
so that~\eqref{eq-h-poly-rec} follows immediately.  Hence, in the rest of
this proof, it will be assumed that $n<m$, and~\eqref{eq-h-poly-rec} will be verified for that case.

First, use~\eqref{eq-h-simp-perm} to
rewrite \eqref{eq-h-poly-rec} as $h_{\cP(m,n+1)}(t)=h_{\cP(m,n)}(t)+t\,\binom{m}{n}\,h_{\Pi(1,\ldots,n)}(t)\,h_{\Delta_{m-n-1}}(t)$,
which is equivalent to
\begin{equation}\label{eq-f-polynomial}
f_{\cP(m,n+1)}(t)=f_{\cP(m,n)}(t)+(t+1)\,\binom{m}{n}\,f_{\Pi(1,\ldots,n)}(t)\,f_{\Delta_{m-n-1}}(t).\end{equation}
Taking coefficients of $t^k$ on both sides of~\eqref{eq-f-polynomial},
for $k=0,\ldots,m$, gives
\begin{multline}\label{eq-f-numbers}
f_k\bigl(\cP(m,n+1)\bigr)=f_k\bigl(\cP(m,n)\bigr)+\binom{m}{n}\sum_{\substack{i,j\ge0\\i+j=k-1}}f_i\bigl(\Pi(1,\ldots,n)\bigr)f_j(\Delta_{m-n-1})\\
+\binom{m}{n}\sum_{\substack{i,j\ge0\\i+j=k}}f_i\bigl(\Pi(1,\ldots,n)\bigr)f_j(\Delta_{m-n-1}).\end{multline}
Let $\mathcal{C}(m,n)_k$ denote the set of chains in $\mathcal{C}(m,n)$ with $k$ missing ranks, and partition $\mathcal{C}(m,n+1)_k$ into the sets
\begin{itemize}
\item $\mathcal{C}(m,n)_k$,
\smallskip
\item $\overline{\mathcal{C}}(m,n)_k=\{(A_1\subsetneq\cdots\subsetneq A_\ell)\in\mathcal{C}(m,n+1)_k\setminus\mathcal{C}(m,n)_k\mid A_1\neq\varnothing\}$\\
$\phantom{\overline{\mathcal{C}}(m,n)_k}=\{(A_1\subsetneq\cdots\subsetneq A_\ell)\in\mathcal{C}(m,n+1)\mid |A_\ell|-\ell+1=k,\;A_1\neq\varnothing,\;|A_\ell|-|A_1|=n\}$,
\smallskip
\item $\widetilde{\mathcal{C}}(m,n)_k=\{(A_1\subsetneq\cdots\subsetneq A_\ell)\in \mathcal{C}(m,n+1)_k\setminus \mathcal{C}(m,n)_k\mid A_1=\varnothing\}$\\
$\phantom{\widetilde{\mathcal{C}}(m,n)_k}=\{(A_1\subsetneq\cdots\subsetneq A_\ell)\in\mathcal{C}(m,n+1)\mid |A_\ell|-\ell+1=k,\;A_1=\varnothing,\;|A_\ell|-|A_2|=n\}$,
\end{itemize}
where the expression in  Remark~\ref{rem-missingranks} for the number of missing ranks 
in a chain has been used.
By Theorem~\ref{thm-faces-chains}, we have $f_k\bigl(\cP(m,n+1)\bigr)=|\mathcal{C}(m,n+1)_k|$ and
$f_k\bigl(\cP(m,n)\bigr)=|\mathcal{C}(m,n)_k|$.
Hence, the required result~\eqref{eq-f-numbers} will follow by showing that
\begin{align}\label{eq-empty}|\overline{\mathcal{C}}(m,n)_k|&=\binom{m}{n}\sum_{\substack{i,j\ge0\\i+j=k-1}}f_i\bigl(\Pi(1,\ldots,n)\bigr)\binom{m-n}{j+1}\\
\intertext{and}
\label{eq-not-empty}|\widetilde{\mathcal{C}}(m,n)_k|&=\binom{m}{n}\sum_{\substack{i,j\ge0\\i+j=k}}f_i\bigl(\Pi(1,\ldots,n)\bigr)\binom{m-n}{j+1},\end{align}
where $f_j(\Delta_{m-n-1})=\binom{m-n}{j+1}$ has been applied.  We will also use the fact that,
for any $S\subseteq [m]$ of size~$n$, $f_i\bigl(\Pi(1,\ldots,n)\bigr)$ is the number of chains in~$\mathcal{B}_m$, with~$i$ missing ranks,
of the form $(\varnothing\subsetneq A_1\subsetneq\cdots\subsetneq A_{n-i-1}\subsetneq S)$.

We obtain~\eqref{eq-empty} by constructing the chains in $\overline{\mathcal{C}}(m,n)_k$ using the following procedure.
First, choose $S\subseteq [m]$ of size $n$.
Next, choose a chain $(\varnothing\subsetneq A_1\subsetneq\cdots\subsetneq A_{n-i-1}\subsetneq S)$
with $i$ missing ranks, for some $i\in\{0,\ldots,k-1\}$.
Last, choose $T\subseteq[m]\setminus S$ of size $j+1=k-i$, so that
\begin{equation*}\bigl(T\subsetneq(A_1\cup T)\subsetneq\cdots\subsetneq(A_{n-i-1}\cup T)\subsetneq(S\cup T)\bigr)\in\overline{\mathcal{C}}(m,n)_k,\end{equation*}
where the requirement $T\ne\varnothing$ follows from $i\le k-1$ 
and $|T|=k-i$, and it can be checked that every chain in
$\overline{\mathcal{C}}(m,n)_k$ is constructed uniquely.

Similarly, we obtain~\eqref{eq-not-empty} by constructing the chains in $\widetilde{\mathcal{C}}(m,n)_k$ as follows.
First, choose $S\subseteq[m]$ of size~$n$.
Next, choose a chain $(\varnothing\subsetneq A_1\subsetneq\cdots\subsetneq A_{n-i-1}\subsetneq S)$
with $i$ missing ranks, for some $i\in\{0,\ldots,k\}$.
Last, choose $T\subseteq[m]\setminus S$ of size $j+1=k-i+1$, so that
\begin{equation*}\bigl(\varnothing\subsetneq T\subsetneq(A_1\cup T)\subsetneq\cdots\subsetneq(A_{n-i-1}\cup T)\subsetneq(S\cup T)\bigr)\in\widetilde{\mathcal{C}}(m,n)_k,\end{equation*}
where the requirement $T\ne\varnothing$ follows from $i\le k$ 
and $|T|=k-i+1$, and it can be checked that every chain in
$\widetilde{\mathcal{C}}(m,n)_k$ is constructed uniquely.
\end{proof}

For $n\ge m$, we can obtain a simpler expression for $h_{\cP(m,n)}(t)$, as given in the following corollary to Theorem~\ref{thm-h-vector-bijective}.

\begin{corollary}\label{coroll-h-poly-1}
For $n\ge m$, the $h$-polynomial of $\cP(m,n)$ is
\begin{equation}\label{eq-h-poly-1}h_{\cP(m,n)}(t)=1+t\,\sum_{i=1}^m\binom{m}{i}A_i(t).\end{equation}
\end{corollary}
\begin{proof}
We first note that the Eulerian polynomials satisfy the identity 
\begin{equation}\label{eq-Eulerian-id}1+t\,\sum_{i=1}^m\binom{m}{i}A_i(t)=\sum_{i=0}^m\binom{m}{i}A_i(t)\,t^{m-i},\end{equation}
which is equivalent to an identity for the Eulerian numbers obtained in~\cite{ChuGraKnu}.
In particular, setting $a=j+1$, $b=m-j-1$ and $k=i$ in Theorem~1 of~\cite{ChuGraKnu}
(and using the symmetry~\eqref{eq-Eulerian-symm} for the Eulerian numbers)
gives $\sum_{i=1}^m\binom{m}{i}\,A(i,j)=\sum_{i=0}^m\binom{m}{i}\,A(i,i+j-m+1)$.
Multiplying both sides of this equation by $t^j$ and summing over $j$ then leads to~\eqref{eq-Eulerian-id}.

For $n\ge m$,~\eqref{eq-h-poly-gen} gives
\begin{equation*}h_{\cP(m,n)}(t)=h_{\cP(m,m)}(t)=1+\sum_{i=0}^{m-1}\:\sum_{j=1}^{m-i}\binom{m}{i}A_i(t)\,t^j.\end{equation*}
Extending the sum over $i$ to include $i=m$ (which is possible since the sum over $j$ is empty for $i=m$),
writing~$\sum_{j=1}^{m-i}t^j$ as $t(t^{m-i}-1)/(t-1)$ and using~\eqref{eq-Eulerian-id}, then gives
\begin{equation*}h_{\cP(m,n)}(t)=1+\frac{t}{t-1}\biggl(1+t\sum_{i=1}^m\binom{m}{i}A_i(t)-\sum_{i=0}^m\binom{m}{i}A_i(t)\biggr),\end{equation*}
which simplifies to~\eqref{eq-h-poly-1}, as required.
\end{proof}

Some remarks on Theorem~\ref{thm-h-vector-bijective} and Corollary~\ref{coroll-h-poly-1} are as follows.
\begin{remark}
As noted in \cite[Remark 3.22]{ferroni2022hilbert}, the expression in~\eqref{eq-h-poly-gen} 
for the $h$-polynomial of $\cP(m,n)$ resembles an expression in \cite[Theorem~3.21]{ferroni2022hilbert} for the Hilbert--Poincar\'e series of the augmented Chow ring of the uniform matroid $U_{n,m}$.  More specifically, if the upper bound in the sum over~$j$ in~\eqref{eq-h-poly-gen} is changed from $m-i$ to $n-i$, then this gives the expression in~\cite[Theorem 3.21]{ferroni2022hilbert} for
$U_{n,m}$.
\end{remark}

\begin{remark}
As discussed in Remark~\ref{rem-stell}, $\cP(m,n)$ with $n\ge m$
is combinatorially equivalent to the $m$-stellohedron.
The $h$-polynomial of the $m$-stellohedron is shown in~\cite[Eq.~(7)]{PRW} to
be the right hand side of~\eqref{eq-h-poly-1},
and hence Corollary~\ref{coroll-h-poly-1} can be regarded as a previously-known result.
Also, as noted in~\cite[Example~3.4]{ferroni2022hilbert}, the $h$-polynomial of the $m$-stellohedron coincides with the Hilbert--Poincar\'e series of the augmented Chow ring of the $m$-th Boolean matroid.
\end{remark}

\begin{remark}
Combining~\eqref{eq-h-poly-1} and~\eqref{eq-Eulerian-id} gives
\begin{equation}\label{eq-h-poly-2}h_{\cP(m,n)}(t)=\sum_{i=0}^m\binom{m}{i}A_i(t)\,t^{m-i},\end{equation}
for $n\ge m$.
Alternatively, this can be obtained from~\eqref{eq-h-poly-1} by applying the palindromicity property~\eqref{eq-h-palind} to $h_{\cP(m,n)}(t)$, and using the symmetry~\eqref{eq-Eulerian-symm} for Eulerian polynomials (which itself is the palindromicity property~\eqref{eq-h-palind} applied to $h_{\Pi(1,\ldots,m)}(t)$).
We also note that applying the palindromicity property~\eqref{eq-h-palind} to $h_{\cP(m,n)}(t)$,
as given by the expression in~\eqref{eq-h-poly-gen} for arbitrary~$m$ and~$n$,
and again using~\eqref{eq-Eulerian-symm}, simply returns the same expression.
\end{remark}

We end this section with some comments on a different type of formula for the $h$-polynomial of~$\cP(m,n)$ that can be obtained by applying a standard technique which has been used to give such formulae for the $h$-polynomials of the permutohedron and related polytopes.

For a nonzero vertex $\vv$ of $\cP(m,n)$ with $k$ nonzero entries, it follows from Proposition~\ref{p-vertices} 
that a permutation $\pi_\vv$ in $\mathfrak{S}_k$ is obtained (in one-line notation) by deleting each zero entry in $\vv$ and subtracting $n-k$ from each nonzero entry in~$\vv$.  Let $\mathrm{des}(\vv)$ 
and $\mathrm{des}(\vv^{-1})$ denote the numbers of descents in $\pi_\vv$ and $(\pi_\vv)^{-1}$, respectively.

Now partition the vertex set of~$\cP(m,n)$ into the sets $\{\mathbf{0}\}$, $V_1$ and $V_2$,
where $V_1$ is the set of vertices of~$\cP(m,n)$ which have~1 as an entry
(or equivalently have every element of $[n]$ as an entry), and~$V_2$ 
is the set of nonzero vertices of~$\cP(m,n)$ which do not have~1 as an entry.
Also, for any $\vv\in V_1$, let $\beta(\vv)$ be the number of~$0$'s to the right of the 
(necessarily unique)~1 in~$\vv$.

It can be shown that the $h$-polynomial of $\mathcal{P}(m,n)$ is 
\begin{equation}\label{eq-h-poly}
h_{\cP(m,n)}(t)=1+\Biggl(\sum_{\vv\in V_1}t^{1+\mathrm{des}(\vv^{-1})+\beta(\vv)}\Biggr)
+\Biggl(\sum_{\vv\in V_2}t^{1+\mathrm{des}(\vv^{-1})}\Biggr).\end{equation}
Note that~$V_2$ is in simple bijection with the set of nonzero vertices of~$\cP(m,n-1)$ (where an element of $V_2$ is mapped to a nonzero vertex of $\cP(m,n-1)$ by subtracting~1 from each nonzero entry),
that if $m=n$ and $\vv\in V_1$ then $\vv$ contains no~0's and $\beta({\vv})=0$, and that if $n>m$ then $V_1=\varnothing$.  It follows straightforwardly from these observations that 
the sum over~$V_2$ in~\eqref{eq-h-poly} can be replaced by $\sum_{\vv\in V_2}t^{1+\mathrm{des}(\vv)}$
for arbitrary $m$ and $n$, and that 
the sum over~$V_1$ in~\eqref{eq-h-poly} can be replaced by $\sum_{\vv\in V_1}t^{1+\mathrm{des}(\vv)}$
for $n\ge m$.  Hence, for $n\ge m$, we have
\begin{equation}h_{\cP(m,n)}(t)=1+\sum_{\vv}t^{1+\mathrm{des}(\vv)},\end{equation}
with the sum being over the set of all nonzero vertices of $\cP(m,n)$. This can be 
regarded as a previously-known result, since it matches an expression for 
the $h$-polynomial of the $m$-stellohedron obtained in~\cite[Eq.~(7)]{PRW}.

We now briefly sketch a proof of~\eqref{eq-h-poly} for $n\leq m$. 
Since, by Proposition~\ref{p-simple}, $\cP(m,n)$ is a simple polytope, we can use an 
orientation of its $1$-skeleton (i.e., graph), as described in \cite[Section~2.2]{PRW}.
Concretely, for a simple polytope $P\subseteq\RR^m$, and a generic vector $\boldsymbol{\lambda}\in\mathbb{R}^m$ 
which is not orthogonal to any edge of $P$,
consider an orientation of the edges $[\vv,\ww]$ of~$P$ given by
$$\vv\to\ww\quad\text{if and only if}\quad
\langle\boldsymbol{\lambda},\vv\rangle<\langle\boldsymbol{\lambda},\ww\rangle.$$
It is then known (see, for example,~\cite[Corollary~2.2]{PRW}) that 
the coefficient of $t^i$ in the $h$-polynomial $h_P(t)$ 
is the number of vertices of~$P$ with indegree $i$. For $\cP(m,n)$, we
choose such a vector $\boldsymbol{\lambda}\in\mathbb{R}^m$ with $\lambda_1>\ldots>\lambda_m>0$.
Then $\mathrm{indegree}(\mathbf{0})=0$, and the main part of the proof of~\eqref{eq-h-poly} involves showing that
\begin{equation*}
\mathrm{indegree}(\vv)=\begin{cases}
1+\des(\vv^{-1})+\beta(\vv),&\text{if }\vv\in V_1,\\
1+\des(\vv^{-1}),&\text{if }\vv\in V_2.\end{cases}\end{equation*}

\section{Volume of the partial permutohedron $\cP(m,n)$ with $n\geq m-1$} \label{sec:volppm}
Let $v(m,n)$ denote the normalized volume of $\cP(m,n)$, i.e.,
\begin{equation*}v(m,n)=\nvol(\cP(m,n)).\end{equation*}
Heuer and Striker~\cite[Figure 6]{HS} used SageMath to compute $v(m,n)$ for $m,n\leq 7$.  
In this section, we consider $v(m,n)$ for $n\geq m-1$.
In Theorem~\ref{thm:recursion for volume of P(m,n)}, we 
obtain a recursive formula for $v(m,n)$ with $n\ge m-1$, from which it follows
that $v(m,n)$ is a polynomial in $n$ of degree $m$.
In Theorem~\ref{thm:closedformulav(m,n)}, we then obtain closed formulae for~$v(m,n)$ with $n\geq m-1$. 
We end the section by providing, in~\eqref{eq-vol-drac-seq} and~\eqref{eq-vol-drac-seq2}, further expressions for~$v(m,n)$ with $n\geq m-1$, which are obtained by relating partial permutohedra to certain generalized permutohedra of~\cite{postnikov}.

\subsection{Recursive formula for $v(m,n)$ with $n\ge m-1$}
Our method for obtaining a recursive formula for $v(m,n)$ with $n\ge m-1$ will rely on the facet description of $\cP(m,n)$ given in Proposition~\ref{p-facetdesc}. Specifically, we subdivide $\cP(m,n)$ into a collection of pyramids (or cones). For each such pyramid, the base is a facet that does not contain the origin, and the apex is the origin. We then compute the volume of each pyramid, and add these using Lemma~\ref{lem:coning} to obtain the volume of~$\cP(m,n)$.
The same approach was used to obtain Theorem~\ref{thm:recursion for volume of P(m,n)} for the case $n=m-1$ in~\cite[Theorem~4.1]{AW}  and~\cite[Part~(d)]{AMM_solution}.  (See also Remark~\ref{rem-parkingfunctionpolytope} for further information.)

It is well-known that the relative volume of the regular permutohedron $\Pi(1,\dots,m)$ is $m^{m-2}$ ~\cite[Example~3.1]{stanley}. 
(See Section~\ref{sec:nonfulldimvolume} for further details on the volume of non-full-dimensional polytopes in $\mathbb{R}^m$.)

We now present our main result of the section. 
\begin{theorem}\label{thm:recursion for volume of P(m,n)}
For any $m$ and $n$ with $n\geq m-1$,  the normalized volume of $\cP(m,n)$ is given recursively by
\begin{equation}
\label{eq:recursion1}
v(m,n) = (m-1)!\,\sum_{k=1}^m k^{k-2}\,\frac{v(m-k,n-k)}{(m-k)!}\left (kn-\binom{k}{2}\right)\binom{m}{k},
\end{equation}
with initial condition $v(0,n)=1$.
Furthermore, for fixed $m$, $v(m,n)$ with $n\ge m-1$ is given by a polynomial in~$n$ of degree $m$.
\end{theorem}
  
\begin{proof} 
We first consider the case $n\ge m$, 
and proceed by subdividing $\cP(m,n)$ into pyramids 
whose apex is the origin, and whose bases are the facets of $\cP(m,n)$ not containing the origin. 
Using Proposition~\ref{p-facetdesc}, there is a facet not containing the origin for each nonempty subset~$S$ of~$[m]$. Hence, for each $k\in[m]$, there are $\binom{m}{k}$ pyramids to consider with $|S|=k$.

For such a subset~$S$ of size $k$, the associated facet is, by~\eqref{facetdesc},
$$\textstyle F_S=\bigl\{\mathbf{x}\in \cP(m,n)\,\bigm|\, \sum_{i\in S}x_i=\binom{n+1}{2}-\binom{n+1-k}{2}\bigr\}=\bigl\{ x \in \cP(m,n)\,\bigm|\, \sum_{i\in S}x_i= k\,n-\binom{k}{2}\bigr\},$$
and the associated pyramid is $P_S=\text{Pyr}(F_S,\mathbf{0})$.
Since~$F_S$ lies in the hyperplane 
\begin{equation*}\mathcal{A}=\biggl\{\mathbf{x}\in\RR^m\,\bigg|\,\sum_{i\in S}x_i= k\,n-\binom{k}{2}\biggr\}
=\biggl\{\mathbf{x}\in\RR^m\,\bigg|\,\langle\mathbf{a},\mathbf{x}\rangle= k\,n-\binom{k}{2}\biggr\},
\end{equation*}
where $\mathbf{a}\in\RR^m$ has $k$ entries of~1 in positions~$S$ and entries of~0 elsewhere,
it follows that the lattice distance from~$\mathcal{A}$ to~$\mathbf{0}$ is 
$k\,n-\binom{k}{2}$.

Now observe that, using the notation of Theorem~\ref{thm-faces-chains}, $F_S=\mathcal{F}_{([m]\setminus S\,\subsetneq\,[m])}$.  Hence, using Proposition~\ref{prop-facevertices}, the vertices of $F_S$ 
are the vectors in $\RR^m$ for which:
\renewcommand{\labelenumi}{(\roman{enumi})}
\begin{enumerate}
\item The entries in positions~$S$ are $n$, $n-1$, \ldots, $n-k+1$, in any order.
\item The entries in positions~$[m]\setminus S$ are 
\begin{equation*}n-k,\:n-k-1,\:\ldots,\:n-k-l+1,\:\underbrace{0,\:\ldots,\:0\!}_{m-k-l}\,,\end{equation*}
in any order, and for any $0\le l\le m-k$.
\end{enumerate}

For a nonempty subset $T=\{t_1,\ldots,t_p\}$ of $[m]$ with $t_1<\cdots<t_p$, let $\Phi_T\colon\RR^m\to\RR^p$ be the projection given
by $(x_1,\ldots,x_m)\mapsto(x_{t_1},\ldots,x_{t_p})$.  Then it follows from condition~(i) for the vertices of $F_S$ that $\Phi_S(F_S)=\Pi(n,\dots,n-k+1)$,
and from condition~(ii) (and using Proposition~\ref{p-vertices}) that for $S\ne[m]$, $\Phi_{[m]\setminus S}(F_S)=\cP(m-k,n-k)$.
Furthermore, $\Pi(n,\dots,n-k+1)$ is simply the translation of $\Pi(1,\dots,k)$ by $(n-k,\ldots,n-k)$.
It now follows (using the unimodularity of the underlying transformations) that $F_S$ is congruent to the Cartesian product
$\Pi(1,\dots,k)\times\cP(m-k,n-k)$ for $k\in[m-1]$, and to $\Pi(1,\dots,m)$ for $k=m$ (i.e., $S=[m]$).  This process is illustrated in Example~\ref{ExCong}.

Thus, using $\relvol(\Pi(1,\dots, k)) = k^{k-2}$ and setting $v(0,n)=1$, the relative volume of the base 
of~$P_S$ is  
$$k^{k-2}\,\frac{v(m-k,n-k)}{(m-k)!},$$
and so, using~\eqref{eq-pyr-vol} for the volume of a pyramid, the volume of~$P_S$ is 
$$k^{k-2}\,\frac{v(m-k,n-k)}{(m-k)!}\cdot\left(kn-\binom{k}{2}\right)
\cdot\frac{1}{m}.$$

To obtain~\eqref{eq:recursion1}, we sum over all pyramids using Lemma~\ref{lem:coning}
(i.e., we multiply the previous expression by $\binom{m}{k}$ and sum over $k\in[m]$), 
and normalize the overall volume by multiplying by~$m!$.  

For the case~$n=m-1$, the only modification required to the previous argument is as follows.
By~\eqref{facetdesc}, there is now a facet $F_S$ not containing the origin for each nonempty subset~$S$ of~$[m]$
of size~$k$, for all $k\in[m-2]$ and for $k=m$, and so the sum over $k\in[m]$ in~\eqref{eq:recursion1} should be replaced by a sum over $k\in[m-2]\cup\{m\}$.  However, the sum over 
$k\in[m]$ can in fact be retained, since evaluating $v(m,m-1)$ for $m\ge2$ then gives a term which includes $v(1,0)$ (for $k=m-1$), but this term vanishes as evaluating $v(1,0)$ involves a sum over $k=1$ only,
for which the factor $k\,n-\binom{k}{2}$ is $1\cdot0-\binom{1}{2}=0$.

Finally, it can easily be seen, using induction on $m$, that~\eqref{eq:recursion1} and the initial condition $v(0,n)=1$ imply that, for fixed~$m$, $v(m,n)$ with $n\ge m-1$ is given by a polynomial in~$n$ of degree~$m$.
\end{proof}

\begin{example}\label{ExCong}
An example of the congruence used in the proof of Theorem~\ref{thm:recursion for volume of P(m,n)} between a facet~$F_S$ and $\Pi(1,\dots,k)\times\cP(m-k,n-k)$, for a nonempty subset~$S$ 
of $[m]$ with $k=|S|$, is as follows.  Consider the case with $m=4$, $n=6$, $k=2$ and $S=\{1,2\}$.  Then $F_{\{1,2\}}=\{\mathbf{x}\in\cP(4,6)\mid x_1+x_2=11\}$, and 
for a vertex~$\mathbf{v}$ of this facet,~$v_1$ and~$v_2$ are~$6$ and~$5$ in either order (due to condition~(i)), while~$v_3$ and~$v_4$ are~$4$ and~$3$ in either order,~$4$ and~$0$ in either order, or both~$0$
(due to condition~(ii)).  Hence, the vertices of $F_{\{1,2\}}$ are $(6,5,4,3)$, $(6,5,3,4)$, $(6,5,4,0)$, $(6,5,0,4)$, $(6,5,0,0)$, $(5,6,4,3)$, $(5,6,3,4)$, $(5,6,4,0)$, $(5,6,0,4)$ and $(5,6,0,0)$.
Also, $\Phi_{\{1,2\}}(F_{\{1,2\}})=\Pi(6,5)$, which is the line segment with vertices $(6,5)$ and $(5,6)$, and $\Phi_{\{3,4\}}(F_{\{1,2\}})=\cP(2,4)$, which is the pentagon with vertices
 $(4,3)$, $(3,4)$, $(4,0)$, $(0,4)$ and $(0,0)$.  It follows that $F_{\{1,2\}}=\Pi(6,5)\times\cP(2,4)$, which (by considering $\Pi(6,5)$ as $\Pi(1,2)$ translated by $(4,4)$) is 
 congruent to $\Pi(1,2)\times\cP(2,4)$.
\end{example}

\begin{remark}\label{rem-arg}
If we attempt to modify the proof of Theorem~\ref{thm:recursion for volume of P(m,n)} 
for the case $n<m-1$, then by~\eqref{facetdesc}, there is now a facet $F_S$ not containing the origin for each nonempty subset~$S$ of~$[m]$ of size~$k$, for all $k\in[n-1]$ and for $k=m$.
Hence, the sum over $k\in[m]$ in~\eqref{eq:recursion1} should be replaced by a sum over $k\in[n-1]\cup\{m\}$. However, certain further modifications are required, since~$F_S$ for $k=m$ 
is $F_{[m]}=\mathcal{F}_{(\varnothing\,\subsetneq\,[m])}=\Pi(0,\ldots,0,1,\ldots,n)$
(using case~(4) in Proposition~\ref{prop-facevertices}), where the number of $0$'s is $m-n\ge2$, 
and we do not have a simple closed formula for the relative volume of $\Pi(0,\ldots,0,1,\ldots,n)$.
Also, the lattice distance from the origin to the hyperplane containing $F_{[m]}$ is now $\binom{n+1}{2}$.
\end{remark}
 
\begin{example}\label{ex:formulas}
For $1\le m\le 7$ and $n\ge m-1$, Theorem~\ref{thm:recursion for volume of P(m,n)} gives the following expressions for $v(m,n)$:
\begin{equation*}{\small
\begin{array}{@{}l@{\;\,}r@{\;\,}r@{\;\,}r@{\;\,}r@{\;\,}r@{\;\,}r@{\;\,}r@{\;\,}r@{}}
v(1,n)=&&n,&&&&&&\\
v(2,n)=&-1&&+2n^2,&&&&&\\
v(3,n)=&-6 &-9n&&+6n^3,&&&&\\
v(4,n)=&-54 &-96n&-72n^2&&+24n^4,&&&\\
v(5,n)=&-840&-1350n &-1200n^2&-600n^3&&+120n^5,&\\
v(6,n)=&-21150 &-30240n&-24300n^2&-14400n^3&-5400n^4&&+720n^6,\\
v(7,n)=&- 782460&- 1036350n&- 740880n^2& - 396900n^3& - 176400n^4& - 52920n^5& &+5040n^7.
\end{array}}\end{equation*}
Note that the expression $v(2,n)=2n^2-1$ was obtained in~\cite[Theorem~5.29]{HS}.
\end{example}

\subsection{Closed formulae for $v(m,n)$ with $n\ge m-1$}\label{sec:closedformv(m,n)}
We now obtain closed formulae for $v(m,n)$ with $n\ge m-1$,
which thereby provide explicit expressions for the coefficients of $v(m,n)$ as a polynomial in~$n$.
Our approach will rely on the recursion of Theorem~\ref{thm:recursion for volume of P(m,n)}, together with methods used for recent computations of the volume of~$\cP(m,m-1)$ in~\cite[Proposition~4.2]{AW} and~\cite[Part~(d)]{AMM_solution}. (See also Remark~\ref{rem-parkingfunctionpolytope} for further information.)

In this section, for any power series $f(z)$, we use the notation $[z^i]f(z)$ to denote the coefficient of~$z^i$ in the expansion of $f(z)$.
We also use the double factorial, defined as $(2i-3)!!=-\prod_{j=1}^i(2j-3)$, for any nonnegative integer $i$.

\begin{theorem}\label{thm:closedformulav(m,n)}
For any $m$ and $n$ with $n\geq m-1$, the normalized volume of $\cP(m,n)$ is given by
\begin{align}\label{eq:vmn}
v(m,n)&=-\frac{m!}{2^m}\sum_{0\le i\le j\le m}\binom{m}{j}\,\binom{j}{i}\,(2i-3)!!\,(2n)^{m-j},\\
\label{eq:vmn2}v(m,n)&=-\frac{m!}{2^m}\,\sum_{i=0}^m\,\binom{m}{i}(2i-3)!!\,(2n+1)^{m-i}\\
\intertext{and}
\label{eq:vmncoeff}v(m,n)&=(m!)^2\,[z^m]\,\sqrt{1-z}\,e^{(n+1/2)z}.
\end{align}
\end{theorem}

\begin{proof}
The right hand sides of~\eqref{eq:vmn} and~\eqref{eq:vmn2} are equal, since $\binom{m}{j}\binom{j}{i}=\binom{m}{i}\binom{m-i}{m-j}$
and $\sum_{j=i}^m\binom{m-i}{m-j}\times\allowbreak(2n)^{m-j}=\sum_{j=0}^{m-i}\binom{m-i}{j}(2n)^j=(2n+1)^{m-i}$.
The right hand sides of~\eqref{eq:vmn2} and~\eqref{eq:vmncoeff} are equal, since
$[z^i]\,\sqrt{1-z}=-(2i-3)!!/(2^i\,i!)$ and $[z^i]\,e^{(n+1/2)z}=(2n+1)^i/(2^i\,i!)$.

In the remainder of this proof, we confirm the validity of~\eqref{eq:vmncoeff}.  This will involve using the well-studied function
\begin{equation}\label{Tdef}T(z)=\sum_{i=1}^\infty i^{i-1}\frac{z^i}{i!}.\end{equation}
For information on $T(z)$, see for example~\cite[pp.~331, 332 and~338]{CorGonHarJefKnu96} or~\cite[pp.~23--28 and~43]{StanleyEC2}.
The properties of $T(z)$ which will be used here are
\begin{equation}\label{Tdiff}T'(z)=\frac{T(z)}{z(1-T(z))},\end{equation}
where $T'(z)$ is the derivative of $T(z)$
(see~\cite[Eq.~(3.2)]{CorGonHarJefKnu96}, with $T(z)=-W(-z)$ from~\cite[p.~331]{CorGonHarJefKnu96}), and 
\begin{equation}\label{Tprop2}
[z^k]\,f\bigl(T(z)\bigr)=[z^k]\,f(z)\,(1-z)\,e^{kz},
\end{equation}
for any power series $f(z)$ and nonnegative integer $k$
(see~\cite[Eq.~(2.38)]{CorGonHarJefKnu96}).

We now denote the unnormalized volume of $\cP(m,n)$ as $V(m,n)$, i.e.,
\begin{equation}\label{Vunnorm}V(m,n)=v(m,n)/m!,\end{equation}
set $V(0,n)=1$ (for consistency with the condition $v(0,n)=1$ in Theorem~\ref{thm:recursion for volume of P(m,n)}),
and introduce the exponential generating function
\begin{equation}\label{Vsdef}V_s(z)=\sum_{i=0}^\infty V(i,i+s)\,\frac{z^i}{i!},\end{equation}
for any integer $s\ge-1$. 

The recursion~\eqref{eq:recursion1} can be written as
\begin{equation*}V(m,n)=\sum_{k=1}^m\binom{m}{k}\,k^{k-1}\,V(m-k,n-k)\,\frac{2n-k+1}{2m},\end{equation*}
which gives
\begin{equation*}V_s(z)=1+\sum_{i=1}^\infty\sum_{k=1}^i\binom{i}{k}\,k^{k-1}\,V(i-k,i-k+s)\,\frac{2i-k+2s+1}{2i}\,\frac{z^i}{i!}.\end{equation*}
Differentiating then gives
\begin{equation*}V'_s(z)=\sum_{i=1}^\infty\sum_{k=1}^i\binom{i}{k}\,k^{k-1}\,V(i-k,i-k+s)\,\frac{2i-k+2s+1}{2}\,\frac{z^{i-1}}{i!}.\end{equation*}
Separating $(2i-k+2s+1)/2$ as $(i-k)+(k/2)+(s+1/2)$, and using~\eqref{Tdef} and~\eqref{Vsdef}, we obtain 
\begin{equation*}V'_s(z)=T(z)\,V'_s(z)+T'(z)\,V_s(z)/2+(s+1/2)\,T(z)\,V_s(z)/z,\end{equation*}
so that 
\begin{equation*}V'_s(z)=\biggl(\frac{T'(z)}{2(1-T(z))}+\frac{(s+1/2)\,T(z)}{z(1-T(z))}\biggr)V_s(z),\end{equation*}
 which, using~\eqref{Tdiff}, becomes
\begin{equation*}V'_s(z)=\Bigl(\frac{1}{2(1-T(z))}+s+1/2\Bigr)\,T'(z)\,V_s(z).\end{equation*}
Solving this first order homogeneous linear differential equation for $V_s(z)$, with the initial condition $V_s(0)=1$
(and noting that $T(0)=0$, from~\eqref{Tdef}), we obtain  
\begin{align*}V_s(z)&=e^{-\log(1-T(z))/2\,+\,(s+1/2)T(z)}\\
&=\frac{e^{(s+1/2)\,T(z)}}{\sqrt{1-T(z)}}.\end{align*}
Hence, using~\eqref{Vunnorm} and~\eqref{Vsdef}, and setting $s=n-m$, gives
\begin{equation*}\label{eq-gen-fun}v(m,n)=(m!)^2\,[z^m]\,\frac{e^{(n-m+1/2)\,T(z)}}{\sqrt{1-T(z)}},\end{equation*}
for $n\ge m-1$.  

We now obtain~\eqref{eq:vmncoeff} by applying~\eqref{Tprop2} with $f(z)=e^{(n-m+1/2)\,z}/\sqrt{1-z}$, thereby completing the proof.
\end{proof}

Some remarks on Theorem~\ref{thm:closedformulav(m,n)} are as follows.

\begin{remark}
A conjectural generalization of Theorem~\ref{thm:closedformulav(m,n)},
which provides an explicit expression for the Ehrhart polynomial of~$\cP(m,n)$ with $n\ge m-1$, will be given in Conjecture~\ref{conj-Ehr-expl}.
\end{remark}

\begin{remark}\label{rem-vol-rec}
The function $g(z)=\sqrt{1-z}\,e^{(n+1/2)z}$,
which appears in~\eqref{eq:vmncoeff}, satisfies
$(z-1)g'(z) =((n+1/2)z-n)\,g(z)$.  Using this and~\eqref{eq:vmncoeff}, it follows
that the unnormalized volume~\eqref{Vunnorm} of $\cP(m,n)$ with $n\ge m-1$
satisfies the recurrence relation
\begin{equation}\label{eq-vol-rec}V(m,n)=(m+n-1)\,V(m-1,n)-(m-1)(n+1/2)\,V(m-2,n).\end{equation}
By setting~$V(0,n)=1$ and setting $V(-1,n)$ arbitrarily (since $V(-1,n)$ has coefficient zero in the 
$m=1$ case of~\eqref{eq-vol-rec}), it follows that~\eqref{eq-vol-rec} with $m\ge1$ is equivalent (using~\eqref{Vunnorm}) to each of the equations in Theorem~\ref{thm:closedformulav(m,n)}.
It would be interesting to find a geometric proof of~\eqref{eq-vol-rec},
since this would provide a more direct combinatorial proof of
Theorem~\ref{thm:closedformulav(m,n)}.
\end{remark}

\begin{remark}\label{rem-winpolytope} 
An alternative perspective on $\cP(m,m-1)$, and its volume, is as follows.  For a graph~$G$ with vertex set~$[m]$, a \emph{partial orientation}~$\mathcal{O}$ of~$G$ is an assignment of a direction to some (which could be all or none) of the edges of~$G$,
and the \emph{win vector} of~$\mathcal{O}$ is the indegree sequence of~$\mathcal{O}$ (i.e., for each $i\in[m]$, the $i$th entry of the win vector is the number of edges incident to~$i$ which are directed towards~$i$ by~$\mathcal{O}$). The \emph{win vector polytope}~$W(G)$ of~$G$, as defined by Bartels, Mount and Welsh~\cite{bartels1997polytope}, is the convex hull in~$\RR^m$ of the win vectors of all partial orientations of~$G$.
It can be shown that, for any $m$, the win vector polytope $W(K_m)$ of the complete graph~$K_m$ is exactly~$\cP(m,m-1)$. (For example, this can be done using a characterization of the vertices of win vector polytopes given in~\cite[Proposition~3.4]{bartels1997polytope}.  
However, note that the characterization of the win vectors of~$K_m$ given in~\cite[Example~3]{bartels1997polytope} appears to contain some errors.) Expressions for the volume and Ehrhart polynomial of~$W(G)$ for arbitrary~$G$ are obtained by Backman in~\cite[Theorem~4.5 and Corollary~4.6]{backman2018partial}, as sums over cycle-path minimal partial orientations
of $3G$ (where this denotes the graph obtained from~$G$ by replacing each edge by three parallel copies).  Hence, taking $G=K_m$ in these expressions provides alternative means for computing the volume and Ehrhart polynomial of~$\cP(m,m-1)$.
\end{remark}

\begin{remark}\label{rem-parkingfunctionpolytope}
Yet another perspective on $\cP(m,m-1)$, and its volume, is as follows. A \emph{parking function} of length~$m$ is an $m$-vector of positive integers whose nondecreasing rearrangement $r_1\le r_2\le\ldots\le r_m$ satisfies $r_i\le i$ for each $i\in[m]$.
The \emph{parking function polytope}~$P_m$, as defined by Stanley in~\cite{AMM_problem}, is the convex hull in~$\RR^m$ of all parking functions of length~$m$.
Stanley asked for enumerations of the vertices and facets of $P_m$~\cite[Parts~(a) and~(b)]{AMM_problem}, and
explicit characterizations of these vertices and facets were subsequently given in~\cite[Section~1]{AW} and~\cite[Parts~(a) and~(b)]{AMM_solution}. By comparing these characterizations with those
of Propositions~\ref{p-vertices} and~\ref{p-facetdesc} for $n=m-1$, it can immediately be seen that~$P_m$ (for $m\ge2$) is simply~$\cP(m,m-1)$ translated by $(1,\ldots,1)$.  Stanley also asked for the volume of~$P_m$~\cite[Part~(d)]{AMM_problem}.  Recursive formulae for
this volume, which match the recursive formula of~\eqref{eq:recursion1} for $n=m-1$, are obtained in~\cite[Theorem~4.1]{AW} and~\cite[Part~(d)]{AMM_solution}, by following the same approach as that used to prove Theorem~\ref{thm:recursion for volume of P(m,n)}. Using this recursion, a more explicit formula for the volume of~$P_m$ is obtained in~\cite[Proposition~4.2]{AW}, and a completely explicit formula for this volume is obtained in~\cite[Part~(d), first equation]{AMM_solution}.  The approach used in the proof of Theorem~\ref{thm:closedformulav(m,n)} is closely related to those used in~\cite[Proposition~4.2]{AW} and~\cite[Part~(d)]{AMM_solution}.  However, note that the explicit formula for $v(m,m-1)$
provided by~\cite[Part~(d), first equation]{AMM_solution} has a different form from that of the $n=m-1$ cases
of~\eqref{eq:vmn} or~\eqref{eq:vmn2} in Theorem~\ref{thm:closedformulav(m,n)}.  Also, two further forms 
of explicit formulae for $v(m,m-1)$, based on results of~\cite{Shevelev},
are given in~\cite[Theorem~2.9~(v) and (vi)]{HanadaLentferVindasMelendez}.
\end{remark}

\begin{remark}\label{rem-parkingfunctionpolytopegen}
The parking function polytope~$P_m$ discussed in Remark~\ref{rem-parkingfunctionpolytope} was recently generalized in~\cite{HanadaLentferVindasMelendez} to a wider class of  related polytopes which includes~$\cP(m,n)$ (up to a simple translation) for any $n\ge m-1$.  More specifically, for positive integers~$a$,~$b$ and~$m$, let an \emph{$(a,b)$-parking function} 
of length~$m$ be an $m$-vector of positive integers whose nondecreasing rearrangement $r_1\le r_2\le\ldots\le r_m$ satisfies $r_i\le a+(i-1)b$ for each $i\in[m]$, and let the polytope~$\mathcal{X}_m(a,b)$ be
the convex hull in~$\RR^m$ of all $(a,b)$-parking functions of length~$m$.  Then $P_m=\mathcal{X}_m(1,1)$, and it follows from~\cite[Proposition~3.16]{HanadaLentferVindasMelendez} that, for any $n\ge m-1$, $\mathcal{X}_m(n-m+2,1)$ is simply~$\cP(m,n)$ translated by $(1,\ldots,1)$.  Theorem~\ref{thm:closedformulav(m,n)}
appeared as a conjecture in the first arXiv version of this paper, and an alternative proof of this conjecture was recently obtained independently in~\cite[Corollary~3.29]{HanadaLentferVindasMelendez}, as a corollary of a result for the volume of~$\mathcal{X}_m(a,b)$ with any~$a$ and~$b$~\cite[Theorem~1.1]{HanadaLentferVindasMelendez}.  
Some of the methods used to obtain~\cite[Theorem~1.1]{HanadaLentferVindasMelendez} are also related to methods used in~\cite[Section~4]{AW} and~\cite[Part~(d)]{AMM_solution}.
\end{remark}

\subsection{Partial permutohedra as generalized permutohedra}
\label{sec:gp}
We now relate partial permutohedra to cases of generalized permutohedra, which enables us to use results of Postnikov~\cite{postnikov} to obtain further expressions for~$v(m,n)$ with $n\geq m-1$.  

A \emph{generalized permutohedron}~\cite{postnikov} in~$\RR^m$ is a polytope in which every edge is parallel to $\mathbf{e}_i-\mathbf{e}_j$, for some $i,j\in[m]$.
Note that a certain relation between $\cP(m,m-1)$ (in the context of the parking function polytope~$P_m$ of Remark~\ref{rem-parkingfunctionpolytope}) and 
generalized permutohedra is identified in~\cite[Section~5]{AW}.
Note also that, up to translation, generalized permutohedra are polymatroid base polytopes.  See, for example,~\cite{black2022flag} for this perspective, and associated connections to partial permutohedra. 

Let $s_{m,n}=n+(n-1)+\dots+\max(n-m+1,0)=\binom{n+1}{2}-\binom{n+1-m}{2}$,
where $\binom{n+1-m}{2}$ is taken to be~0 for $n\le m$, and 
define the affine map
\begin{equation*}
\textstyle\phi_{m,n}\colon\RR^m\to\RR^{m+1},\quad(x_1\dots,x_m)\mapsto(x_1,\dots,x_m,s_{m,n}-\sum_{i=1}^mx_i),
\end{equation*}
and the polytope
\begin{equation}\label{liftedP}
\widetilde{\cP}(m,n)=\phi_{m,n}(\cP(m,n)).
\end{equation}
Then $\widetilde{\cP}(m,n)$ is affinely isomorphic to $\cP(m,n)$ and can be seen to be a  generalized permutohedron by using Theorem~\ref{thm-faces-chains} and Proposition~\ref{prop-facevertices} to characterize the edges of~$\cP(m,n)$.  (See also
Remark~\ref{rem-edge} for a characterization of these edges.)
Furthermore, $\cP(m,n)$ and $\widetilde{\cP}(m,n)$ have the same Ehrhart polynomial, and thus the same relative volume.  

We now focus on $\cP(m,n)$ and $\widetilde{\cP}(m,n)$ for the case $n\geq m-1$.
It can be shown that $\cP(m,n)$ with $n\geq m-1$ has the Minkowski sum decomposition
\begin{equation}\label{eq-Minkowski1}
\cP(m,n)=(n-m+1)\underbrace{\sum_{i=1}^m\textrm{ConvexHull}(\{\mathbf{e}_0,\mathbf{e}_i\})}_{=[0,1]^m}+\sum_{1\le i<j\le m}\textrm{ConvexHull}(\{\mathbf{e}_0,\mathbf{e}_i,\mathbf{e}_j\}).
\end{equation}
Note that for $n\geq m-1$, we have $\cP(m,n)=\widehat{\Pi}(n,n-1,\dots,n-m+1)$, 
using notation which will be introduced in Definition~\ref{def:Pi} and a result which
will be given in~\eqref{eq-anti-block}, and that the associated permutohedron 
$\Pi(n,n-1,\dots,n-m+1)$ has the well-known Minkowski sum decomposition
\begin{equation}
\Pi(n,n-1,\dots,n-m+1)=(n-m+1)\sum_{i=1}^m\{\mathbf{e}_i\}+\sum_{1\le i<j\le m}\textrm{ConvexHull}(\{\mathbf{e}_i,\mathbf{e}_j\}).\end{equation}
It follows from~\eqref{liftedP} and~\eqref{eq-Minkowski1} that $\widetilde{\cP}(m,n)$ with $n\geq m-1$ has the Minkowski sum decomposition
\begin{equation}\label{eq-Minkowski2}
\small{\widetilde{\cP}(m,n)
=(n-m+1)\underbrace{\sum_{i=1}^m\textrm{ConvexHull}(\{\mathbf{e}_i,\mathbf{e}_{m+1}\})}_{=\{(x_1,\ldots,x_m,m-\sum_{i=1}^mx_i)\,\mid\,\mathbf{x}\in[0,1]^m\}}+\sum_{1\le i<j\le m} \textrm{ConvexHull}(\{\mathbf{e}_i,\mathbf{e}_j,\mathbf{e}_{m+1}\}),}
\end{equation}
where $\mathbf{e}_i$ now denotes the $i$th standard unit vector in $\RR^{m+1}$.
Hence,~$\widetilde{\cP}(m,n)$ is a so-called type-$\cY$ generalized permutohedron,
i.e., it has the form $P^y_n(\{y_I\})$, as defined in~\cite[p.~1042]{postnikov}.
Various results for volumes and Ehrhart polynomials of type-$\cY$ generalized permutohedra are obtained in~\cite{postnikov},
and we now apply some of these to $\widetilde{\cP}(m,n)$.

By applying~\cite[Theorem~9.3]{postnikov} to~\eqref{eq-Minkowski2}, and simplifying,
it follows that the normalized volume of~$\cP(m,n)$ with $n\geq m-1$
is given by
\begin{equation}\label{eq-vol-drac-seq}
v(m,n)=\sum_{(a_1,\ldots,a_{m(m+1)/2})}\binom{m}{a_1,\ldots,a_{m(m+1)/2}}\,(n-m+1)^{a_1+\ldots+a_m},\end{equation}
where the sum is over all so-called draconian sequences for this case, with these defined as follows.  Let $I_i=\{i\}$ for $1\le i\le m$, and let $I_{m+1},\ldots,I_{m(m+1)/2}$ be the sets
$\{i,j\}$ for $1\le i<j\le m$ (in any fixed order). Then the draconian sequences
in~\eqref{eq-vol-drac-seq} are those sequences $(a_1,\ldots,a_{m(m+1)/2})$ of nonnegative integers such that $\sum_{k=1}^{m(m+1)/2}a_k=m$ and $\sum_{k\in S}a_k\le|\cup_{k\in S}I_k|$, for all $\varnothing\subsetneq S\subsetneq[m(m+1)/2]$.
Note that taking $S$ to be a singleton in $\sum_{k\in S}a_k\le|\cup_{k\in S}I_k|$ gives 
$a_i\le 1$ for $i=1,\ldots,m$, and $a_i\le2$ for $i=m+1,\ldots,m(m+1)/2$.

It follows from~\eqref{eq-vol-drac-seq} that, for any fixed $m$,
$v(m,n)$ with $n\ge m-1$ is given by a polynomial in $n-m+1$ with positive integer
coefficients and (since $m$ $1$'s followed by $m(m-1)/2$ $0$'s is a draconian
sequence) degree $m$.

\begin{example}\label{ex-vol-drac}
For $m=2$, we have $I_1=\{1\}$, $I_2=\{2\}$ and $I_3=\{1,2\}$, and the draconian sequences in~\eqref{eq-vol-drac-seq} are $(0,0,2)$, $(0,1,1)$, $(1,0,1)$ and $(1,1,0)$, 
which gives $v(2,n)=1+2(n-1)+2(n-1)+2(n-1)^2$. 
\end{example}

\begin{example}\label{ex:formulas_change_variable}
For $1\le m\le6$ and $n\ge m-1$, the expressions for $v(m,n)$
in terms of $N=n-m+1$ are:
\begin{equation*}{\small
\begin{array}{@{}l@{\;\,}r@{\;\,}r@{\;\,}r@{\;\,}r@{\;\,}r@{\;\,}r@{\;\,}r@{\;\,}r@{}}
v(1,n)=&&N,&&&&&&\\
v(2,n)=&1&+4N&+2N^2,&&&&&\\
v(3,n)=&24 &+63N&+36N^2&+6N^3,&&&&\\
v(4,n)=&954 &+2064N&+1224N^2&+288N^3&+24N^4,&&&\\
v(5,n)=&59040&+113850N &+68400N^2&+18600N^3&+2400N^4&+120N^5,&\\
v(6,n)=&5295150&+9446760N&+5699700N^2&+1677600N^3&+264600N^4&+21600N^5&+720N^6.
\end{array}}\end{equation*}
\end{example}

By applying~\cite[Theorem~10.1]{postnikov} to~\eqref{eq-Minkowski2}, 
and simplifying, it follows that
the normalized volume of~$\cP(m,n)$ with $n\geq m-1$
is also given by
\begin{align}
\notag&v(m,n)=\\
&\label{eq-vol-drac-seq2}\sum_{\sigma\in\mathfrak{S}_{m+1}}\!\!\!
\frac{\displaystyle
\biggl(\,\sum_{i=1}^{\sigma^{-1}(m+1)-1}\!\!(n-i+1)\lambda_{\sigma(i)}+
{\textstyle\frac{1}{2}}(m-\sigma^{-1}(m+1)+1)(2n-m-\sigma^{-1}(m+1)+2)\lambda_{m+1}\!\biggr)^m}{
\prod_{i=1}^m(\lambda_{\sigma(i)}-\lambda_{\sigma(i+1)})},
\end{align}
for any distinct parameters $\lambda_1,\ldots,\lambda_{m+1}$.

\begin{remark}
It is shown in~\cite[Theorem~2.9~(vii)]{HanadaLentferVindasMelendez},
using a result of~\cite{Shevelev}, that $v(m,m-1)$ is the number of 
$m\times m$ $(0,1)$-matrices for which there are exactly two~$1$'s in each row, and the permanent is positive.  
Since the permanent of an $m\times m$ $(0,1)$-matrix~$A$ is the number of perfect
matchings of the bipartite graph whose biadjacency matrix is~$A$, it follows
that~$v(m,m-1)$ is the number of bipartite graphs with $m$ vertices in each part,
such that each vertex in one part has degree~2, and there exists a perfect matching.
This result can also be obtained from~\eqref{eq-vol-drac-seq} as follows, which may be relevant to~\cite[Conjecture~5.2 and Problem~5.4]{HanadaLentferVindasMelendez}.
For $n=m-1$, the summand in~\eqref{eq-vol-drac-seq} is zero unless the draconian sequence $(a_1,\ldots,a_{m(m+1)/2})$ has $a_1=\dots=a_m=0$.  Hence, in this case,~\eqref{eq-vol-drac-seq} simplifies to
\begin{equation}\label{eq-vol-drac-seq3}v(m,m-1)=\sum_{(b_{12},b_{13}\ldots,b_{m-1,m})}\binom{m}{b_{12},\,b_{13},\,\ldots,\,b_{m-1,m}},\end{equation}
where the sum is over all
sequences $(b_{ij})_{1\le i<j\le m}$ of nonnegative integers such that $\sum_{1\le i<j\le m}b_{ij}=m$ and  $\sum_{(i,j)\in S}b_{ij}\le|\cup_{(i,j)\in S}\{i,j\}|$, for all $\varnothing\subsetneq S\subsetneq\{(i,j)\mid 1\le i\le j\le m\}$.
The desired interpretation of $v(m,m-1)$ can then be obtained using Hall's marriage theorem for the existence of a perfect matching in a bipartite graph.  
In particular,
a sequence $(b_{12},b_{13}\ldots,b_{m-1,m})$ in~\eqref{eq-vol-drac-seq3} is regarded as corresponding to $b_{ij}$ copies of $\{i,j\}$, for each $1\le i<j\le m$.  These~$m$ two-element sets are then permuted in
all $\binom{m}{b_{12},\,b_{13},\,\ldots,\,b_{m-1,m}}$ ways, and the $k$th set in such a permutation is taken to be the set of neighbours of vertex~$k'$ in a bipartite graph with vertices $1',\ldots,m',1,\ldots,m$, and which has a perfect matching 
due to the conditions on $(b_{12},b_{13}\ldots,b_{m-1,m})$.
\end{remark}

\section{The partial permutohedron $\cP(m,n)$ with $n\le4$}\label{sec:volppn}
We now shift our focus to $\cP(m,n)$ with arbitrary $m$ and fixed $n\le 4$.
Heuer and Striker conjectured that the normalized volume of $\cP(m,2)$ is $v(m,2)=3^m-m$ \cite[Conjecture  5.30]{HS}. In this section, we compute
explicit expressions for the Ehrhart polynomials of $\mathcal{P}(m,2)$ and $\mathcal{P}(m,3)$, and thereby obtain a proof of the conjecture for $v(m,2)$
and an expression for $v(m,3)$.  We then also obtain an explicit expression for 
$v(m,4)$.

\subsection{Ehrhart polynomials}
We begin by recalling some basic facts about Ehrhart polynomials.
For a lattice polytope $\mathcal{P}\subseteq \mathbb{R}^m$, the function $|t\,\mathcal{P}\cap\mathbb{Z}^m|$ 
of a positive integer variable~$t$ (i.e., the number of integer points in the 
$t$-th dilate $t\,\mathcal{P}=\{t\,\mathbf{x}\mid\mathbf{x}\in\cP\}$ of~$\cP$) is known to 
agree with a polynomial $\ehr(\mathcal{P},t)\in\mathbb{Q}[t]$ of degree $\dim(\mathcal{P})$, called the \emph{Ehrhart polynomial} of~$\cP$.
Furthermore, the coefficient of the leading term of $\ehr(\mathcal{P},t)$ is the relative volume of~$\cP$.

An immediate consequence of the definition is that $\ehr(n\cP,t)=\ehr(\cP,nt)$, for any positive integer~$n$.

For lattice polytopes $\cP_1$ and $\cP_2$, the Cartesian product $\cP_1\times\cP_2$ is
a lattice polytope, and we have $\ehr(\cP_1\times\cP_2,t)=\ehr(\cP_1,t)\ehr(\cP_2,t)$.
The Ehrhart polynomial also satisfies an inclusion-exclusion property~\cite[Section 5]{beck2018combinatorial},
as follows. For lattice polytopes $\cP_1$ and $\cP_2$ such that 
the polytope $\cP_1\cap\cP_2$ is a lattice polytope and 
$\cP_1\cup\cP_2$ is a polytope, we have the properties that $\cP_1\cup\cP_2$ 
is a lattice polytope, and 
\begin{equation}\label{eq:ie-general}
\ehr(\cP_1\cup\cP_2,t)=\ehr(\cP_1,t)+\ehr(\cP_2,t)-\ehr(\cP_1\cap\cP_2,t). 
\end{equation}
We will often apply \eqref{eq:ie-general} to the case in which $\cP_1\cap\cP_2$ is a facet of $\cP_2$.
A set of the form $\cP\setminus\cF$, with~$\cF$ a facet of~$\cP$, is called a half-open polytope.
We can extend the definition of Ehrhart polynomials to lattice half-open polytopes by $\ehr(\cP\setminus\cF,t)=\ehr(\cP,t)-\ehr(\cF,t)$.
		
\begin{example}\label{ex:half-open}
For the standard $m$-simplex $\simp_m$, we have $\ehr(\simp_m,t)=\binom{t+m}{m}$,
since $t\simp_m\cap\mathbb{Z}^m=\{\mathbf{x}\in\mathbb{Z}^m\mid x_i\ge0\text{ for all }i\in[m],\;\sum_{i=1}^mx_i\le t\}$.
Similarly, it can be seen that, for each $i\in[m]$, the facet $\{\mathbf{x}\in\simp_m\mid x_i=0\}$ of $\simp_m$ has Ehrhart polynomial $\binom{t+m-1}{m-1}$, and that the remaining facet 
$\{\mathbf{x}\in\simp_m\mid\sum_{i=1}^mx_i=1\}$ of $\simp_m$ also has Ehrhart polynomial
$\binom{t+m-1}{m-1}$.  More technically,
any facet of $\simp_m$ is unimodularly equivalent to $\simp_{m-1}$,
and so its Ehrhart polynomial is $\ehr(\simp_{m-1},t)$.
Hence, the Ehrhart polynomial of $\simp_m$ minus any facet is $\binom{t+m}{m}-\binom{t+m-1}{m-1}=\binom{t+m-1}{m}$.
We call such a half-open polytope a standard half-open simplex, and denote it as $\widetilde{\Delta}_m$.
\end{example}

\begin{remark}\label{rem:pyramid}
The Ehrhart polynomial of a lattice polytope $\cP$ of dimension $d$ is often expressed
as $\ehr(\cP,t)=\sum_{i=0}^d h_i^\ast\,\binom{t+d-i}{d}$, where $(h^\ast_0,\ldots,h^\ast_d)$ is called the $h^\ast$-vector of $\cP$.
This binomial coefficient basis is helpful for computing the Ehrhart polynomial of certain pyramids, as follows.
A \emph{lattice pyramid} is a lattice polytope which is unimodularly equivalent to a pyramid $\textrm{Pyr}(\mathcal{B}',\mathbf{e}_m)\subseteq \mathbb{R}^m$, 
where $\mathcal{B}\subseteq\RR^{m-1}$ is a lattice polytope and $\mathcal{B}'=\{(x_1,\ldots,x_{m-1},0)\mid(x_1,\ldots,x_{m-1})\in\mathcal{B}\}$.
It follows from~\cite[Theorem~2.4]{BR} that the $h^\ast$-vector of such a lattice pyramid is obtained by simply appending a zero to the right end of the $h^\ast$-vector of $\mathcal{B}$.
Hence, if $\mathcal{B}$ has dimension $d$ and $\ehr(\mathcal{B},t)=\sum_{i=0}^d h_i^\ast\,\binom{t+d-i}{d}$, then $\textrm{Pyr}(\mathcal{B}',\mathbf{e}_m)$ has
dimension $d+1$ and $\ehr(\textrm{Pyr}(\mathcal{B}',\mathbf{e}_m),t)=\sum_{i=0}^d h_i^\ast\,\binom{t+d+1-i}{d+1}$.
\end{remark}
		
\subsection{The sculpting strategy}
All of our computations of the normalized volume or Ehrhart polynomial of~$\cP(m,n)$ 
in Section~\ref{sec:volppn}, and many of those in Section~\ref{sec:Ehrhart}, 
follow the same sculpting strategy. We start with a well-known polytope and remove other known polytopes by adding inequalities, until we obtain the desired 
polytope $\cP(m,n)$.	
More precisely, we  create a sequence of lattice polytopes $\cP_1,\cP_2,\dots,\cP_k=\cP(m,n)$, where~$\cP_1$ is either the $\binom{n+1}{2}$-dilated standard $m$-simplex $\binom{n+1}{2}\simp_m$ (in Section~\ref{sec:volppn}) or the $m$-cube $[0,n]^m$ of side-length~$n$ (in Section~\ref{sec:Ehrhart}).  We then obtain $\cP_{i+1}$ from $\cP_i$ by adding inequalities to $\cP_i$, i.e., by taking an intersection of~$\cP_i$ with closed halfspaces, and thus removing some pieces from $\cP_i$.

A simple example which illustrates this idea is as follows.
\begin{example}
Figure \ref{fig:32exploded} shows the partial permutohedron $\cP(3,2)=
\{\mathbf{x}\in\RR^3\mid 0\le x_i\le 2\text{ for all }i\in[3],\;x_1+x_2+x_3\le3\}$ as the $3$-dilated standard $3$-simplex
$3\simp_3$, minus three copies of the standard $3$-simplex $\simp_3$.  
\begin{figure}[ht]
\centering
\includegraphics[height=2in]{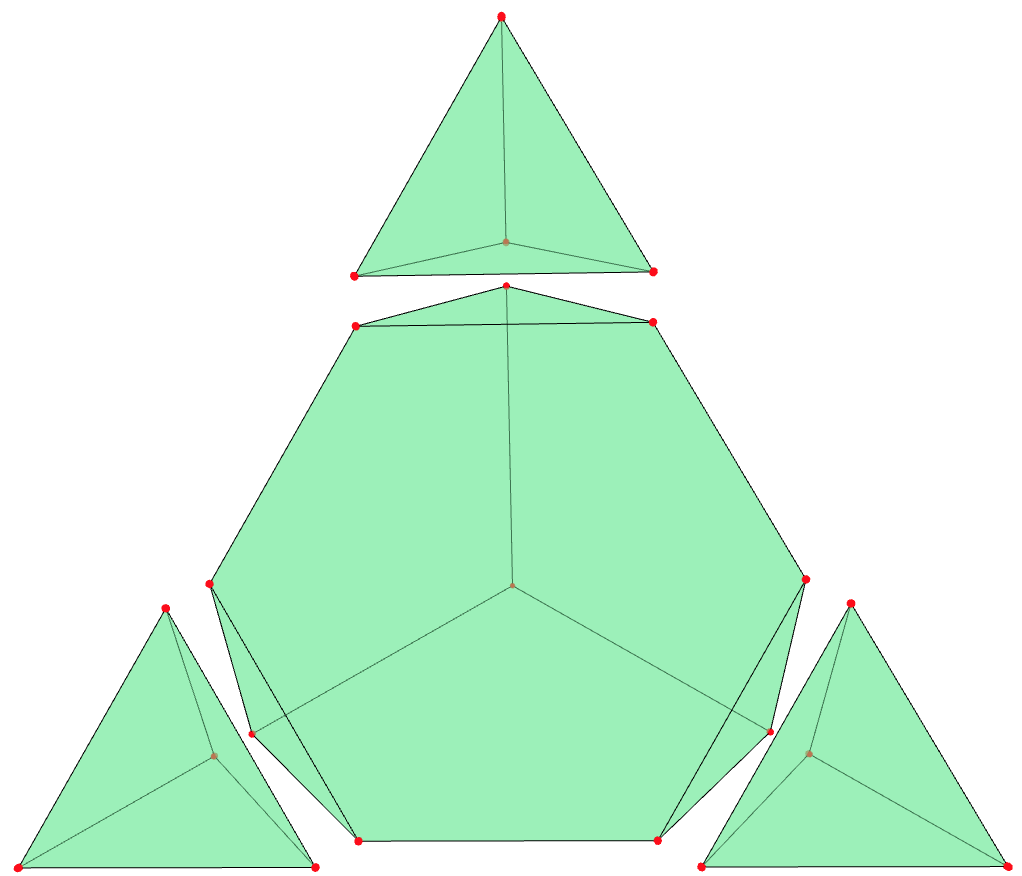}
\caption{Illustration of $P(3,2)$ as $3\simp_3$, minus three copies of $\simp_3$.}
\label{fig:32exploded}
\end{figure}
\end{example}

When applying the sculpting process, we will keep track of the removed pieces using the following lemma.
		
\begin{lemma}\label{lem:strip}
Consider a polytope $\cP\subseteq\mathbb{R}^m$, $\mathbf{a}\in\mathbb{R}^m$ and
$b\in\RR$, and define the polytopes
\begin{equation*}
\cP'=\{\mathbf{x}\in\cP\mid\langle \mathbf{a},\mathbf{x}\rangle \leq b\},\ \cQ=\{\mathbf{x}\in\cP\mid\langle \mathbf{a},\mathbf{x}\rangle \geq b\}\ 
\text{and}\ \cF=\{\mathbf{x}\in\cP\mid\langle \mathbf{a},\mathbf{x}\rangle = b\},
\end{equation*}
where we require that $\cF$ is a facet of both $\cP'$ and $\cQ$.
Then the vertices of $\cQ$ are given by all of the following:
\begin{enumerate}
\item Vertices $\mathbf{v}$ of $\cP$ such that $\langle \mathbf{a},\mathbf{v}\rangle \ge b$.
\item The unique point in $\{\mathbf{x}\in\mathbb{R}^m\mid\langle \mathbf{a},\mathbf{x}\rangle = b\}\cap [\mathbf{v},\mathbf{w}]$, for all edges $[\mathbf{v},\mathbf{w}]$ of $\cP$ such that $\langle \mathbf{a},\mathbf{w}\rangle<b<\langle\mathbf{a},\mathbf{v}\rangle$. 
\end{enumerate}
Furthermore, assuming that $\cP$, $\cP'$, $\cQ$ and $\cF$ are lattice polytopes, we have 
\begin{equation}\label{eq:ie-lemma}
\ehr(\cP',t)=\ehr(\cP,t)-\left(\ehr(\cQ,t)-\ehr(\cF,t)\right).
\end{equation}
\end{lemma}

\begin{proof}
The characterization of the vertices of~$\cQ$ follows from the computation of the vertices of $\cF$, since every vertex of $\cF$ can be obtained as
the intersection of $\{\mathbf{x}\in\mathbb{R}^m\mid\langle \mathbf{a},\mathbf{x}\rangle = b\}$ with an edge of~$\cP$. (This intersection may occur at the endpoint of the edge, in which case 
the vertex of~$\cF$ is a vertex of~$\cP$.)
We can obtain \eqref{eq:ie-lemma} from~\eqref{eq:ie-general},
by taking~$\cP_1$ and~$\cP_2$ in~\eqref{eq:ie-general} to be~$\cP'$ and~$\cQ$.
\end{proof}

For the polytopes in Lemma~\ref{lem:strip},
we say that $\cP'$ is obtained from $\cP$ by removing the half-open polytope $\cQ\setminus\cF$,
or that $\cP'$ is obtained from $\cP$ by adding the inequality $\langle \mathbf{a},\mathbf{x}\rangle \leq b$.  For the vertices of $\cQ$ in Lemma~\ref{lem:strip},
we say that those of type~(1) are on the \emph{forbidden side}, and those of type~(2)
are obtained by \emph{cutting along edges}.

To apply Lemma~\ref{lem:strip} to the sculpting process, we need to know the vertices and edges of the intermediate polytopes which are used.
These intermediate polytopes are covered by the following definition.
		
\begin{definition}\label{def:Pi}
For any $\mathbf{z}\in\mathbb{R}^m$ with nonnegative entries, consider the permutohedron $\Pi(\mathbf{z})=\textrm{ConvexHull}(\{(z_{\sigma(1)},\ldots,z_{\sigma(m)})\mid\sigma\in\mathfrak{S}_m\})$, and 
define the related polytope 
\begin{equation}\label{Pihat}
\widehat{\Pi}(\mathbf{z})=\bigl\{\mathbf{x}\in\mathbb{R}^m\bigm|\mathbf{0}\leq \mathbf{x}\leq\mathbf{y}\text{ for some }\mathbf{y}\in\Pi(\mathbf{z})\bigr\},\end{equation}
where for $\mathbf{x},\mathbf{y}\in\mathbb{R}^m$ we write $\mathbf{x}\leq \mathbf{y}$ if $x_i\leq y_i$ for all $i\in[m]$.
\end{definition}

Some basic properties of $\widehat{\Pi}(\mathbf{z})$ are outlined in the following remark.
\begin{remark} 
It can be seen, using~\eqref{Pihat}, that
\begin{equation}\label{Pihat1}\widehat{\Pi}(\mathbf{z})=\bigl\{(x_1y_1,\ldots,x_my_m)\bigm|\mathbf{x}\in[0,1]^m,\ \mathbf{y}\in\Pi(\mathbf{z})\bigr\},\end{equation}
and it can also be shown straightforwardly that 
\begin{equation}\label{Pihat2}\widehat{\Pi}(\mathbf{z})=\textrm{ConvexHull}\bigl(\bigl\{(a_1z_{\sigma(1)},\ldots,a_mz_{\sigma(m)})\bigm|\mathbf{a}\in\{0,1\}^m,\ \sigma\in\mathfrak{S}_m\bigr\}\bigr).\end{equation}
It follows from~\eqref{Pihat2} that $\widehat{\Pi}(\mathbf{z})$ is indeed a polytope, and that its set of vertices is a subset of 
$\bigl\{(a_1z_{\sigma(1)},\ldots,a_mz_{\sigma(m)})\bigm|\mathbf{a}\in\{0,1\}^m,\ \sigma\in\mathfrak{S}_m\bigr\}$.  It can also be seen that $\Pi(\mathbf{z})=\{\mathbf{x}\in\widehat{\Pi}(\mathbf{z})\mid\sum_{i=1}^mx_i=\sum_{i=1}^mz_i\}$.
It follows from~\eqref{Pihat} that $\widehat{\Pi}(\mathbf{z})$ contains the $z_i$-dilated
standard $m$-simplex $z_i\simp_m$, for each $i\in[m]$
(since $\widehat{\Pi}(\mathbf{z})$ contains each vertex of $z_i\simp_m$). Hence, provided that $\mathbf{z}$ is nonzero, $\widehat{\Pi}(\mathbf{z})$ has dimension~$m$.
\end{remark}
   
We refer to a polytope $\cP$ which is contained in the nonnegative orthant
of~$\RR^m$, and which has the property that if $\mathbf{x}\in\RR^m$ and $\mathbf{y}\in\cP$
satisfy $\mathbf{0}\leq \mathbf{x}\leq \mathbf{y}$ then $\mathbf{x}\in\cP$, as an \emph{anti-blocking} polytope. 
Hence, $\widehat{\Pi}(\mathbf{z})$ is an anti-blocking polytope, and can be regarded
as an anti-blocking version of the permutohedron~$\Pi(\mathbf{z})$.  Certain pairs of anti-blocking polytopes are studied in~\cite{fulkerson},
and certain sets, namely \emph{convex corners} or \emph{compact convex down-sets},
which include anti-blocking polytopes, are studied in~\cite{bollobas}.

The vertices and edges of $\widehat{\Pi}(\mathbf{z})$ are characterized in the following proposition.

\begin{proposition}\label{prop-vertices-edges}
Consider $\mathbf{z}\in\mathbb{R}^m$ with $z_1\geq z_2\geq \dots\geq z_m\geq 0$.
Then the vertices of $\widehat{\Pi}(\mathbf{z})$ are the vectors in $\RR^m$ with entries of zero in any $m-k$ positions, and with the other~$k$ entries being $z_1,\dots,z_k$ in any order, where $k$ ranges from $0$ to $m$.

Two vertices of $\widehat{\Pi}(\mathbf{z})$ form an edge of $\widehat{\Pi}(\mathbf{z})$ if and only if one of the following conditions holds:
\begin{enumerate}
\item One vertex can be obtained from the other by setting its smallest nonzero entry to zero.
\item The vertices differ only by interchanging the positions of entries $z_i$ and $z_{i+1}$, for some $i\in[m-1]$.
\end{enumerate}
\end{proposition}
 		
\begin{proof}
Recall that faces of a polytope~$\cP$ are obtained by maximizing a linear functional over~$\cP$,
and that vertices of~$\cP$ are obtained as those points in~$\cP$ which are the unique maximizers of a linear functional.

Consider a linear functional $\langle\mathbf{w},\cdot\rangle$, for $\mathbf{w}\in\RR^m$,
and let $\alpha_1>\dots>\alpha_t$ be the distinct positive entries of $\mathbf{w}$.
Using the entries of $\mathbf{w}$, we define a partition of $[m]$ as follows.
\begin{itemize}
\item $X_k=\{i\in[m]\mid w_i=\alpha_k\}$, for $k=1,\dots,t$.
\item $Y=\{i\in[m]\mid w_i<0\}$.
\item $Z=\{i\in[m]\mid w_i=0\}$.
\end{itemize}
If $\mathbf{p}\in\widehat{\Pi}(\mathbf{z})$ maximizes $\langle\mathbf{w},\cdot\rangle$ over $\widehat{\Pi}(\mathbf{z})$, then each of the following conditions is satisfied.
\begin{itemize}
\item[(i)] For each $i\in Y$, $p_i=0$.  This occurs because if we had $p_i>0$ for some $i\in Y$, then we could change $\mathbf{p}$ to another vector $\mathbf{p}'$ by replacing $p_i$ by zero, and by the anti-blocking property we would then have $\mathbf{p}'\in\widehat{\Pi}(\mathbf{z})$ and $\langle\mathbf{w},\mathbf{p}\rangle<\langle\mathbf{w},\mathbf{p}'\rangle$.
\item[(ii)] The largest $|X_1|$ entries (allowing equal entries) of $\mathbf{p}$ must be in positions~$X_1$, the next largest $|X_2|$ entries (allowing equal entries) must be in positions $X_2$, and so on up to the entries in positions~$X_t$. Furthermore, the nonzero entries of $\mathbf{p}$ are $z_1,\dots,z_l$ for some $l$.
These conditions occur due to the following reasons. The definition of $\Pi(\mathbf{z})$
as the convex hull of vectors obtained by permuting entries of $\mathbf{z}$ implies that $\sum_{i=1}^mx_i=\sum_{i=1}^mz_i$ for all $\mathbf{x}\in\Pi(\mathbf{z})$,
and the definition of $\widehat{\Pi}(\mathbf{z})$ as an anti-blocking version of $\Pi(\mathbf{z})$ then implies that $\sum_{i=1}^mx_i\le\sum_{i=1}^mz_i$ for all $\mathbf{x}\in\widehat{\Pi}(\mathbf{z})$.
Hence,  the $i$-th largest entry (allowing equal entries) of~$\mathbf{p}$ is at most $z_i$,
and in order to maximize $\langle\mathbf{w},\cdot\rangle$, the nonzero entries of $\mathbf{p}$ must be $z_1,\dots,z_l$ for some $l$.
\item[(iii)] The entries of $\mathbf{p}$ in positions $Z$ are irrelevant, as they do not affect the value of 
$\langle\mathbf{w},\mathbf{p}\rangle$.
\end{itemize}
Now if $\mathbf{p}$ is a \emph{unique} maximizer of $\langle\mathbf{w},\cdot\rangle$ over $\widehat{\Pi}(\mathbf{z})$, then each of the following conditions is also satisfied.
\begin{itemize}
\item[(iv)] We have $z_1=\dots =z_{|X_1|}$,  $z_{|X_1|+1}=\dots=z_{|X_1|+|X_2|}$, \ldots,  $z_{|X_1|+\ldots+|X_{t-1}|+1}=\dots=z_{|X_1|+\ldots+|X_t|}$.
\smallskip
\item[(v)] We have $z_{|X_1|+\ldots+|X_t|+1}=\dots=z_{|X_1|+\ldots+|X_t|+|Z|}=0$,
with the entries of $\mathbf{p}$ in positions~$Z$ all being~$0$.
\end{itemize}
It follows that the vertices of $\widehat{\Pi}(\mathbf{z})$ are precisely those specified in the proposition.

Lastly, recall that an edge is a face with exactly two vertices.
Hence, we now need to consider a linear functional $\langle\mathbf{w},\cdot\rangle$ which is maximized by exactly two vertices of $\widehat{\Pi}(\mathbf{z})$.
By reviewing the argument above (and again using the entries of~$\mathbf{w}$ to partition $[m]$ into sets $X_1$, \ldots, $X_t$, $Y$ and~$Z$),
this can only occur provided that conditions~(i)--(iii) still hold for each of the vertices, 
and conditions~(iv) or~(v) also hold, but with one or other of the following modifications.
\begin{itemize}
\item There exists exactly one $j$ for which $|X_j|=2$, and 
$z_{|X_1|+\ldots+|X_{j-1}|+1}>z_{|X_1|+\ldots+|X_{j-1}|+2}$, with the remaining equalities in condition~(iv) still holding.
In this case, the two vertices differ by a swap of the first kind, as described in~(1)
of the proposition.
\item We have $z_{|X_1|+\ldots+|X_t|+1}>0$ and $z_{|X_1|+\ldots+|X_t|+2}=\dots=z_{|X_1|+\ldots+|X_t|+|Z|}=0$,
with the entries of the two vertices in positions~$Z$ all being~$0$, except for one such entry
in one of the vertices.
In this case, one of the vertices can be obtained from the other by setting its smallest entry equal to zero, as described in~(2)
of the proposition.\qedhere
\end{itemize}
\end{proof}
 		
From the characterization of vertices of $\cP(m,n)$ given by Proposition~\ref{p-vertices},
and the characterization of vertices of $\widehat{\Pi}(\mathbf{z})$ given by Proposition~\ref{prop-vertices-edges}, we obtain the following corollary to those results.
\begin{corollary}\label{cor-anti-block}
The partial permutohedron $\cP(m,n)$ is a special case of an anti-blocking polytope~$\widehat{\Pi}(\mathbf{z})$, with
\begin{equation}\label{eq-anti-block}\cP(m,n)=\begin{cases}
\widehat{\Pi}(n,n-1,n-2,\dots,1,\underbrace{0,\ldots,0\!}_{m-n}\,),&\text{ if }n\le m-2,\\[4mm]
\widehat{\Pi}(n,n-1,n-2,\dots,n-m+1),&\text{ if }n\ge m-1.\end{cases}\end{equation}
\end{corollary}

Some remarks on Theorem~\ref{prop-vertices-edges} and Corollary~\ref{cor-anti-block} are as follows.
\begin{remark}\label{rem-edge}
The characterization of edges of $\widehat{\Pi}(\mathbf{z})$ given by Proposition~\ref{prop-vertices-edges} provides a characterization of edges of~$\cP(m,n)$, due to Corollary~\ref{cor-anti-block}.
Alternatively, this characterization for~$\cP(m,n)$ could have been obtained using
the bijection of Theorem~\ref{thm-faces-chains}, through which 
the edges of~$\cP(m,n)$ correspond to the chains in~$\mathcal{C}(m,n)$ with one missing rank.
\end{remark}

\begin{remark}
Although, by Proposition~\ref{p-simple}, $\cP(m,n)$ is a simple polytope,
$\widehat{\Pi}(\mathbf{z})$ is not a simple polytope for all~$\mathbf{z}$.
For example, consider $\widehat{\Pi}(1,1,0,0)$.  This polytope has dimension $4$,
but has a vertex $(1,1,0,0)$ which is adjacent to~$6$ other vertices 
(specifically, $(1,0,0,0)$, $(0,1,0,0)$, $(0,1,1,0)$, $(0,1,0,1)$,
$(1,0,1,0)$ and $(1,0,0,1)$), so the polytope is not simple.
\end{remark}

\subsection{The specific results}\label{sec:pmproofs}
We now provide the specific results for $\cP(m,n)$ with 
arbitrary $m$ and $n\le4$.  

For $n=1$, we have $\cP(m,1)=\Delta_m$
(as seen in Example~\ref{ex-Pm1}),
and so, using Example~\ref{ex:half-open}, $\cP(m,1)$ has Ehrhart polynomial $\binom{t+m}{m}$ and normalized volume $1$.

Proceeding to $n=2$, $n=3$ and $n=4$,
we note that, although the descriptions
we will use for $\cP(m,n)$ 
(specifically,~\eqref{eq:cPm2},~\eqref{eq:cPm3} and~\eqref{eq:Pm4-2})
will be taken from~\eqref{facetdesc} with $m\ge n$, all of these 
descriptions remain valid for
arbitrary $m$, since some of the inequalities 
within the descriptions become either redundant or empty for $m<n$.  For example, consider the description~\eqref{eq:Pm4-2} for~$\cP(m,4)$ which is used in the proof of Theorem~\ref{theorem:Pm4}.  For $m=3$, the first inequality, $x_1+x_2 +x_3\leq 10$,
in~\eqref{eq:Pm4-2} is redundant,
since the last class of inequalities, $x_i+x_j+x_k\le9$,
has the single case $x_1+x_2+x_3\le9$.  For $m=2$,
the first inequality, $x_1+x_2\leq 10$, in~\eqref{eq:Pm4-2} is again redundant (since there is also an inequality $x_1+x_2\le7$), and the last class of inequalities, $x_i+x_j+x_k\le9$, is now empty (since there are
no $i$, $j$ and $k$ with $1\le i<j<k\le 2$).

Our next result confirms Conjecture~5.30 in~\cite{HS}. 
\begin{theorem}[Conjecture~5.30 in \cite{HS}]
\label{theorem:pm2}
For any $m$, the Ehrhart polynomial of $\cP(m,2)$ is 
\begin{equation}\label{eq:pm2}\ehr(\cP(m,2),t)=\binom{3t+m}{m}-m\binom{t+m-1}{m},\end{equation}
and thus, taking $m!$ times the coefficient of $t^m$ in~$\ehr(\cP(m,2),t)$,
the normalized volume of $\cP(m,2)$ is 
\begin{equation}v(m,2)=3^m-m.\end{equation}
\end{theorem}
\begin{proof}
We first consider the polytope
\[\cP_1 = 3\simp_m= \left \{ \mathbf{x} \in \mathbb{R}^m  \,\middle |\, \begin{matrix}  x_1+x_2 + \cdots +x_m \leq 3,  \\  0 \leq x_i \text{ for all } 1 \leq i \leq m \end{matrix} \right \},\]
whose Ehrhart polynomial is $\binom{3t+m}{m}$.
		
We then consider
\begin{equation}\label{eq:cPm2} \cP_2=\cP(m,2)= \left \{ \mathbf{x} \in \mathbb{R}^m \,\middle |\, \begin{matrix}  x_1+x_2 + \cdots +x_m \leq 3,  \\  0 \leq x_i \leq 2 \text{ for all } 1 \leq i \leq m \end{matrix} \right \},\end{equation}
and we observe  that $\cP_2$ is $\cP_1$ with $m$ half-open polytopes $\widetilde{\cR}_1,\ldots,\widetilde{\cR}_m$ removed, where 
\begin{align*}\widetilde{\cR}_j&=\left\{\mathbf{x}\in\mathbb{R}^m\,\middle|\,\begin{matrix}x_1+x_2 + \cdots +x_m \leq 3,\\[1.2mm]
x_i\ge0\text{ for all }i\in[m]\setminus\{j\},\ x_j>2\end{matrix}\right\}\\
&=3\Delta_m\cap\{\mathbf{x}\in\mathbb{R}^m\mid x_j> 2\},\end{align*}
for $1\le j\le m$.  These half-open polytopes are pairwise disjoint since,
for $j\ne k$, any $\mathbf{x}\in\widetilde{\cR}_j\cap\widetilde{\cR}_k$ 
would need to satisfy $4<x_j+x_k\leq 3$, which is impossible. 
Also, $\widetilde{\cR}_j$ is a translation by $2\mathbf{e}_j$ of the half-open simplex $\widetilde{\Delta}_m$,
where this can be verified by, for example, observing that 
the closure~$\cR_j$ of~$\widetilde{\cR}_j$ has vertices $2\mathbf{e}_j+\mathbf{e}_i$ for $i\in\{0,1,\dots,m\}$.
Since the Ehrhart polynomial of $\widetilde{\Delta}_m$ is $\binom{t+m-1}{m}$, it follows that 
the Ehrhart polynomial of $\cP(m,2)$ is $\binom{3t+m}{m}-m\binom{t+m-1}{m}$, as required.
\end{proof} 
		
We now extend our method to $\cP(m,3)$. 

\begin{theorem} \label{theorem:Pm3}
For any $m$, the Ehrhart polynomial of $\cP(m,3)$ is
\begin{multline}
\label{eq:ehrpm3}
\ehr(\cP(m,3),t)\\
=\binom{6t+m}{m}-m\binom{3t+m-1}{m}-\binom{m}{2}\left(\binom{t+m-1}{m}+(m-2)\binom{t+m-2}{m}\right),
\end{multline}
and thus, taking $m!$ times the coefficient of $t^m$ in~$\ehr(\cP(m,3),t)$,
the normalized volume of $\cP(m,3)$ is 
\begin{equation}v(m,3)= 6^m-m\,3^m-(m-1)\binom{m}{2}.\end{equation}
\end{theorem}
\begin{proof}
We construct $\cP(m,3)$ in three steps. 
These are illustrated for the case $m=3$ in Figures~\ref{fig:33exploded} and~\ref{fig:33step}.

\noindent \textbf{(1).} We first consider the polytope 
\[\cP_1=6\Delta_m=\left \{ \mathbf{x} \in \mathbb{R}^m \,\middle |\, \begin{matrix}  x_1+x_2 + \cdots +x_m \leq 6,  \\  0 \leq x_i \text{ for all } 1 \leq i \leq m \end{matrix} \right \},\]
whose Ehrhart polynomial is $\binom{6t+m}{m}$.
				
\noindent \textbf{(2).} We now add inequalities to $\cP_1$, and consider
\[\cP_2=\left \{ \mathbf{x}\in \mathbb{R}^m \,\middle |\, \begin{matrix}  x_1+x_2 + \cdots +x_m \leq 6,\\  0 \leq x_i \leq 3 \text{ for all } 1 \leq i \leq m \end{matrix} \right \}=\widehat{\Pi}(3,3,0,\dots,0),\]
where for $m=1$, we take $\widehat{\Pi}(3,3,0,\dots,0)$ to be $\widehat{\Pi}(3)$. 
Thus, $\cP_2$ is obtained from $\cP_1$ by removing~$m$
half-open polytopes $\widetilde{\cR}_1,\ldots,\widetilde{\cR}_m$,
where
\begin{align*}\widetilde{\cR}_j&=\left\{\mathbf{x}\in\mathbb{R}^m\,\middle|\,\begin{matrix}x_1+x_2 + \cdots +x_m \leq 6,\\[1.2mm]
x_i\ge0\text{ for all }i\in[m]\setminus\{j\},\ x_j>3\end{matrix}\right\}\\
&=6\Delta_m\cap\{\mathbf{x}\in\mathbb{R}^m\mid x_j> 3\},\end{align*}
for $1\le j\le m$.
These half-open polytopes are pairwise disjoint since, 
for $j\ne k$, any $\mathbf{x}\in\widetilde{\cR}_j\cap\widetilde{\cR}_k$ 
would need to satisfy $6<x_j+x_k\leq 6$, which is impossible.
Similarly to the corresponding step in the proof of Theorem \ref{theorem:pm2}, it can be seen that~$\widetilde{\cR}_j$ is a translation by $3\mathbf{e}_j$ of $3\widetilde{\Delta}_m$, and we conclude that the Ehrhart polynomial of $\cP_2$ is $\binom{6t+m}{m}-m\binom{3t+m-1}{m}$.
				
\noindent\textbf{(3).} Finally, we add the remaining inequalities and consider
\begin{equation}\label{eq:cPm3}\cP_3=\cP(m,3) = \left \{ \mathbf{x} \in \mathbb{R}^m \,\middle |\, \begin{matrix}  x_1+x_2 + \cdots +x_m \leq 6,  \\  0 \leq x_i \leq 3 \text{ for all } 1 \leq i \leq m, \\ x_i+x_j\leq 5 \text{ for all } 1\leq i<j\leq m \end{matrix} \right \}.\end{equation}
Thus, $\cP_3$ is obtained from $\cP_2$ by removing the $\binom{m}{2}$ 
half-open polytopes 
$\widetilde{\cR}_{ij}=\cP_2\cap\{\mathbf{x}\in\mathbb{R}^m \mid  x_i+x_j> 5\}$, for distinct $i,j$ in $[m]$.
To show that these half-open polytopes are pairwise disjoint, we consider two cases, as follows.
\begin{description}
\item[$\widetilde{\cR}_{ij}$ and $\widetilde{\cR}_{kl}$, with $i,j,k,l$ all different] Any $\mathbf{x}\in\widetilde{\cR}_{ij} \cap \widetilde{\cR}_{kl}$ would need to satisfy
$10 < x_i + x_j + x_k + x_l \leq 6$, which is impossible.
\item[$\widetilde{\cR}_{ij}$ and $\widetilde{\cR}_{jk}$, with $i,j,k$ all different] Any $\mathbf{x}\in\widetilde{\cR}_{ij} \cap \widetilde{\cR}_{jk}$ would need to satisfy 
$10=5+5<(x_i+x_j)+(x_j+x_k)=(x_i+x_j+x_k)+x_j\leq 6+3=9$, which gives the contradiction
$10<9$.
\end{description}

 Note that each $\widetilde{\cR}_{ij}$ is congruent to $\widetilde{\cR}_{12}=\cP_2\cap\{\mathbf{x}\in\mathbb{R}^m \mid  x_1+x_2> 5\}$. 
Using Lemma \ref{lem:strip}, we compute the closure $\cR_{12}$ of $\widetilde{\cR}_{12}$ to be congruent to  a lattice pyramid with base $\Pi(3,2)\times \Delta_{m-2}$ and apex $3\mathbf{e}_1+3\mathbf{e}_2$.
The Ehrhart polynomial of lattice pyramids is easier to handle using a binomial coefficient basis.
The Ehrhart polynomial of the base is $$\binom{t+1}{1}\binom{t+m-2}{m-2}=\binom{t+m-1}{m-1}+(m-2)\binom{t+m-2}{m-1},$$ so, by Remark~\ref{rem:pyramid}, the Ehrhart polynomial of the pyramid is $\binom{t+m}{m}+(m-2)\binom{t+m-1}{m}$.
The Ehrhart polynomial of $\cP_3$ is $\ehr(\cP_2,t)-\binom{m}{2}\left(\binom{t+m}{m}+(m-2)\binom{t+m-1}{m}-\binom{t+m-1}{m-1}-(m-2)\binom{t+m-2}{m-1}\right)$, since we removed the pyramid, but we have to replace its base.
		 
Putting everything together, we obtain the desired formula in~\eqref{eq:ehrpm3}.
\end{proof}

\begin{figure}[ht]
\includegraphics[height=2in]{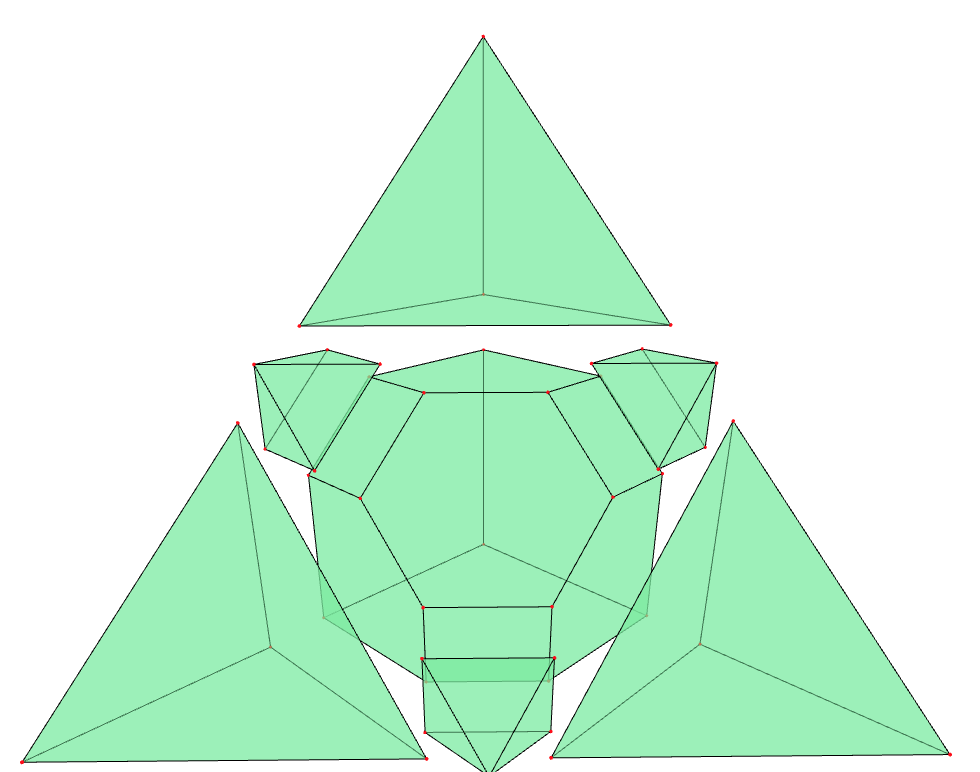}
\caption{The construction of $\cP(3,3)$ (in the center), as used in the proof of
Theorem~\ref{theorem:Pm3}.}
\label{fig:33exploded}
\end{figure}
		
\begin{figure}[ht]
\centering
\includegraphics[height=2in]{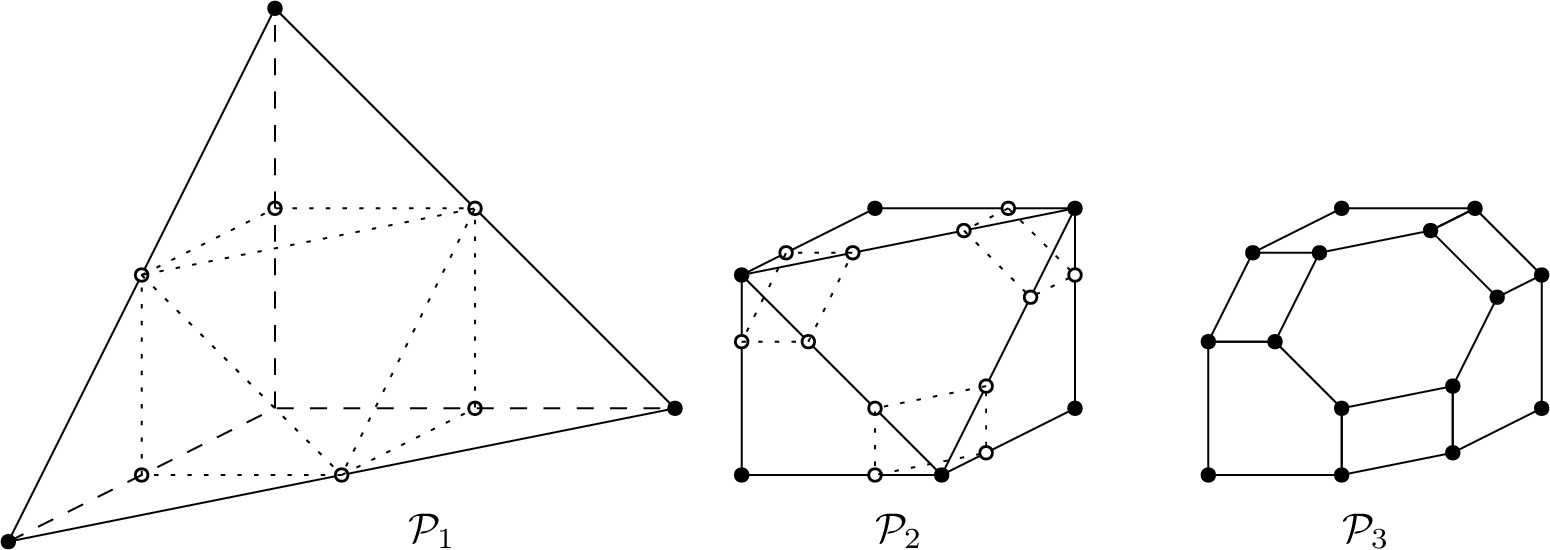}
\caption{A step-by-step illustration of the proof of Theorem \ref{theorem:Pm3}, for the
case $m=3$.}
\label{fig:33step}
\end{figure}
		
To conclude this section, we extend
our methods one step further to $\cP(m,4)$, but the proof becomes more involved as the pieces we are removing become more complicated.
In light of this, we compute only the normalized volume.
		
\begin{theorem} \label{theorem:Pm4}
For any $m$, the normalized volume of $\cP(m,4)$ is
\begin{equation}
\label{eq:Pm4}
v(m,4)=10^m-m\,6^m-\frac{m(m-1)(m-3)}{6}\,3^m-(3m^2-6m+1)\binom{m}{3}.
\end{equation}
\end{theorem}
		
\begin{proof}
We construct $\cP(m,4)$ in four steps.
		
\noindent\textbf{(1).} We first consider
 $$\cP_1=10\Delta_m=\left \{ \mathbf{x} \in \mathbb{R}^m \,\middle |\, \begin{matrix}  x_1+x_2 + \cdots +x_m \leq 10,  \\  0 \leq x_i \text{ for all } 1 \leq i \leq m\end{matrix} \right \},$$ 
which has normalized volume $10^m$.
			
\noindent\textbf{(2).}  We next consider $$\cP_2= \left \{ \mathbf{x} \in \mathbb{R}^m \,\middle |\, \begin{matrix}  x_1+x_2 + \cdots +x_m \leq 10, \\  0 \leq x_i \leq 4 \text{ for all } 1 \leq i \leq m\end{matrix} \right \}=\widehat{\Pi}(4,4,2,0,0,0,\dots),$$
where the first $m$ terms of $4,4,2,0,0,0,\dots$ are used.
Thus,~$\cP_2$ is obtained from~$\cP_1$ by removing~$m$ half-open polytopes, each congruent to $\widetilde{\cQ}_1= \cP_1\cap\{\mathbf{x}\in\mathbb{R}^m\mid x_1> 4\}$.
The closure $\cQ_1$ of $\widetilde{\cQ}_1$ is congruent to $6\Delta_m$, and hence has normalized volume $6^m$.
However, we need to replace the $\binom{m}{2}$ pairwise intersections of the removed pieces, each of which is congruent to $\widetilde{\cT}_1=\cP_1\cap\{\mathbf{x}\in\mathbb{R}^m\mid x_1>4,\, x_2>4\}$.
It can be seen that the closure $\cT_1$ of $\widetilde{\cT}_1$ is congruent to $2\Delta_m$, and hence has normalized volume $2^m$.
It can also be seen that are no triple intersections among the removed pieces
(since any~$\mathbf{x}$ in such an intersection would need to satisfy
$12<x_i+x_j+x_k\le10$ for some distinct $i,j,k$ in $[m]$),
so we conclude that $\cP_2$ 
has normalized volume $10^m-m\,6^m+\binom{m}{2}\,2^m$.
			
\noindent\textbf{(3).}  We now consider 
$$ \cP_3= \left \{ \mathbf{x} \in \mathbb{R}^m \,\middle |\, \begin{matrix}  x_1+x_2 + \cdots +x_m \leq 10,\\  0 \leq x_i \leq 4 \text{ for all } 1 \leq i \leq m,\\ x_i+x_j\leq 7 \text{ for all } 1\leq i<j\leq m \end{matrix} \right \}=\widehat{\Pi}(4,3,3,0,0,0,\dots),$$
where the first $m$ terms of $4,3,3,0,0,0,\ldots$ are used.
Thus, $\cP_3$ is obtained from~$\cP_2$ by removing $\binom{m}{2}$ half-open polytopes, each congruent to $$\widetilde{\cQ}_2=\cP_2\cap\{\mathbf{x}\in\mathbb{R}^m\mid x_1+x_2> 7\}.$$
That the removed pieces are pairwise disjoint follows from an argument analogous to 
that used in step~(3) of the proof of Theorem \ref{theorem:Pm3}.
Using Lemma \ref{lem:strip}, we can obtain a description of the closure~$\cQ_2$ of~$\widetilde{\cQ}_2$.
Consider \[\Delta_{m-2}=\textrm{ConvexHull}(\{\mathbf{e}_0,\mathbf{e}_3,\mathbf{e}_4,\dots,\mathbf{e}_m\}).\]
Then $\cQ_2$ is the convex hull of the union of $4\mathbf{e}_1+4\mathbf{e}_2+2\Delta_{m-2}$, $4\mathbf{e}_1+3\mathbf{e}_2+3\Delta_{m-2}$ and $3\mathbf{e}_1+4\mathbf{e}_2+3\Delta_{m-2}$. 
By Lemma \ref{lem:aux1}, this has normalized volume $2^m-3^m+m\,3^{m-1}$, and
so $\cP_3$ has normalized volume $10^m-m\,6^m+\binom{m}{2}\,2^m-\binom{m}{2}(2^m-3^m+m\,3^{m-1})=10^m-m\,6^m+\binom{m}{2}\,(3^m-m\,3^{m-1})$.

\noindent\textbf{(4).}  Finally, we consider
\begin{equation}\label{eq:Pm4-2}
\cP_4=\cP(m,4) = \left \{ \mathbf{x}\in \mathbb{R}^m \,\middle |\, \begin{matrix}  x_1+x_2 + \cdots +x_m \leq 10,\\  0 \leq x_i \leq 4 \text{ for all } 1 \leq i \leq m,\\ x_i+x_j\leq 7 \text{ for all } 1\leq i<j\leq m,\\ x_i+x_j+x_k \leq 9 \text{ for all } 1\leq i<j<k\leq m \end{matrix} \right \}.\end{equation}
Thus, $\cP_4$ is obtained from~$\cP_3$ by removing the $\binom{m}{3}$ half-open polytopes
$\widetilde{\cR}_{ijk}=\cP_3\cap\{\mathbf{x}\in\mathbb{R}^m\mid x_i+x_j+x_k>9\}$, for distinct
$i,j,k$ in $[m]$.
That these pieces are pairwise disjoint again follows from an argument analogous to
that used in step~(3) of the proof of Theorem~\ref{theorem:Pm3}.
(In fact, there is now one more case to consider since, for pieces $\widetilde{\cR}_{i_1j_1k_1}$ and
$\widetilde{\cR}_{i_2j_2k_2}$, the intersection of $\{i_1,j_1,k_1\}$ and
$\{i_2,j_2,k_2\}$ can be of size~0,~1 or~2. However, for the case 
in which this intersection is of size~2, contradictory inequalities can again be found
for the entries of any point in $\widetilde{\cR}_{i_1j_1k_1}\cap\widetilde{\cR}_{i_2j_2k_2}$.)  It can be seen that each $\widetilde{\cR}_{ijk}$ is congruent to $\widetilde{\cR}_{123}$,
and using Lemma~\ref{lem:strip}, we can obtain a vertex description of the closure $\cR_{123}$ 
of $\widetilde{\cR}_{123}$. These vertices consist of  $4\mathbf{e}_1+3\mathbf{e}_2+3\mathbf{e}_3,$ $3\mathbf{e}_1+4\mathbf{e}_2+3\mathbf{e}_3$ and $3\mathbf{e}_1+3\mathbf{e}_2+4\mathbf{e}_3$, and the vertices of $\Pi(4,3,2)\times\Delta_{m-3}$, with $\Delta_{m-3}
=\textrm{ConvexHull}(\{\mathbf{e}_0,\mathbf{e}_4,\mathbf{e}_5,\dots,\mathbf{e}_m)\})$.
By Lemma~\ref{lem:aux2},~$\cR_{123}$ has normalized volume $3m^2-6m+1$.

Putting everything together now gives the desired volume formula in~\eqref{eq:Pm4}.
\end{proof}

Having calculated $v(m,n)$ for $n\le4$, a natural open problem remains.
\begin{oproblem}
Find $v(m,n)$ for all $m$ and $n$ with $n>4$.
\end{oproblem}

Extrapolating from our results for $n\leq 4$, it is natural to conjecture
that $v(m,n)$ can be expressed in a certain form, as follows.
\begin{conjecture}\label{conj-vmn}
We conjecture that 
\begin{multline}\label{eq-conj-vmn}
v(m,n)=\binom{n+1}{2}^m-\binom{m}{1}\,\binom{n}{2}^m-p_{n,1}(m)\,\binom{m}{2}\,\binom{n-1}{2}^{m-1}
-p_{n,2}(m)\,\binom{m}{3}\,\binom{n-2}{2}^{m-2}\\
-\ldots-p_{n,n-2}(m)\,\binom{m}{n-1}\,\binom{2}{2}^{m-n+2},\end{multline}
where $p_{n,i}(m)$ is a polynomial in $m$ of
degree $i$, with positive leading coefficient.
Note that the first two terms on the right hand side of~\eqref{eq-conj-vmn} can be regarded
as arising from the first two steps of the sculpting method.
\end{conjecture}

\section{Ehrhart polynomial of the partial permutohedron~$\cP(m,n)$ with $n\ge m-1$}\label{sec:Ehrhart}
We now return to a consideration of $\cP(m,n)$ with $n\geq m-1$, as in Section~\ref{sec:volppm}, and obtain results and a conjecture for its Ehrhart polynomial.

In Section~\ref{sec:Ehr4}, we use a sculpting approach to compute explicitly the Ehrhart polynomial of~$\cP(m,n)$ with $m\le4$ and $n\ge m-1$.  

In Section~\ref{ssec-gpformulae}, we use the Minkowski sum decomposition~\eqref{eq-Minkowski2}, and a result 
for generalized permutohedra~\cite[Theorem~11.3]{postnikov}, to obtain in~\eqref{eq-Ehr-drac-seq} an expression for 
the Ehrhart polynomial of~$\cP(m,n)$ with $n\ge m-1$, as a sum over certain sequences.  
Various general properties can be deduced from this expression.  For example, it follows that 
$\ehr(\cP(m,n),t)$ with $n\ge m-1$ has (as a polynomial in~$t$ of degree~$m$) coefficients which are all positive (for fixed~$m$ and~$n$), and that it is also a polynomial in~$n$ of degree~$m$.
It was observed by Heuer and Striker~\cite[Remark~5.31]{HS} that all the coefficients of $\ehr(\cP(m,n),t)$ (as a polynomial in~$t$) are positive for any fixed $m,n\le7$.  Hence, we now have a proof that this property holds for all $n\ge m-1$, but we still lack a proof
for all $n<m-1$.

In Section~\ref{sec-further}, we conjecture an explicit formula for 
the Ehrhart polynomial of~$\cP(m,n)$ with $n\ge m-1$.

\subsection{Ehrhart polynomial of~$\cP(m,n)$ with $m\le4$ and $n\ge m-1$}\label{sec:Ehr4}
We now obtain explicit expressions for the Ehrhart polynomial of $\cP(m,n)$ for fixed $m\le 4$ and arbitrary $n\geq m-1$.

For $m=1$ and any $n$, the Ehrhart polynomial of $\cP(1,n)$ is 
\begin{equation}\label{eq:ehrp1n}
\ehr(\cP(1,n),t)=nt+1,\end{equation}
since $\cP(1,n)$ is the line segment $[0,n]$.

For $m=2$ and any $n$, the Ehrhart polynomial of $\cP(2,n)$ is 
\begin{equation}\label{eq:ehrp2n}
\ehr(\cP(2,n),t)=(n^2-1/2)\,t^2+(2n-1/2)\,t+1,
\end{equation}
where this could be obtained, as a very simple application of the sculpting strategy, 
by constructing $\cP(2,n)$ as a square $[0,n]^2$
(which has Ehrhart polynomial $(nt+1)^2$), from which a 
half-open triangle $\textrm{ConvexHull}(\{(n-1,n),\,(n,n-1),\,(n,n)\})\setminus \textrm{ConvexHull}(\{(n-1,n),\,(n,n-1)\})$ (which has
Ehrhart polynomial $\binom{t+1}{2}$) has been removed.

We now continue to use the sculpting strategy to obtain the Ehrhart polynomial
of $\cP(m,n)$ for $m=3$ and $m=4$ with $n\ge\binom{m}{2}$, by removing pieces from the $n$-dilated unit $m$-cube $[0,n]^m$.  
The reason for the restriction $n\ge\binom{m}{2}$ is related to the fact  that the second step of the sculpting process involves the removal
of a $\binom{m}{2}$-dilated standard $m$-simplex from $[0,n]^m$.
Nevertheless, the Ehrhart polynomials for the remaining four cases of $m=3$ and $m=4$ with $m-1\le n<\binom{m}{2}$
(i.e., $\cP(3,2)$, $\cP(4,3)$, $\cP(4,4)$ and $\cP(4,5)$) can be computed separately, and 
are found to match the expressions obtained for $n\ge\binom{m}{2}$.

\begin{theorem}\label{thm:p3n}
For any $n\geq 2$, the Ehrhart polynomial of $\cP(3,n)$ is
\begin{equation}\label{eq:ehrp3n}
\ehr(\cP(3,n),t)=
\left(n^3-3n/2-1\right)t^3+\left(3n^2-3n/2-3/2\right)t^2+\left(3n-3/2\right)t+1.\end{equation}
\end{theorem}

\begin{proof}
The Ehrhart polynomial of $\cP(3,2)$ can be computed individually, for example using SageMath, as $4t^3+15/2\,t^2+9/2\,t+1$,
which matches~\eqref{eq:ehrp3n}.

We now compute the Ehrhart polynomial of $\cP(3,n)$ with $n\ge3$ in three steps, as illustrated in Figure~\ref{fig:thm:p3n}.  (Note that, for $n=2$, the equation $\cP_2=\widehat{\Pi}(n,n,n-3)$ in step~(2), and certain subsequent statements, would no longer hold.)

\noindent\textbf{(1).} We begin with an $n$-dilated unit cube,
$$\cP_1=\left \{ \mathbf{x} \in \mathbb{R}^3 \mid  0 \leq x_i\leq n \text{ for all } 1 \leq i \leq 3\right \},$$
which has Ehrhart polynomial $(nt+1)^3$.

\noindent\textbf{(2).} Next, we add one inequality, and consider  $$\cP_2=\left \{ \mathbf{x} \in \mathbb{R}^3 \,\middle |\, \begin{matrix}   0 \leq x_i \leq n \text{ for all } 1 \leq i \leq 3,\\x_1+x_2+x_3 \leq 3n-3 \end{matrix} \right \}=\widehat{\Pi}(n,n,n-3).$$
Thus, $\cP_2$ is obtained from $\cP_1$ by removing a half-open simplex from the $(n,n,n)$ corner of the $n$-dilated cube, namely $$\textrm{ConvexHull}
\begin{bmatrix}n-3&n&n&n\\n&n-3&n&n\\n&n&n-3&n\end{bmatrix}$$ 
minus one facet, where in this proof and the proof of Theorem~\ref{thm:p4n}, the convex hull of a matrix
 denotes the convex hull of the set of vectors formed by the columns of the matrix.
By removing the half-open simplex, we adjust the Ehrhart polynomial by $-\binom{3t+2}{3}$.

\noindent\textbf{(3).} We now add the three remaining inequalities,
and consider
$$\cP_3=\cP(3,n)=\left \{ \mathbf{x} \in \mathbb{R}^3 \,\middle |\, \begin{matrix}   0 \leq x_i \leq n \text{ for all } 1 \leq i \leq 3,
\\x_1+x_2+x_3 \leq 3n-3,\\
x_1+x_2\leq2n-1,\, x_1+x_3\leq2n-1,\, x_2+x_3\leq2n-1\end{matrix} \right \}.$$
Thus, $\cP_3$ is obtained from $\cP_2$ by removing the three half-open polytopes $\widetilde{R}_{ij}=\cP_2\cap\{\mathbf{x}\in\mathbb{R}^m\mid x_i+x_j>2n-1\}$,
for distinct $i,j$ in $[3]$. These half-open polytopes are pairwise disjoint since, 
for distinct $i,j,k$ in $[3]$, any $\mathbf{x}\in\widetilde{\cR}_{ij}\cap\widetilde{\cR}_{jk}$ 
would need to satisfy $4n-2=(2n-1)+(2n-1)<(x_i+x_j)+(x_j+x_k)=(x_i+x_j+x_k)+x_j
=(x_1+x_2+x_3)+x_j\leq(3n-3)+n=4n-3$, which gives the contradiction $4n-2<4n-3$.  It can be seen that each $\widetilde{R}_{ij}$ is congruent
to $\widetilde{R}_{12}$, which is almost a triangular prism.
More precisely, $\widetilde{R}_{12}$ is
\begin{equation}\label{eq:matrix_ad_hoc}
\textrm{ConvexHull}\begin{bmatrix}
n&n&n-1&n&n-1&n\\
n&n-1&n&n-1&n&n\\
0&0&0&n-2&n-2&n-3
\end{bmatrix},
\end{equation}
minus the rectangle on the plane $\{\mathbf{x} \in \mathbb{R}^3\mid x_1+x_2=2n-1\}$.
To compute its Ehrhart polynomial, we start by considering $\binom{t+2}{2}((n-2)t+1)$, the Ehrhart polynomial of the full prism.
Then we correct by $-\binom{t+2}{3}$ since we have to subtract a half-open simplex arising from
\[\textrm{ConvexHull}\begin{bmatrix}
n&n&n-1&n\\
n&n-1&n&n\\
n-2&n-2&n-2&n-3
\end{bmatrix}.\]
Finally, we subtract the Ehrhart polynomial of the rectangle, which
is $((n-2)t+1)(t+1)$.
 
All in all, the Ehrhart polynomial of $\cP(3,n)$ is
\[(nt+1)^3-\binom{3t+2}{3}-3\left[\binom{t+2}{2}((n-2)t+1)-\binom{t+2}{3}-((n-2)t+1)(t+1)\right],\]
which simplifies to the polynomial in~\eqref{eq:ehrp3n}.
\end{proof}

\begin{figure}[ht]
\centering
\includegraphics[height=1.8in]{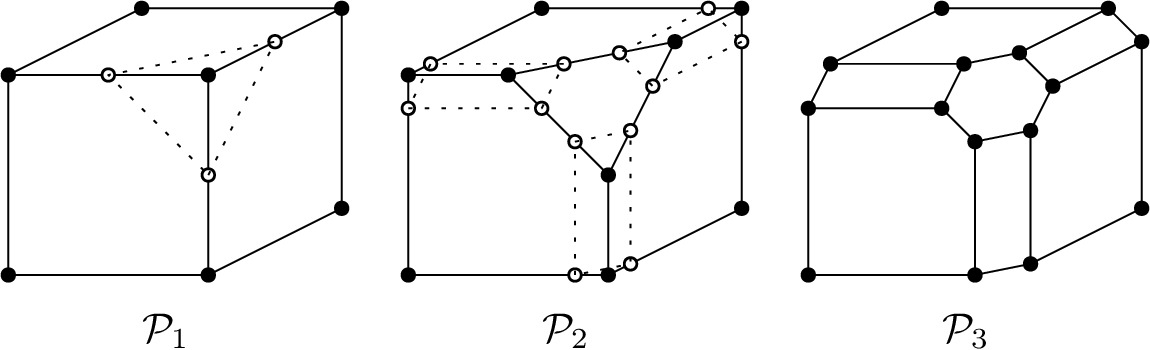}
\caption{A step-by-step illustration of the proof of Theorem \ref{thm:p3n}.}
\label{fig:thm:p3n}
\end{figure}

Note that by taking $3!$ times the coefficient of $t^3$ in Theorem~\ref{thm:p3n}, we recover the formula for $v(m,3)$ in Example~\ref{ex:formulas}, but now with a more geometric proof.

We can push the sculpting strategy one dimension higher to compute the Ehrhart polynomial of $\cP(4,n)$, though the calculations become considerably more tedious.

\begin{theorem}\label{thm:p4n}
For any $n\geq3$, the Ehrhart polynomial of $\cP(4,n)$ is
\begin{multline}
\label{eq:ehrp4n}
\ehr(\cP(4,n),t)=\left(n^4-3n^2-4n-9/4\right)t^4+\left(4n^3-3n^2-6n-5/2\right)t^3\\
+\left(6n^2-6n-9/4\right)t^2+\left(4n-3\right)t+1.
\end{multline}
\end{theorem}

\begin{proof}
The Ehrhart polynomials of $\cP(4,3)$, $\cP(4,4)$ and $\cP(4,5)$ can be computed individually, for example using SageMath, and are found to match~\eqref{eq:ehrp4n}.

We now compute the Ehrhart polynomial of $\cP(4,n)$ with $n\ge6$ in four steps.  
(Note that, for $n\le5$, the equation $\cP_2=\widehat{\Pi}(n,n,n,n-6)$ in step~(2), and certain subsequent statements, would no longer hold.)

\noindent\textbf{(1).} We begin with an $n$-dilated unit $4$-cube,
$$\cP_1=\left \{ \mathbf{x} \in \mathbb{R}^4 \mid  0 \leq x_i\leq n \text{ for all } 1 \leq i \leq 4\right \},$$ 
which has Ehrhart polynomial $(nt+1)^4$.

\noindent\textbf{(2).} Next, we consider 
$$\cP_2=\left \{ \mathbf{x} \in \mathbb{R}^4 \,\middle |\, \begin{array}{c}
0 \leq x_i\leq n \text{ for all } 1 \leq i\leq 4,  \\
x_1+x_2+x_3+x_4\leq 4n-6
\end{array} \right \}=\widehat{\Pi}(n,n,n,n-6).$$ 
We have removed from $\cP_1$ everything with $x_1+x_2+x_3+x_4>4n-6$. By Lemma~\ref{lem:strip}, the vertices of this removed piece can be obtained as follows. The only vertex on the forbidden side 
(i.e., with $x_1+x_2+x_3+x_4>4n-6$) is $(n,n,n,n)$, and we cut along its four incoming edges
(from $(0,n,n,n)$ and its permutations), giving four further vertices, consisting of $(n-6,n,n,n)$ and its permutations.
Since we have removed a half-open simplex, we correct by $-\binom{6t+3}{4}$.

\noindent\textbf{(3).} We now consider
$$\cP_3=\left \{ \mathbf{x} \in \mathbb{R}^4 \,\middle |\, \begin{array}{c}
0\leq x_i\leq n \text{ for all } 1 \leq i\leq 4,  \\
x_1+x_2+x_3+x_4\leq 4n-6,\\
x_1+x_2+x_3\phantom{+x_4}\leq 3n-3,\\
x_1+x_2\phantom{+x_3}+x_4\leq 3n-3,\\
x_1\phantom{+x_2}+x_3+x_4\leq 3n-3,\\
\phantom{x_1+}x_2+x_3+x_4\leq 3n-3\end{array} \right \}=\widehat{\Pi}(n,n,n-3,n-3).$$ 
Thus, $\cP_3$ is obtained from $\cP_2$ by removing the four half-open polytopes
$\widetilde{R}_{ijk}=\cP_2\cap\{\mathbf{x}\in\mathbb{R}^m\mid x_i+x_j+x_k>3n-3\}$,
for distinct $i,j,k$ in $[4]$. These half-open polytopes are pairwise disjoint since, 
for distinct $i,j,k,l$ in $[4]$, any $\mathbf{x}\in\widetilde{\cR}_{ijk}\cap\widetilde{\cR}_{jkl}$ 
would need to satisfy $6n-6=(3n-3)+(3n-3)<(x_i+x_j+x_k)+(x_j+x_k+x_l)=(x_i+x_j+x_k+x_l)+x_j+x_k
=(x_1+x_2+x_3+x_4)+x_j+x_k\leq(4n-6)+n+n=6n-6$, which gives the contradiction $6n-6<6n-6$.
Note that each $\widetilde{\cR}_{ijk}$ is congruent to $\widetilde{\cR}_{123}=\widehat{\Pi}(n,n,n,n-6)\cap\{\mathbf{x}\in\mathbb{R}^4\mid x_1+x_2+x_3>3n-3\}$. 
By Lemma~\ref{lem:strip} (and Proposition~\ref{prop-vertices-edges} for
$\widehat{\Pi}(n,n,n,n-6)$), the vertices of the closure $\cR_{123}$ of $\widetilde{\cR}_{123}$ can be obtained as follows. The only vertices on the forbidden side are $(n,n,n,n-6)$ and $(n,n,n,0)$, and the
vertices obtained by cutting along edges are $(n,n,n-3,n-3)$, $(n,n-3,n,n-3)$ and $(n-3,n,n,n-3)$ (which arise from edges incident to $(n,n,n,n-6)$), and $(n,n,n-3,0)$, $(n,n-3,n,0)$ and $(n-3,n,n,0)$ (which arise from edges incident to $(n,n,n,0)$).
This gives almost a prism over a simplex.
Indeed, by adding the point $(n,n,n,n-3)$ we get the prism $$\textrm{ConvexHull}
\left[\begin{array}{cccc}
n&n-3&n&n\\
n&n&n-3&n\\
n&n&n&n-3
\end{array}\right]\times [0,n-3],$$
which has Ehrhart polynomial $\binom{3t+3}{3}((n-3)t+1)$. We need to replace a half-open simplex with Ehrhart polynomial $\binom{3t+3}{4}$,
and we also need to replace 
$$\textrm{ConvexHull}
\left[\begin{array}{ccc}
n-3&n&n\\
n&n-3&n\\
n&n&n-3\end{array}\right]\times [0,n-3],$$
which has Ehrhart polynomial $\binom{3t+2}{2}((n-3)t+1)$.

In total, this step contributes $4\left[-\binom{3t+3}{3}((n-3)t+1)+\binom{3t+3}{4}+\binom{3t+2}{2}((n-3)t+1)\right]$ to the count.

\noindent\textbf{(4).} Finally, we remove six pieces, each congruent to 
$\widehat{\Pi}(n,n,n-3,n-3)\cap\{\mathbf{x}\in\mathbb{R}^4\mid x_1+x_2>2n-1 \}$,
and obtain $\cP(4,n)$.  As in the previous step, it can be checked that the pieces are piecewise disjoint.
By Lemma \ref{lem:strip} (and Proposition~\ref{prop-vertices-edges} for
$\widehat{\Pi}(n,n,n-3,n-3)$), the vertices of the closure of the removed piece 
with $x_1+x_2>2n-1$ are as follows:
\begin{enumerate}
\item The vertices on the forbidden side, which are $(n,n,n-3,n-3)$, $(n,n,n-3,0)$, $(n,n,0,n-3)$ and $(n,n,0,0)$.
\item The vertices obtained by cutting along edges, which are the columns of
\[\left[\begin{array}{cccccccccccccc}
n&n&n&n&n&n-1&n-1&n-1&n-1&n-1\\
n-1&n-1&n-1&n-1&n-1&n&n&n&n&n\\
n-2&n-3&n-2&0&0&n-2&n-3&n-2&0&0\\
n-3&n-2&0&n-2&0&n-3&n-2&0&n-2&0
\end{array}\right],\]
where columns~$1$,~$2$,~$6$ and~$7$ arise from edges incident to $(n,n,n-3,n-3)$,
columns~$3$ and~$8$ arise from edges incident to $(n,n,n-3,0)$,
columns~$4$ and~$9$ arise from edges incident to $(n,n,0,n-3)$,
and columns~$5$ and~$10$ arise from edges incident to $(n,n,0,0)$.
\end{enumerate}
This is shown in Lemma \ref{lem:aux3} to give a contribution of 
$$6\left(-\frac{n^2t^4}{2} - \frac{n^2t^3}{2} + \frac{7nt^4}{3} + 2nt^3 - \frac{nt^2}{3} - \frac{21t^4}{8} - \frac{23t^3}{12} + \frac{5t^2}{8} - \frac{t}{12}\right).$$

Adding the contributions of the four steps above, and simplifying, then gives
the desired expression in~\eqref{eq:ehrp4n}.
\end{proof}
Note that by taking $4!$ times the coefficient of $t^4$ in Theorem~\ref{thm:p4n}, we recover the formula for $v(m,4)$ in Example~\ref{ex:formulas}, again with a more geometric proof.

\subsection{Generalized permutohedra results}\label{ssec-gpformulae}
We now use the generalized permutohedron point of view developed in Section~\ref{sec:gp} to obtain certain results regarding the Ehrhart polynomial of~$\cP(m,n)$ with $n\ge m-1$.

By applying~\cite[Theorem~11.3]{postnikov} to~\eqref{eq-Minkowski2} (which involves 
regarding $t\,\widetilde{\cP}(m,n)$ as the so-called trimmed version of 
$\simp_{m+1}+t\,\widetilde{\cP}(m,n)$, 
where $\widetilde{\cP}(m,n)$ is defined in~\eqref{liftedP} and $\simp_{m+1}$ is the standard simplex in $\RR^{m+1}$), it follows, after some simplification, that the Ehrhart polynomial of~$\widetilde{\cP}(m,n)$ and~$\cP(m,n)$ with $n\geq m-1$ is given by
\begin{equation}\label{eq-Ehr-drac-seq}
\ehr(\cP(m,n),t)=\sum_{(a_1,\ldots,a_{m(m+1)/2})}\:\prod_{i=1}^m\binom{(n-m+1)t+a_i-1}{a_i}
\:\prod_{i=m+1}^{m(m+1)/2}\binom{t+a_i-1}{a_i},\end{equation}
where the sum is over all draconian sequences for this case, with these defined as follows.
Let $I_1,\ldots,I_{m(m+1)/2}$ be the same as for the draconian sequences in~\eqref{eq-vol-drac-seq}.
Then the draconian sequences in~\eqref{eq-Ehr-drac-seq} are those sequences $(a_1,\ldots,a_{m(m+1)/2})$ of nonnegative integers such that $\sum_{k\in S}a_k\le|\cup_{k\in S}I_k|$, for all $\varnothing\subsetneq S\subseteq[m(m+1)/2]$.  (Hence, the definition of the draconian sequences in~\eqref{eq-Ehr-drac-seq} 
can be obtained from the definition of the draconian sequences in~\eqref{eq-vol-drac-seq} by simply relaxing the condition $\sum_{k=1}^{m(m+1)/2}a_k=m$ to $\sum_{k=1}^{m(m+1)/2}a_k\le m$.)  Note that taking $S$ to be a singleton
in the condition $\sum_{k\in S}a_k\le|\cup_{k\in S}I_k|$ gives 
$a_i\le 1$ for $i=1,\ldots,m$, and $a_i\le2$ for $i=m+1,\ldots,m(m+1)/2$,
as  also occurs in~\eqref{eq-vol-drac-seq}.

It follows from~\eqref{eq-Ehr-drac-seq} that, for any fixed $m$ and $t$,
$\ehr(\cP(m,n),t)$ with $n\ge m-1$ is given by a polynomial in $(n-m+1)t$, 
that all the coefficients of this polynomial are positive integers if~$t$ is a positive integer,
and that this polynomial has degree~$m$ (since $m$ $1$'s followed by $m(m-1)/2$ $0$'s is again a draconian sequence).
It can also be seen that, for any fixed~$m$ and~$n$ with $n\ge m-1$,
$\ehr(\cP(m,n),t)$  is given by a polynomial in~$t$ of degree~$m$,
where this also follows from general Ehrhart theory,
and that all the coefficients of this polynomial are positive, where this also follows from the general property
that the Ehrhart polynomial of any lattice type-$\cY$ generalized permutohedron has positive coefficients (see, for example,~\cite[Corollary 3.1.5]{liu}).

\begin{example}
For $m=2$, the draconian sequences in~\eqref{eq-Ehr-drac-seq} are
$(0,0,2)$, $(0,1,1)$, $(1,0,1)$, $(1,1,0)$, $(0,0,1)$, $(0,1,0)$, $(1,0,0)$ and $(0,0,0)$ (where the first four of these are the draconian sequences of Example~\ref{ex-vol-drac}). This gives 
\begin{equation*}\textstyle\ehr(\cP(2,n),t)=\binom{t+1}{2}+\binom{(n-1)t}{1}\binom{t}{1}+
\binom{(n-1)t}{1}\binom{t}{1}+\binom{(n-1)t}{1}^2+\binom{t}{1}+\binom{(n-1)t}{1}+\binom{(n-1)t}{1}+1,
\end{equation*}
which matches~\eqref{eq:ehrp2n}.  
\end{example}

The draconian sequences in~\eqref{eq-Ehr-drac-seq}
for $m=3$ and $m=4$ can also easily be obtained using a computer (there being 
51 and 455 sequences, respectively), and these can then be used to give alternative
proofs of Theorems~\ref{thm:p3n} and~\ref{thm:p4n}.

Note that for $m-1>n\ge2$, $\widetilde{\cP}(m,n)$ does not appear to be a type-$\cY$ generalized permutohedron, so it seems that there are no 
similar shortcuts to the results in Section~\ref{sec:volppn} for the normalized volume and Ehrhart polynomial of~$\cP(m,n)$ with $m-1>n\ge2$.

\begin{remark}\label{rem-drac-parking}
For $n=m-1$, the summand in~\eqref{eq-Ehr-drac-seq} is zero unless the draconian sequence $(a_1,\ldots,a_{m(m+1)/2})$ has $a_1=\dots=a_m=0$.  Hence, in this case,~\eqref{eq-Ehr-drac-seq} simplifies analogously to the simplification of~\eqref{eq-vol-drac-seq} to~\eqref{eq-vol-drac-seq3}.  Specifically, we obtain
\begin{equation}\label{eq-rem-drac-parking}
\ehr(\cP(m,m-1),t)=\sum_{(b_{12},b_{13}\ldots,b_{m-1,m})}\,\prod_{1\le i<j\le m}\!\binom{t+b_{ij}-1}{b_{ij}},\end{equation}
where the sum is over all
sequences $(b_{ij})_{1\le i<j\le m}$ of nonnegative integers such that $\sum_{(i,j)\in S}b_{ij}\le|\cup_{(i,j)\in S}\{i,j\}|$, for all $\varnothing\subsetneq S\subseteq\{(i,j)\mid 1\le i< j\le m\}$.
It follows that $\ehr(\cP(m,m-1),1)$, i.e., the number of integer points of $\cP(m,m-1)$,
is simply the number of such sequences $(b_{ij})_{1\le i<j\le m}$.
Note that an expression for $\ehr(\cP(m,m-1),1)$, as a sum over certain sequences of subsets of~$[m]$, is obtained in~\cite[Theorem~5.1]{AW} (in the context of the parking function polytope~$P_m$ of Remark~\ref{rem-parkingfunctionpolytope}), and that this expression can
be related to the number of sequences~$(b_{ij})_{1\le i<j\le m}$.  
Note also that, by identifying $\cP(m,m-1)$ with the win vector polytope of the complete graph~$K_m$ (see Remark~\ref{rem-winpolytope}) and using~\cite[Theorem~3.10]{bartels1997polytope}, it follows
that the set of integer points of~$\cP(m,m-1)$ is the set of win vectors of all partial orientations
of~$K_m$.  Finally, note that~\eqref{eq-rem-drac-parking} provides an answer to Question~(b) in~\cite[Section~6]{AW}, and to certain other questions which will be discussed in Remark~\ref{rem-conj-Ehr-expl}.
\end{remark}

\subsection{Further directions}\label{sec-further}
We end with a conjecture which provides a completely explicit formula for $\ehr(\cP(m,n),t)$ with $n\ge m-1$.
As in Section~\ref{sec:closedformv(m,n)}, $[z^i]f(z)$ denotes the coefficient of~$z^i$ in the expansion of 
a power series $f(z)$, and the double factorial is $(2i-3)!!=-\prod_{j=1}^i(2j-3)$, for any nonnegative integer $i$.

\begin{conjecture}\label{conj-Ehr-expl}
We conjecture that, for any~$m$ and~$n$ with $n\ge m-1$, the Ehrhart polynomial of~$\cP(m,n)$ is given by
\begin{multline}\label{eq-Ehr-expl1}
\ehr(\cP(m,n),t)\\
=\frac{1}{2^m}\sum_{i=0}^{\lfloor m/2\rfloor}\sum_{j=2i}^m(-1)^{i+1}\,\binom{m}{m-j,\,j-2i,\,i,\,i}\,i!\:(2j-4i-3)!!\:t^{j-i}\,(2nt+t+2)^{m-j}\end{multline}
and
\begin{equation}\label{eq-Ehr-expl2}
\ehr(\cP(m,n),t)=m!\,[z^m]\,\sqrt{1-tz}\,e^{(nt+t/2+1)z-tz^2/4}.\end{equation}
\end{conjecture}

We note that a conjectural expression for the form of
$\ehr(\cP(m,n),t)$ with $n\ge m-1$ appeared as Conjecture~6.3 in the first arXiv version of this paper.

It can be shown straightforwardly that the right hand sides of~\eqref{eq-Ehr-expl1} and~\eqref{eq-Ehr-expl2} are equal.
Indeed, by considering natural generalizations of~\eqref{eq:vmncoeff}, we initially conjectured~\eqref{eq-Ehr-expl2} and then evaluated this explicitly to obtain~\eqref{eq-Ehr-expl1}.

Conjecture~\ref{conj-Ehr-expl} can be seen to generalize Theorem~\ref{thm:closedformulav(m,n)}, as follows. Since
$v(m,n)/m!$ is the leading coefficient of $\ehr(\cP(m,n),t)$, as a polynomial in~$t$, and since the degree of this polynomial is~$m$, we have $v(m,n)/m!=\bigl(t^m\,\ehr(\cP(m,n),1/t)\bigr)\big|_{t\to0}$.
Defining $g(z,t)=\sqrt{1-tz}\,e^{(nt+t/2+1)z-tz^2/4}$, and assuming
that~\eqref{eq-Ehr-expl2} holds, we then have, for $n\ge m-1$,
\begin{align*}
v(m,n)/(m!)^2&=\bigl(t^m\,[z^m]\,g(z,1/t)\bigr)\big|_{t\to0}=[z^m]\,g(tz,1/t)\big|_{t\to0}=[z^m]\,\sqrt{1-z}\,e^{(n+1/2+t)z-tz^2/4}\big|_{t\to0}\\[1.7mm]
&=[z^m]\,\sqrt{1-z}\,e^{(n+1/2)z},\end{align*}
which reproduces~\eqref{eq:vmncoeff}.

It can also easily be checked that~\eqref{eq-Ehr-expl1} reproduces the expressions~\eqref{eq:ehrp1n}--\eqref{eq:ehrp4n} for $\ehr(\cP(m,n),t)$ with $m\le4$ and $n\ge m-1$.

\begin{remark}
Following the same approach as used in Remark~\ref{rem-vol-rec}, we can obtain a conjectural recurrence
relation for $\ehr(\cP(m,n),t)$ with $n\ge m-1$, which generalizes~\eqref{eq-vol-rec}. Specifically, the function $g(z)=\sqrt{1-tz}\,e^{(nt+t/2+1)z-tz^2/4}$,
which appears in~\eqref{eq-Ehr-expl2} (and is denoted above as $g(z,t)$), satisfies
$(1-tz)g'(z) =(t^2z^2/2-(nt+t/2+3/2)tz+nt+1)\,g(z)$.  Using this, and assuming that~\eqref{eq-Ehr-expl2} holds, then gives 
\begin{multline}\label{eq-Ehr-rec}\ehr(\cP(m,n),t)\\
=(mt+nt-t+1)\ehr(\cP(m-1,n),t)-(m-1)(nt+t/2+3/2)t\ehr(\cP(m-2,n),t)\\
+(m-1)(m-2)t^2\ehr(\cP(m-3,n),t)/2.\end{multline}
By setting~$\ehr(\cP(0,n),t)=1$, and setting $\ehr(\cP(-1,n),t)$ and $\ehr(\cP(-2,n),t)$  arbitrarily (since $\ehr(\cP(-1,n),t)$ and $\ehr(\cP(-2,n),t)$ have 
coefficients zero in the $m=1$ and $m=2$ cases of~\eqref{eq-Ehr-rec}),
it follows that~\eqref{eq-Ehr-rec} with $m\ge1$ is equivalent
to each of the equations in Conjecture~\ref{conj-Ehr-expl}.
\end{remark}

\begin{remark}\label{rem-conj-Ehr-expl}
Confirmation of the $n=m-1$ case of Conjecture~\ref{conj-Ehr-expl} 
would provide explicit answers to 
certain questions related to the parking function polytope, as follows.  As discussed in Remark~\ref{rem-parkingfunctionpolytope}, $\cP(m,m-1)$ is (up to a simple translation) the parking function polytope~$P_m$.  Stanley asked for an enumeration of the integer points in~$P_m$~\cite[Part~(c)]{AMM_problem}, and as discussed in Remark~\ref{rem-drac-parking}, 
the number of these points is the number of sequences $(b_{ij})_{1\le i<j\le m}$ in~\eqref{eq-rem-drac-parking}, 
which is also given by~\cite[Theorem~5.1]{AW}.  Taking $n=m-1$ and
$t=1$ in~\eqref{eq-Ehr-expl1} now provides an explicit conjectural expression for this number, i.e., $\frac{1}{2^m}\sum_{i=0}^{\lfloor m/2\rfloor}\sum_{j=2i}^m(-1)^{i+1}\,\binom{m}{m-j,\,j-2i,\,i,\,i}\,i!\:(2j-4i-3)!!\,(2m+1)^{m-j}$.
As shown in~\cite[Corollary~4.2]{selig}, the number of integer points in~$P_m$ is also the number of stochastically recurrent states in the stochastic sandpile model on the complete graph~$K_m$. The $n=m-1$ and $t=1$ case of~\eqref{eq-Ehr-expl1} thereby provides a conjectural answer to a question in~\cite[Section~6]{selig} asking for an explicit expression for this number.
Finally, the $n=m-1$ case of~\eqref{eq-Ehr-expl1} with~$t$ arbitrary provides an explicit conjectural answer to a question in~\cite[Problem~5.5]{HanadaLentferVindasMelendez} asking for the Ehrhart polynomial of $P_m$,
while~\eqref{eq-rem-drac-parking} provides a less explicit, but non-conjectural, answer to this question.
\end{remark}

\bibliography{biblio}
\bibliographystyle{amsplain}
	
\appendix \section{Auxiliary lemmas}
Here, we present some intermediate technical results which are used in the
proofs of Theorems~\ref{theorem:Pm4} and~\ref{thm:p4n}. 

\begin{lemma}\label{lem:aux1}
Consider the polytope
\begin{align*}
\cQ(m)&=\Conv\bigl((4\mathbf{e}_1+4\mathbf{e}_2+2\Delta_{m-2})\cup
(4\mathbf{e}_1+3\mathbf{e}_2+3\Delta_{m-2})\cup(3\mathbf{e}_1+4\mathbf{e}_2+3\Delta_{m-2})\bigr),\\
\intertext{where}
\Delta_{m-2}&=
\Conv\left(\{\mathbf{e}_0,\mathbf{e}_3,\mathbf{e}_4,\dots,\mathbf{e}_m\}\right).
\end{align*}
Then we have
\[\nvol(\cQ(m)) =2^m-3^m+m\,3^{m-1}.\]
\end{lemma}
	
\begin{proof}
Let $f(m)=\nvol(\cQ(m))$.
We will establish a recurrence relation for~$f(m)$.
Fix the vertex  $\vv=4\mathbf{e}_1+4\mathbf{e}_2+2\mathbf{e}_m$ of $\cQ(m)$.
By Lemma \ref{lem:coning}, $f(m)$ can be computed by taking pyramids whose apex is $\vv$ and 
whose bases are the facets of $\cQ(m)$ that do not contain $\vv$.
There are two such facets, as follows:
\begin{enumerate}
\item $F_1=\{\mathbf{x}\in\cQ(m)\mid x_m=0\}$.
This facet is congruent to $\cQ(m-1)$, and hence it has relative volume $f(m-1)/(m-1)!$.
Since the lattice distance from $\mathbf{v}$ to the hyperplane containing $F_1$ is~$2$, then the pyramid with base $F_1$ and apex $\mathbf{v}$ has normalized volume $f(m-1)/(m-1)!\cdot 2\cdot(1/m)\cdot m!=2f(m-1)$.
\item $F_2=\{\mathbf{x}\in\cQ(m)\mid x_1+x_2=7\}$. This facet is congruent to $\Pi(4,3)\times 3\Delta_{m-2}$, and hence it has relative volume $1\cdot3^{m-2}/(m-2)!$.
Since the lattice distance from $\mathbf{v}$ to the hyperplane containing $F_2$ is $1$, then the pyramid with base $F_2$ and apex $\mathbf{v}$ has normalized volume $3^{m-2}/(m-2)!\cdot 1\cdot(1/m)\cdot m!= 3^{m-2}(m-1)$.
\end{enumerate}
This gives the recurrence relation $f(m)=2f(m-1)+3^{m-2}(m-1)$, with the initial condition $f(3)=8$.
The solution to this is $f(m)=2^m-3^m+m\,3^{m-1}$, as required.
\end{proof}
		
\begin{lemma}
\label{lem:aux2}
Consider the polytope
\begin{align*}
\cQ(m)&= \Conv\bigl(\{4\mathbf{e}_1+3\mathbf{e}_2+3\mathbf{e}_3, 3\mathbf{e}_1+4\mathbf{e}_2+3\mathbf{e}_3, 3\mathbf{e}_1+3\mathbf{e}_2+4\mathbf{e}_3\}\cup (\Pi(4,3,2)\times\Delta_{m-3})\bigr),\\
\intertext{where}
\Delta_{m-3}&=\Conv\left(\{\mathbf{e}_0,\mathbf{e}_4,\mathbf{e}_5,\dots,\mathbf{e}_m\}\right).
\end{align*}
Then we have
\[\nvol(\cQ(m))=3m^2-6m+1.\]
\end{lemma}
		
\begin{proof}
We fix the vertex $\vv=4\mathbf{e}_1+3\mathbf{e}_2+3\mathbf{e}_3$ of $\cQ(m)$, and subdivide $\cQ(m)$ into pyramids whose bases are the facets not containing $\vv$. There are four such facets, as follows, where in each case, the lattice distance from~$\vv$ to the hyperplane containing the facet is~$1$:
\begin{enumerate}
\item $\{\mathbf{x}\in\cQ(m)\mid x_2=4\}$. This facet is an $(m-1)$-polytope congruent to the pyramid with apex $3\mathbf{e}_1+3\mathbf{e}_2$ and base $\Pi(3,2)\times \Delta_{m-3}$. Hence, this has relative volume $1/(m-3)!1/(m-1)$, and the associated pyramid has normalized volume $1/(m-3)!1/(m-1)\cdot1\cdot(1/m)\cdot m!=m-2$.
\item $\{\mathbf{x}\in\cQ(m)\mid x_3=4\}$. Analogously to the previous case, 
the pyramid has volume $m-2$.
\item $\{\mathbf{x}\in\cQ(m)\mid x_2+x_3=7\}$. This facet is the polytope $\Pi(4,3)\times(2\mathbf{e}_1+\Delta_{m-2})$ (with 
$\Delta_{m-2}=\textrm{ConvexHull}(\{\mathbf{e}_0,\mathbf{e}_1,\mathbf{e}_4,\mathbf{e}_5,\dots,\mathbf{e}_m\})$), so it has relative volume $1/(m-2)!$, and the associated pyramid has normalized volume $1/(m-2)!\cdot1\cdot(1/m)\cdot m!=m-1$. 
\item $\{\mathbf{x}\in\cQ(m)\mid x_1+x_2+x_3=9\}$. This facet is $\Pi(4,3,2)\times \Delta_{m-3}$, so it has relative volume $3/(m-3)!$ and the associated pyramid has normalized volume $3/(m-3)!\cdot1\cdot(1/m)\cdot m! = 3(m-1)(m-2)$.
\end{enumerate}
Combining the contributions from the four facets, gives $\nvol(\cQ(m))=
2(m-2)+(m-1)+3(m-1)(m-2)=3m^2-6m+1$, as required.
\end{proof}
		
\begin{lemma}\label{lem:aux3}
Let $n\geq 4$, and consider the polytope $\cQ$ given by
\begin{equation*}
\Conv\tiny{
\left[\begin{array}{cccccccccccccc}
n&n&n&n&n&n-1&n-1&n-1&n-1&n-1&n&n&n&n\\
n-1&n-1&n-1&n-1&n-1&n&n&n&n&n&n&n&n&n\\
n-2&n-3&n-2&0&0&n-2&n-3&n-2&0&0&n-3&n-3&0&0\\
n-3&n-2&0&n-2&0&n-3&n-2&0&n-2&0&n-3&0&n-3&0
\end{array}\right].}\end{equation*}
Then the Ehrhart polynomial of $\cQ$, minus the facet given by the
convex hull of the first ten columns in the matrix above, is
\begin{equation}
\label{eq:ehrlast}
\frac{n^2t^4}{2} + \frac{n^2t^3}{2} - \frac{7nt^4}{3} - 2nt^3 + \frac{nt^2}{3} + \frac{21t^4}{8} + \frac{23t^3}{12} - \frac{5t^2}{8} + \frac{t}{12}.
\end{equation}
\end{lemma}

\begin{proof}
Note that $\cQ$
is almost the Cartesian product of the triangle with vertices $(n,n)$, $(n,n-1)$ and $(n-1,n)$,
and the pentagon with vertices $(n-3,n-2)$, $(n-2,n-3)$, $(n-2,0)$, $(0,n-2)$ and $(0,0)$.
It is not so, since for the vertex $(n,n)$ of the triangle, instead of the pentagon we have a shrunk version which is the square with vertices $(n-3,n-3)$, $(n-3,0)$, $(0,n-3)$ and $(0,0)$.

To compute the Ehrhart polynomial, we divide and conquer.
We start by subdividing the pentagon as shown in Figure \ref{fig:pentagon}.

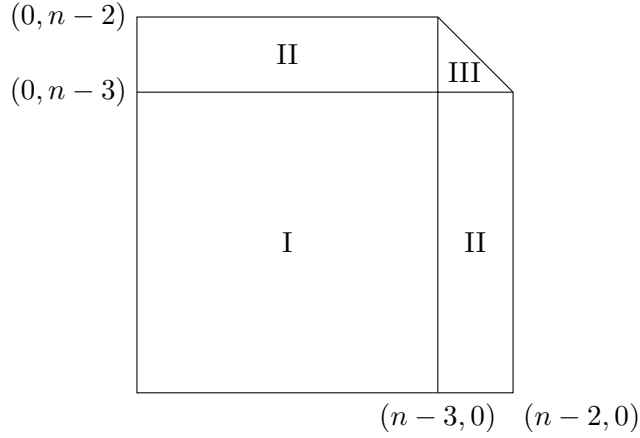
\begin{figure}[ht]
\centering
\begin{tikzpicture}

\node[below] at (4,0) {$(n-3,0)$};
\node[below right] at (5,0) {$(n-2,0)$};
\node[left] at (0,4) {$(0,n-3)$};
\node[left] at (0,5) {$(0,n-2)$};

\node at (2,2) {I};
\node at (4.5,2) {II};
\node at (2,4.5) {II};
\node[above right] at (4,4) {III};

\draw (0,0)--(5,0)--(5,4)--(4,5)--(0,5)--cycle;
\draw (4,0)--(4,5) (0,4)--(5,4);

\end{tikzpicture}
\caption{The spirit of the argument in the proof of Lemma~\ref{lem:aux3}.}
\label{fig:pentagon}
\end{figure}

Based on this, we subdivide $\cQ$ into the following pieces.
\begin{enumerate}
\item The first piece, labeled I in Figure \ref{fig:pentagon}, is the polytope 
\[
\textrm{ConvexHull}\begin{bmatrix}n&n-1&n\\n-1&n&n \end{bmatrix}\times [0,n-3]^2.
\]
\item The second piece, labeled II on the right in Figure \ref{fig:pentagon}, is the polytope given as the prism of length $n-3$ over the pyramid
\[
\textrm{ConvexHull}
\left[\begin{array}{ccccc}
n&n-1&n&n-1&n\\
n-1&n&n-1&n&n\\
n-3&n-3&n-2&n-2&n-3
\end{array}\right],\]
which is a lattice pyramid with apex given by the fifth column and whose base is the square formed by the first four columns.
The other piece labeled II is congruent to this one.
\item The third piece, labeled III in Figure \ref{fig:pentagon}, is a lattice pyramid with apex $(n,n,n-3,n-3)$ and whose base is the prism
\[\textrm{ConvexHull}\begin{bmatrix}n&n-1\\n-1&n\end{bmatrix}\times \textrm{ConvexHull}\begin{bmatrix}n-3&n-3&n-2\\n-3&n-2&n-3\end{bmatrix}.\]
\end{enumerate}

All the pieces are obtained by combinations of prisms, Cartesian products and lattice pyramids, so we can compute each of their Ehrhart polynomials and their intersections. We leave 
the details to the reader.
\end{proof}
\end{document}